\documentclass[oneside,english]{amsart}
\usepackage[T1]{fontenc}
\usepackage[latin9]{inputenc}
\usepackage{color}
\usepackage{babel}
\usepackage{units}
\usepackage{amsthm}
\usepackage{amstext}
\usepackage{amssymb}
\usepackage{esint}
\usepackage[unicode=true,pdfusetitle,
 bookmarks=true,bookmarksnumbered=false,bookmarksopen=false,
 breaklinks=false,pdfborder={0 0 0},backref=false,colorlinks=true]
 {hyperref}

\makeatletter
\theoremstyle{plain}
\newtheorem{thm}{\protect\theoremname}[section]
  \theoremstyle{plain}
  \newtheorem{lem}[thm]{\protect\lemmaname}
  \theoremstyle{definition}
  \newtheorem{defn}[thm]{\protect\definitionname}
  \theoremstyle{remark}
  \newtheorem*{rem*}{\protect\remarkname}
  \theoremstyle{plain}
  \newtheorem{cor}[thm]{\protect\corollaryname}
  \theoremstyle{plain}
  \newtheorem{prop}[thm]{\protect\propositionname}
  \theoremstyle{definition}
  \newtheorem{example}[thm]{\protect\examplename}
  \theoremstyle{remark}
  \newtheorem*{claim*}{\protect\claimname}

\newcommand{\Rinf}{\underline{\mathbb{R}}}

\makeatother

  \providecommand{\claimname}{Claim}
  \providecommand{\corollaryname}{Corollary}
  \providecommand{\definitionname}{Definition}
  \providecommand{\examplename}{Example}
  \providecommand{\lemmaname}{Lemma}
  \providecommand{\propositionname}{Proposition}
  \providecommand{\remarkname}{Remark}
\providecommand{\theoremname}{Theorem}

\begin{document}

\title{On Interpolation and Curvature via Wasserstein Geodesics}

\author{Martin Kell}
\begin{abstract}
In this article, a proof of the interpolation inequality along geodesics
in $p$-Wasserstein spaces is given. This interpolation inequality
was the main ingredient to prove the Borel-Brascamp-Lieb inequality
for general Riemannian and Finsler manifolds and led Lott-Villani
and Sturm to define an abstract Ricci curvature condition. Following
their ideas, a similar condition can be defined and for positively
curved spaces one can prove a Poincar\'e inequality. Using Gigli's
recently developed calculus on metric measure spaces, even a $q$-Laplacian
comparison theorem holds on $q$-infinitesimal convex spaces.

In the appendix, the theory of Orlicz-Wasserstein spaces is developed
and necessary adjustments to prove the interpolation inequality along
geodesics in those spaces are given.
\end{abstract}

\email{mkell@mis.mpg.de}

\address{Max-Planck-Institute for Mathematics in the Sciences, Inselstr. 22,
04103 Leipzig, Germany}

\thanks{The author wants to thank Prof. Jürgen Jost and the MPI MiS for providing
a stimulating research environment. The research was funded by the
IMPRS ``Mathematics in the Sciences''.}

\maketitle
The proof of the Borell\textendash{}Brascamp\textendash{}Lieb (BBL)
inequality for Riemannian manifolds by Cordero-Erausquin-McCann-Schmuckenschläger
\cite{CMS2001}, and later for Finsler manifolds by Ohta \cite{Ohta2009},
led Lott-Villani \cite{LV2009,LV2007} and Sturm \cite{Sturm2006a,Sturm2006}
to a new notion of a lower bound on the generalized Ricci curvature
for metric measure spaces, called curvature dimension. Both, the BBL
inequality and the curvature condition, rely on geodesics in the $2$-Wasserstein
space, which was a natural candidate because of its connection to
convex analysis in the Euclidean setting. 

Based on Ohta's proof \cite{Ohta2009} we show how to prove the BBL
inequality via geodesics in the p-Wasserstein spaces for any $p>1$.
Following Lott-Villani-Sturm, a new curvature condition can be defined
via convexities along geodesics in the $p$-Wasserstein space, and
many known results, like Poincaré inequality and Bishop-Gromov volume
comparison, follow by similar arguments. 

The proof of the BBL inequality relied on three ingredients: (1) a
solution to the Monge problem and a prescription of the interpolation
maps, (2) second order differentiability of the solution potential
and a cut locus description, and (3) positive (semi-) definiteness
of the Jacobian of the interpolation map. The solution to the Monge
problem easily follows by combining \cite{McCann2001} and \cite{Ohta2009}.
The interpolation maps already give the idea that optimal transport
is along geodesics, which is well-known by Lisini's result \cite{Lisini2006}.
For the proof of second order differentiability, we rely on two observations:
(1) Ohta \cite{Ohta2009} noticed that the lack of $C^{2}$-smoothness
of square distance $d^{2}(\cdot,\cdot)$ at the diagonal can be avoided
by splitting the transport plan into a moving and a non-moving part,
this actually works for all smooth functions of the distance, (2)
the set of $c_{p}$-concave functions is star-shaped. (1) says we
only need to check where the transport maps maps to different points
and (2) helps to move the terminal point away from the cut-locus.
Using this, a proof of the (almost) semiconcavity of solution potentials
(Theorem \ref{thm:semiconc}) is given, which is shorter than Ohta's
orginial proof \cite{Ohta2008}, yet it doesn't show that $c_{p}$-concave
functions are everywhere locally semiconcave. However, it easily adapts
to the Orlicz case, see Theorem \ref{thm:Orlicz-semiconc}. 

The star-shapedness of $c_{p}$-concave functions, resp. pseudo star-shapedness
of $c_{L}$-concave functions, and positive (semi-) definiteness of
the Jacobian rely on the following, quite innocent looking inequalities:
if $z\in Z_{t}(x,y)$ then for any $m$ 
\[
t^{p-1}d^{p}(m,y)\le d^{p}(m,z)+t^{p-1}(1-t)d^{p}(x,y)
\]
and 
\[
t^{-1}L(d(m,y))\le L(d(m,z)/t)+t^{-1}(1-t)L(d(x,y))
\]
where $L$ is a strictly increasing convex function.

As a \textquotedbl{}vertical dual\textquotedbl{} one can use the recent
theory and calculus developed around the $q$-Cheeger energy ($q$
is the Hölder conjugate of $p$) by Ambrosio-Gigli-Savaré \cite{Ambrosio2013,Ambrosio2011,Gigli2012}
to even get a q-Laplacian comparison, which, however, is equivalent
to the usual one in the smooth setting. In a second paper \cite{Kell2013}
we will study the gradient flow of the $q$-Cheeger energy, called
$q$-heat flow, and use the \textquotedbl{}duality\textquotedbl{}
and curvature condition to identify it with the gradient flow of the
$(3-p)$-Renyi entropy if $p\in(1,3)$.

In the end, we show how to prove the interpolation via geodesics in
Orlicz-Wasserstein spaces. Since Orlicz-Wasserstein spaces are notationally
more involved and we need some additional result, like the geodesic
character of Orlicz-Wasserstein spaces, we give those results in the
appendix. However, the metric of Orlicz-Wasserstein spaces is not
defined via a single optimization problem. Thus there is no natural
dual problem and by now no \textquotedbl{}vertical dual\textquotedbl{}
to the theory of Orlicz-Wasserstein spaces, in particular, there is
no Orlicz-Cheeger energy and no Orlicz-Laplacian.

Now we will give an outline of the result and the structure of the
paper: In the first section, we will given an overview of the used
concepts. The second section will develop the theory of $c_{p}$-concave
functions, their Lipschitz regularity and star-shapedness. The next
section will deal with the smooth setting, i.e. the Brenier-McCann-Ohta
solution, second order differentiability and the interpolation inequality,
which can be stated as 
\[
\mathbf{J}_{t}(x)^{\nicefrac{1}{n}}\ge(1-t)\mathfrak{v}_{t}^{>}(x,y_{1})^{\nicefrac{1}{n}}+t\mathfrak{v}_{t}^{<}(x,y_{1})^{\nicefrac{1}{n}}\mathbf{J}_{1}(x)^{\nicefrac{1}{n}}
\]
where $\mathbf{J}_{t}(x)$ is the Jacobian of the interpolation map
and $\mathfrak{v_{t}^{>}}$ and $\mathfrak{v}_{t}^{<}$ are the (backward,
resp. forward) volume distortion coefficients. The convexity of functionals
in $\mathcal{DC}_{N}$ immediately follows along the by-now-standard
lines and hence a version of Lott-Villani's curvature condition.

In the fourth section, we will define the curvature dimension condition
$CD_{p}(K,N)$ along the lines of Lott-Villani-Sturm. Since most proofs,
which don't use Cauchy-Schwary, can be easily adapted, we will only
show a Poincar\'e inequality for positively curved $CD_{p}(K,\infty)$-spaces,
i.e. $K>0$. For those spaces, we get
\[
\left(\int(h-\bar{h})^{2}d\mu\right)^{\frac{1}{2}}\le\frac{1}{\sqrt{2K}}\left(\int|D^{-}h|^{q}d\mu\right)^{\frac{1}{q}}
\]
for any Lipschitz function $h$. In the end of this section we show
a version of the metric Brenier theorem and use Gigli's recently developed
calculus \cite{Gigli2012} to give a $q$-Laplacian comparison theorem,
namely
\[
\Delta_{q}\phi\le N\tilde{\sigma}_{K,N}(|\nabla\phi|_{w}^{q-1})d\mu
\]
 for any $c_{p}$-concave function in the domain of the $q$-Laplacian.

In the appendix, we develop the theory of Orlicz-Wasserstein spaces
and show how to adapt the proofs of the interpolations inequality.

\section{Preliminaries}

In this part, we will introduce the main concepts used in this work.
For n general introduction to the theory of optimal transport and
curvature via $2$-Wasserstein spaces see \cite{Villani2009}, especially
its Chapter 6 on Wasserstein spaces. We follow Ohta's notation \cite{Ohta2008,Ohta2009}
for Finsler manifolds and otherwise refer to \cite{BCS2000,Shen2001}.

As a convention we will always assume that $(M,d,\mu)$ is a locally
compact metric space equipped with a locally finite Borel measure
$\mu$ and if not otherwise stated it is assumed to be geodesic (see
below). Since we will also deal with spaces which are not locally
compact (e.g. $(\mathcal{P}_{p}(M),w_{p})$ with $M$ non-compact),
the sections below do not assume that $(X,d)$ is locally compact.
And as an abbreviation define 
\[
\underline{\mathbb{R}}:=\mathbb{R}\cup\{-\infty\}.
\]

\subsection*{Metric spaces}

Let $(X,d)$ be a (complete) metric space and for simplicity we assume
that $X$ has no isolated points.

\subsubsection*{Absolutely continuous curves and geodesics}

If $I\subset\mathbb{R}$ is an open interval then we say that a curve
$\gamma:I\to X$ is in $AC^{p}(I,X)$ (we drop the metric $d$ for
simplicity) for some $p\in[1,\infty]$ if 
\[
d(\gamma_{s},\gamma_{t})\le\int_{s}^{t}g(r)dr\quad\forall s,t\in J:s<t
\]
for some $g\in L^{p}(J)$. In case $p=1$ we just say that $\gamma$
is absolutely continuous. It can be shown \cite[Theorem 1.1.2]{AmbGigSav2008}
that in this case the metric derivative 
\[
|\dot{\gamma}_{t}|:=\limsup_{s\to t}\frac{d(\gamma_{s},\gamma_{t})}{|s-t|},
\]
with $\lim$ for a.e. $t\in I$, is a minimal representative of such
a $g$. We will say $\gamma$ has constant (unit) speed if $|\dot{\gamma}_{t}|$
is constant (resp. $1$) almost everywhere in $I$.

It is not difficult to see that $AC^{p}(I,X)\subset C(\bar{I},X)$
where $C(\bar{I},X)$ is equipped with the $\sup$ distance $d^{*}$
\[
d^{*}(\gamma,\gamma'):=\sup_{t\in\bar{I}}d(\gamma_{t},\gamma_{t}').
\]
For each $t\in\bar{I}$ we can define the evaluation map $e_{t}:C(\bar{I},X)\to X$
by 
\[
e_{t}(\gamma)=\gamma_{t}.
\]

We will say that $(X,d)$ is a geodesic space if for each $x_{0},x_{1}\in X$
where is a constant speed curve $\gamma:[0,1]\to X$ with $\gamma_{i}=x_{i}$
and 
\[
d(\gamma_{s},\gamma_{t})=|t-s|d(\gamma_{0},\gamma_{1}).
\]
In this case, $\gamma$ is called constant speed geodesic. The space
of all constant speed geodesics $\gamma:[0,1]\to X$ will be donated
by $\operatorname{Geo}(X)$. Using the triangle inequality it is not
difficult to show the following:
\begin{lem}
Assume $\gamma:[0,1]\to X$ is a curve such that 
\[
d(\gamma_{s},\gamma_{t})\le|t-s|d(\gamma_{0},\gamma_{1})
\]
then $\gamma$ is a geodesic from $\gamma_{0}$ to $\gamma_{1}$.
\end{lem}

A weaker concept is the concept of a length space: In such spaces
the distance between point $x_{0}$ and $x_{1}\in X$ is given by
\[
d(x_{0},x_{1})=\inf\int_{0}^{1}|\dot{\gamma}_{t}|dt
\]
where the infimum is taken over all absolutely continuous curves connecting
$x_{0}$ and $x_{1}$. In case $X$ is complete and locally compact,
the two concepts agree. Furthermore, Arzela-Ascoli also implies:
\begin{lem}
If $(X,d)$ is locally compact then so is $(\operatorname{Geo}(X),d^{*})$
where $d^{*}$ is the $\sup$-distance on $C(\bar{I},X)$.
\end{lem}

\subsubsection*{Lipschitz constants and upper gradients}

Given a function $f:X\to\overline{\mathbb{R}}=[-\infty,\infty]$,
the local Lipschitz constant $|Df|:X\to[0,\infty]$ is given by 
\[
|Df|(x):=\limsup_{y\to x}\frac{|f(y)-f(x)|}{d(y,x)}
\]
for $x\in D(f)=\{x\in X\,|\, f(x)\in\mathbb{R}\}$, otherwise $|Df|(x)=\infty$.
The one-sided versions $|D^{+}f|$ and $|D^{-}f|$, also called ascending
slope (resp. descending slope)
\begin{eqnarray*}
|D^{+}f|(x) & := & \limsup_{y\to x}\frac{[f(y)-f(x)]_{+}}{d(y,x)}\\
|D^{-}f|(x) & := & \limsup_{y\to x}\frac{[f(y)-f(x)]_{-}}{d(y,x)}
\end{eqnarray*}
for $x\in D(f)$ and $\infty$ otherwise, where $[r]_{+}=\max\{0,r\}$
and $[r]_{-}=\max\{0,-r\}$. It is not difficult to see that $|Df|$
is (locally) bounded iff $f$ is (locally) Lipschitz. 

We say that $g:X\to[0,\infty]$ is an upper gradient of $f:X\to\overline{\mathbb{R}}$
if for any absolutely continuous curve $\gamma:[0,1]\to D(f)$ the
curve $t\mapsto g(\gamma_{s})|\dot{\gamma}_{s}|$ is measurable in
$[0,1]$ (with convention $0\cdot\infty=0$) and 
\[
|f(\gamma_{1})-f(\gamma_{0})|\le\int_{0}^{1}g(\gamma_{t})dt.
\]
It is not difficult to see that $ $the local Lipschitz constant and
the two slopes are upper gradients in case $f$ is (locally) Lipschitz.

\subsubsection*{Optimal Transport}

Let $(M,d)$ be a proper metric space. Given two probability measure
$\mu_{0},\mu_{1}\in\mathcal{P}(M)$ and a (non-negative) cost function
$c:M\times M\to[0,\infty)$ one can define the following Kantorovich
problem
\[
C(\mu_{0},\mu_{1})=\inf_{\pi\in\Pi(\mu_{0},\mu_{1})}\int c(x,y)d\pi(x,y)
\]
where $\Pi(\mu_{0},\mu_{1})$ is the set of all $\pi\in\mathcal{P}(M\times M)$
such that $(p_{1})_{*}\pi=\mu_{0}$ and $(p_{2})_{*}\pi=\mu_{1}$
with $p_{i}$ being the projections to the $i$-th coordinate.

It is well-known that problem has a solution $\pi_{opt}$, i.e. a
probability measure $\pi_{opt}$ in $\Pi(\mu_{0},\mu_{1})$ such that
\[
C(\mu_{0},\mu_{1})=\int c(x,y)d\pi_{opt}(x,y).
\]

Given any such cost function one can define a dual problem 
\[
\tilde{C}(\mu_{0},\mu_{1})=\sup_{\phi(x)+\psi(y)\le c(x,y)}\int\phi d\mu_{0}+\int\psi d\mu_{1}.
\]
It is not difficult to see that $\tilde{C}\le C$. 

The solution to this problem can be described by a pair of $c$-concave
potentials: if $\psi:M\to\Rinf$ then one can define the $c$-transform
as 
\[
\psi^{c}(y)=\inf_{x\in M}c(x,y)-\phi(x).
\]
We say that $\phi$ is $c$-concave, if it is the $c$-transform of
some function $\psi$. Similarly, with $c$ replaced by $\bar{c}(x,y)=c(y,x)$
one define $\bar{c}$-transform of $\phi$ and says $\psi$ is $\bar{c}$-concave
if it is the $\bar{c}$-transform of a function $\phi$.

Given a $c$-concave function $\phi=\psi^{c}$ one can define the
$c$-subdifferential $\partial^{c}\phi$ by
\[
\partial^{c}\phi(x)=\{y\in M\,|\,\phi(x)+\psi(y)=c(x,y)\}.
\]

One of the major results in optimal transport theory is the following:
\begin{thm}
\cite[Theorem 5.11]{Villani2009}One always has 
\[
\tilde{C}(\mu_{0},\mu_{1})=C(\mu_{0},\mu_{1})
\]
and the dual problem is attained by a pair $(\phi,\psi)$ of $c$-concave/$\bar{c}$-concave
functions with $\phi=\psi^{c}$ and $\psi=\phi^{\bar{c}}$. Assuming,
for simplicity, that $c$ is continuous, then the optimal transport
measure $\pi_{opt}$ is supported on the graph of the $c$-subdifferential
which is $c$-cyclically monotone, i.e. given $n$ couples $(x_{i},y_{i})\in\partial^{c}\phi$
one has 
\[
\sum_{i=0}^{n-1}c(x_{i},y_{i})\le\sum_{i=1}^{n-1}c(x_{i},y_{i+1}).
\]
Furthermore, if $\partial^{c_{p}}\phi(\cdot)$ is single-valued $\mu_{0}$-almost
everywhere, then $\pi_{opt}$ is concentrated on the graph of a measurable
function $T$ where $T$ is a measurable selection of $x\mapsto\partial^{c}\phi(x)$
which is uniquely defined $\mu_{0}$-a.e..
\end{thm}

\subsubsection*{$p$-Wasserstein spaces}

The $p$-Wasserstein space for $1<p<\infty$ is the space of all probability
measures with finite $p$-Moments 
\[
\mathcal{P}_{p}(M)=\{\mu\in\mathcal{P}(M)\,|\,\int d^{p}(x,x_{0})d\mu(x)<\infty\}
\]
equipped with the metric
\[
w_{p}(\mu_{0},\mu_{1})=\left(C_{p}(\mu_{0},\mu_{1})\right)^{\frac{1}{p}}
\]
where the cost function is given by $c_{p}(x,y)=d^{p}(x,y)/p$. 

It is well know that $(\mathcal{P}_{p}(M),w_{p})$ is a complete metric
measure space if $(M,d)$ is, and it is compact iff $M$ is (see \cite[Chapter 6]{Villani2009}).
However, it is not locally compact if $M$ is just locally compact.
Nevertheless, in case $M$ is a proper metric space there is a sufficiently
nice weak topology induced by the subspace topology of $\mathcal{P}(M)$
with its weak topology. 
\begin{lem}
[see e.g. {\cite[Theorem 6]{Kell2011}}] Let $(M,d)$ be a proper
metric space, then every bounded set in $\mathcal{P}_{p}(M)$ is precompact
w.r.t. to the weak topology induced by $\mathcal{P}_{p}(M)\subset\mathcal{P}(M)$.
\end{lem}
Furthermore, if $M$ is a geodesic space than so is $\mathcal{P}_{p}(M)$
(see \cite{Lisini2006}). 

In the appendix, we introduce more general Wasserstein spaces, called
Orlicz-Wasserstein space. For those the distance is not given by a
single optimization problem and so far there is no nicely defined
dual problem.

\subsection*{Finsler manifolds}

In this section, we recall some notation and facts from Finsler geometry.
We will mainly follow the notation of \cite{Ohta2009,Ohta2008} and
otherwise refer to \cite{BCS2000,Shen2001}.

\subsubsection*{Finsler structures}

Let $M$ be a connected, $n$-dimensional $C^{\infty}$-manifold. 
\begin{defn}
[Finsler structure] A $C^{\infty}$-Finsler structure on $M$ is
a function $F:TM\to[0,\infty)$ such that the following holds
\begin{enumerate}
\item (Regularity) $F$ is $C^{\infty}$ on $TM\backslash\{\mathbf{0}\}$
where $\mathbf{0}$ stands for the zero section,
\item (Positive homogeneity) for any $v\in TM$ and any $\lambda>0$, it
holds $F(\lambda v)=\lambda F(v)$,
\item (Strong convexity) In local coordinates $(x^{i})_{i=1}^{n}$ on $U\subset M$
the matrix
\[
\left(g_{ij}(v)\right):=\left(\frac{1}{2}\frac{\partial^{2}(F^{2})}{\partial v^{i}\partial v^{j}}(v)\right)
\]
is positive-definite at every $v\in\pi^{-1}(U)\backslash0$ where
$\pi:TM\to M$ is the natural projection of the tangent bundle.
\end{enumerate}
\end{defn}
Strictly speaking, this is nothing more than defining a Minkowski
norm $F|_{T_{x}M}$ on each $T_{x}M$ with some regularity requirements
depending on $x$. We don't require $F$ to be absolutely homogeneous,
i.e. $F(v)\ne F(-v)$ is possible. In such a case the ``induced''
distance (see below) is not symmetric. As an abbreviation we let $\bar{F}$
denote the reverse Finsler structure, i.e. $\bar{F}(v)=F(v)$.

On any $C^{\infty}$-manifold one can define the differential $df$
of a $C^{1}$-function $f$. In order to define the gradient of $f$
one needs the following: let $\mathcal{L}:T^{*}M\to TM$ be the Legendre
transform associating to each co-vector $\alpha\in T_{x}^{*}M$ the
unique vector $v=\mathcal{L}_{x}(\alpha)\in T_{x}M$ such that $F(v)=F^{*}(v)$
and $\alpha(v)=F(v)^{2}$, where $F^{*}$ is the dual norm of $F$
on $T^{*}M$. This transform is $C^{\infty}$ from $T^{*}M\backslash\{0\}$
to $TM\backslash\{0\}$ and is $C^{\infty}$ in case $F$ is a Riemannian
structure, i.e. the parallelogram inequality holds on each $T_{x}M$.
The gradient $\nabla f$ at $x$ of $f$ is now defined by $\nabla f(x)=\mathcal{L}_{x}(df_{x})\in T_{x}M$.
Then we have for every unit speed $C^{1}$-curve $\eta:[0,l]\to M$
(i.e. $F(d\eta/dt)\equiv1$) 
\[
-\int_{0}^{l}F(\nabla(-f)(\eta_{t}))dt\le f(\eta(l))-f(\eta(0))\le\int_{0}^{l}F(\nabla f(\eta_{t}))dt.
\]
Thus one can define an intrinsic metric of the Finsler manifold by
\[
d(x,y)=\sup_{f\in C^{1},F(\nabla f)\le1}f(y)-f(x)
\]
which is symmetric iff $F=\bar{F}$. 

Similar to the gradient, there is no notion of (Finsler) Hessian of
a $C^{2}$-function $f$, so that we will use the well-defined differential
of $df:M\to T^{*}M$ which can be written in local coordinates as
\[
d(df)_{x}=\sum_{i,j=1}^{n}\left(\delta_{j}^{i}\frac{\partial}{\partial x^{i}}\bigg|_{df_{x}}+\frac{\partial^{2}f}{\partial x^{i}\partial x^{j}}(x)\frac{\partial}{\partial v^{i}}\bigg|_{df_{x}}\right)dx^{j}\big|_{x}.
\]
Note, however, that this expression is not coordinate free.

\subsubsection*{Chern connection, covariant derivatives and curvature}

In contrast to Riemannian manifolds there is no ``unique'' canonical
connection defined on a Finsler manifold. As in \cite{Ohta2009} we
will only use the Chern connection in this article which is the same
as the Levi-Civita connection in the Riemannian case. In order to
reduce the notation we will only use the Chern connection and denote
it by $\nabla$ without stating its exact property (\cite[Definition 2.2]{Ohta2009}).
For a thorough introduction see \cite{Ohta2009,BCS2000,Shen2001}. 

Recall that by strong convexity of $F$ the matrix $(g_{ij}(v))$
is positive definite for every $v\in T_{x}M\backslash\{0\}$ and hence
defines a scalar product on $ $$T_{x}M$ which will be denoted by
$g_{v}(\cdot,\cdot)$, i.e.
\[
g_{v}(\sum_{i=1}^{n}w_{1}^{i}\frac{\partial}{\partial x^{i}}\bigg|_{x},\sum_{j=1}^{n}w_{2}^{j}\frac{\partial}{\partial x^{j}}\bigg|_{x})=\sum_{i,j=1}^{n}g_{ij}(v)w_{1}^{i}w_{2}^{j}.
\]
Using the definition of Legendre transform one sees that $\mathcal{L}_{x}^{-1}(v)(w)=g_{v}(v,w)$
for $w\in T_{x}M$ and thus $g_{v}(v,v)=F(v)^{2}$. Different from
Riemannian metrics, $g_{v}$ is non-constant and the following tensor,
called Cartan tensor is non-zero (at least for some $v\in TM\backslash\{0\}$).
\[
A_{ijk}(v):=\frac{F(v)}{2}\frac{\partial g_{ik}}{\partial v^{k}}(v)=\frac{F(v)}{4}\frac{\partial^{3}(F^{2})}{\partial v^{i}\partial v^{j}\partial v^{k}}(v).
\]

Further, we can define the formal Christoffel symbol by 
\[
\gamma_{jk}^{i}(v):=\frac{1}{2}\sum_{l=1}^{n}g^{il}(v)\left\{ \frac{\partial g_{lj}}{\partial x^{k}}(v)-\frac{\partial g_{jk}}{\partial x^{l}}(v)+\frac{\partial g_{kl}}{\partial x^{j}}(v)\right\} 
\]
 for $v\in TM\backslash0$ and also 
\[
N_{j}^{i}(v):=\sum_{k=1}^{n}\gamma_{jk}^{i}(v)v^{k}-\frac{1}{F(v)}\sum_{k,l,m=1}A_{jk}^{i}(v)\gamma_{lm}^{k}(v)v^{l}v^{m}
\]
where $(g^{ij})$ is the inverse of $(g_{ij})$ and $A_{jk}^{i}:=\sum_{l}g^{il}A_{ljk}$.

Given the Chern connection $\nabla$ let $\omega_{j}^{i}$ be its
connection one-forms which are defined by
\[
\nabla_{v}\frac{\partial}{\partial x^{j}}=\sum_{i=1}^{n}\omega_{j}^{i}(v)\frac{\partial}{\partial x^{i}},\nabla_{v}dx^{i}=\sum_{j=1}^{n}-\omega_{j}^{i}(v)dx^{j}
\]
and by torsion-freeness can be written as 
\[
\omega_{j}^{i}=\sum_{k}\Gamma_{jk}^{i}dx^{k}.
\]
Given two non-zero vector $v,w\in T_{x}M\backslash\{0\}$, a $C^{\infty}$-vector
field $X$ and the connection one-forms, one can define the covariant
derivative $D_{v}^{w}X$ with reference vector $w$ as
\[
(D_{v}^{w}X)(x):=\sum_{i,j=1}^{n}\left\{ v^{j}\frac{\partial X^{i}}{\partial x^{j}}+\sum_{k=1}^{n}\Gamma_{jk}^{i}(w)v^{j}X^{k}\right\} \frac{\partial}{\partial x^{i}}\bigg|_{x}.
\]
In the Riemannian case, the covariant derivative does not depend on
the vector $w$ and is just the usual covariant derivative.

From the Chern connection one can also define its connection two-forms
\[
\Omega_{i}^{j}:=dw_{j}^{i}-\sum_{k=1}^{n}\omega_{j}^{k}\wedge\omega_{k}^{i}
\]
which can be also written as
\[
\Omega_{i}^{j}(v)=\frac{1}{2}\sum_{k,l=1}^{n}R_{jkl}^{i}(v)dx^{k}\wedge dx^{l}+\frac{1}{F(v)}\sum_{k,l=1}^{n}P_{jkl}^{i}(v)dx^{x}\wedge\delta v^{l}
\]
where we require $R_{jkl}^{i}=-R_{jlk}^{i}$ and $\delta v^{k}=dv^{k}+\sum_{l}N_{l}^{k}dx^{l}$.

With the help of $R_{jkl}^{i}$ one can define the Riemannian tensor
with reference vector $v\in TM$
\[
R^{v}(w,v)v:=\sum_{i,j,k,l=1}^{n}v^{j}R_{jkl}^{i}(v)w^{k}v^{l}\frac{\partial}{\partial x^{i}}\bigg|_{x}
\]
which enjoys the following 
\[
g_{v}(R^{v}(w,v)v,w')=g_{v}(R^{v}(w',v)v,w)\mbox{ and }R^{v}(v,v)=0.
\]

Given all those definition we finally have the flag curvature 
\[
\mathcal{K}(v,w):=\frac{g_{v}(R^{v}(w,v)v,w)}{g_{v}(v,v)g_{v}(w,w)-g_{v}(v,w)^{2}}
\]
and the Ricci curvature 
\[
Ric(v):=\sum_{i=1}^{n-1}\mathcal{K}(v,e_{i})
\]
where $e_{1},e_{2},\cdots,e_{n-1},v/F(v)$ form an orthonormal basis
of $T_{x}M$ w.r.t. $g_{v}$.

On unweighted Finsler manifolds we say that $(M,F)$ has Ricci curvature
bounded from below if 
\[
Ric(v)\ge K
\]
for every unit vector $v\in TM$. For weighted manifolds we need the
following: Let $\mu$ be the reference measure and $\operatorname{vol}_{g_{v}}$
be the Lebesgue measure on $T_{x}M$ induced by $g_{v}$. If $\mu_{x}$
denotes the measure $T_{x}M$ induced by $\mu$ define 
\[
\mathcal{V}(v):=\log\left(\frac{\operatorname{vol}_{g_{v}}(B_{T_{x}M}^{+}(0,1)}{\mu_{x}(B_{T_{x}M}^{+}(0,1)}\right)
\]
 where $B_{T_{x}M}^{+}(0,1)$ denotes the (forward) unit ball of radius
$1$ w.r.t. the norm $F|_{T_{x}M}$. Further, let 
\[
\partial_{v}\mathcal{V}:=\frac{d}{dt}\bigg|_{t=0}\mathcal{V}(\dot{\eta}(t)),\partial_{v}^{2}\mathcal{V}:=\frac{d}{dt}\bigg|_{t=0}\mathcal{V}(\dot{\eta}(t))
\]
where $\eta:(-\epsilon,\epsilon)\to M$ is a geodesic with $\dot{\eta}(0)=v$. 
\begin{defn}
[Weighted Ricci curvature] Define the following objects:
\begin{enumerate}
\item $Ric_{n}(v):=\begin{cases}
Ric(v)+\partial_{v}^{2}\mathcal{V} & \mbox{if }\partial_{v}\mathcal{V}=0\\
-\infty & \mbox{otherwise}
\end{cases}$
\item $Ric_{N}(v):=Ric(v)+\partial_{v}^{2}\mathcal{V}+\frac{\partial_{v}\mathcal{V}}{N-n}$
for $N\in(n,\infty)$.
\item $Ric_{\infty}(v):=Ric(v)+\partial_{v}^{2}\mathcal{V}$
\end{enumerate}
Which is called the (weighted) $n$-Ricci curvature, resp. $N$- and
$\infty$-Ricci curvature of the weighted Finsler manifold $(M,F,\mu)$. \end{defn}
\begin{rem*}
By a recent paper of Ohta \cite{Ohta2013} it also makes sense to
define the $ $$N$-Ricci curvature for negative $N$. 
\end{rem*}
Now a lower curvature bound $K$ on the $N$-Ricci curvature (resp.
$n$-, $\infty$-Ricci curvature) is nothing but 
\[
Ric_{N}(v)\ge K
\]
for all unit vector $v\in TM$.

\subsubsection*{Geodesics and first and second variation formula}

Given a $C^{1}$-curve $\eta:[0,r]\to M$ its arclength is defined
by 
\[
\mathcal{L}(\eta):=\int_{0}^{r}F(\dot{\eta}_{t})dt
\]
where $\dot{\eta}_{t}=\frac{d}{dt}\eta_{t}$. We say that a $C^{\infty}$-curve
$\eta$ is a geodesic (of constant speed) if $D_{\dot{\eta}}^{\dot{\eta}}\dot{\eta}=0$
on $(0,r)$. Note however that the reverse curve $\bar{\eta}_{t}=\eta_{(r-t)}$
may not be a geodesic (not even w.r.t. the reverse Finsler structure
$\bar{F}$). 

The exponential map is given by $exp(v)=\exp_{\pi(v)}v:=\eta(1)$
if there is a geodesic $\eta:[0,1]\to M$ with $\dot{\eta}_{0}=v$.
Note however, that the exponential map is only $C^{1}$ at the zero
section. We say that $(M,F)$ is forward geodesically complete if
the exponential map is define on all of $TM$, i.e. if we can extend
any constant speed geodesic $\eta$ to geodesic $\eta:[0,\infty)\to M$.
For such case, we can connect any two points of $M$ by a minimal
geodesic, i.e. for every $x,y\in M$ there is a geodesic $\eta$ from
$x$ to $y$ such that $\mathcal{L}(\eta)=d(x,y)$.

Given a unit vector $v\in T_{x}M$, let $r(v)\in(0,\infty]$ be the
the supremum of all $r>0$ such that $t\mapsto exp_{x}tv$ is a minimal
geodesic. If $r(v)<\infty$ then we say that $exp_{x}(r(v)v)$ is
a cut-point of $x$ and denote by $\operatorname{Cut}(x)$ the set
of all cut points of $x$, also called the cut locus of $x$. One
can show that the exponential map is a $C^{\infty}$-diffeomorpism
from $\{tv\,|\, v\in T_{x}M,F(v)=1,t\in(0,r(v))\}$ to $M\backslash(\operatorname{Cut}(x)\cup\{x\})$.
This also shows that the distance $d(x,\cdot)$ is $C^{\infty}$ away
from $x$ and the cut locus of $x$. In particular, if $L:[0,\infty)\to[0,\infty)$
is $C^{\infty}$ away from $0$ then $L(d(x,\cdot))$ is $ $$C^{\infty}$
away from $x$ and the cut locus of $x$.

A variation of a $C^{\infty}$-curve $\eta:[0,r]\to M$ is a $C^{\infty}$-map
$\sigma:[0,r]\times(-\epsilon,\epsilon)\to M$ such that $\eta(t)=\sigma(t,0)$.
We abbreviate the derivatives as 
\[
T(t,s)=\partial_{t}\sigma(t,s),U(t,s)=\partial_{s}\sigma(t,s).
\]

The first variation of the arclenth is given by
\[
\frac{\partial\mathcal{L}(\sigma_{s})}{\partial s}=\left[\frac{g_{T}(U,T)}{F(T)}\right]_{t=0}^{r}-\int_{0}^{r}g_{T}\left(U,D_{T}^{T}\left[\frac{T}{F(T)}\right]\right)dt.
\]
where we dropped the dependency on $t$ and $s$. In case $\eta$
is a geodesic, the second term is zero. Furthermore, the second variation
along a geodesic has the form
\[
\frac{\partial^{2}\mathcal{L}(\sigma_{s})}{\partial s^{2}}\bigg|_{s=0}=I(U,U)+\left[\frac{g_{T}(U,T)}{F(T)}\right]_{t=0}^{r}-\int_{0}^{r}\frac{1}{F(T)}\left(\frac{\partial F(T)}{\partial s}\right)^{2}dt
\]
where 
\[
I(V,W):=\frac{1}{F(\dot{\eta})}\int_{0}^{r}\left\{ g_{\dot{\eta}}(D_{\dot{\eta}}^{\dot{\eta}}V,D_{\dot{\eta}}^{\dot{\eta}}W)-g_{\dot{\eta}}(R^{\dot{\eta}}(V,\dot{\eta})\dot{\eta},W)\right\} dt.
\]
Since the tensor $R^{\dot{\eta}}$ enjoys some symmetry, we easily
see that $I(V,W)=I(W,V)$. And if $V$ is a Jacobi field then the
second term is zero and one can show 
\[
I(V,W)=\frac{1}{F(\dot{\eta})}\left[g_{\dot{\eta}}(D_{\dot{\eta}}^{\dot{\eta}}V,W)\right]_{t=0}^{r}.
\]

And finally, we say that a $C^{\infty}$-vector field $J$ along a
geodesic $\eta:[0,r]\to M$ is a Jacobi field if it satisfies
\[
D_{\dot{\eta}}^{\dot{\eta}}D_{\dot{\eta}}^{\dot{\eta}}J+R^{\dot{\eta}}(J,\dot{\eta})\dot{\eta}=0.
\]
Any Jacobi field can be represented as a variational vector field
of some geodesic variation $\sigma$ (each $\sigma_{s}$ is a geodesic)
and vice versa.

\section{$c_{p}$-concave Functions}

Assume throughout $M$ is a proper geodesic space.

Define for $1<p<\infty$ 
\[
c_{p}(x,y)=\frac{d^{p}(x,y)}{p}.
\]
 We say that a function $\phi:X\to\Rinf$ is proper, if it is not
identically $\text{-\ensuremath{\infty}}.$
\begin{rem*}
Almost all results about $c_{p}$-concave functions also hold for
$c_{L}$-concave functions by exchanging $c_{p}$ with $c_{L}$ where
$L$ is a strictly convex, increasing, function differentiable in
$(0,\infty)$ and 
\[
c_{L}(x,y)=L(d(x,y)).
\]
If $L$ is fixed then $c_{t}$ will be an abbreviation for $c_{L_{t}}$
where $L_{t}(r)=L(r/t)$.

The definition of $c_{p}$-transform can be localized. This has the
advantage to give properness of the function and Lipschitz regularity
on the domain also in the non-compact setting.\end{rem*}
\begin{defn}
[$c_p$-transform and the subset $\mathcal{I}^c_p(X,Y)$] Let $X$
and $Y$ be two subsets of $M$. The $c_{p}$-transform relative to
$(X,Y)$ of a function $\phi:X\to\Rinf$ is defined as 
\[
\phi^{c_{p}}(y)=\inf_{x\in M}c_{p}(x,y)-\phi(x).
\]
In case $X=Y=M$ we just write $c_{p}$-transform. Similarly, we define
the $\bar{c}_{p}$-transform relative to $(Y,X)$ of a function $\psi:Y\to\Rinf$
as
\[
\psi^{\bar{c}_{p}}(x)=\inf_{y\in Y}c_{p}(x,y)-\psi(y).
\]
We say that a proper function $\phi:X\to\Rinf$ is $c_{p}$-concave
(relative to $(X,Y)$) if there is a function $\psi:Y\to\Rinf$ such
that $\phi=\psi^{\bar{c}_{p}}$. Similarly, we define $\bar{c}_{p}$-concave
function relative to $(Y,X)$ as those proper function $\psi$ such
that $\psi=\phi^{c_{p}}$ for some function $\phi:X\to\Rinf$.

Let $\mathcal{I}^{c_{p}}(X,Y)$ (resp. $\mathcal{I}^{\bar{c}_{p}}(Y,X)$)
denote the set of all $c_{p}$-concave functions relative to $(X,Y)$
(resp. the set of all $\bar{c}_{p}$-concave functions relative to
$(Y,X)$).
\end{defn}
Note that $\mathcal{I}^{c_{p}}(X,Y')\subset\mathcal{I}^{c_{p}}(X,Y)$
for all $Y'\subset Y$. Indeed, if $\phi\in\mathcal{I}^{c_{p}}(X,Y')$
and $\psi':Y'\to\Rinf$ is such that $\phi=(\psi')^{\bar{c}_{p}}$
then let 
\[
\psi(y)=\begin{cases}
\psi'(y) & \quad\mbox{if }y\in Y'\\
\text{-\ensuremath{\infty}} & \quad\mbox{if }y\in Y\backslash Y'.
\end{cases}
\]
Then obviously $\phi=(\psi')^{\bar{c}_{p}}=\psi^{\bar{c}_{p}}$ and
thus $\phi\in\mathcal{I}^{c_{p}}(X,Y)$. Similarly, if $X'\subset X$,
we can extend any function $\phi\in\mathcal{I}^{c_{p}}(X',Y)$ to
a $c_{p}$-concave $\phi\in\mathcal{I}^{c_{p}}(X,Y)$ by letting $\phi$
be the $\bar{c}_{p}$-transform of $\psi:Y\to\Rinf$ relative to $(Y,X)$.

The following is easy to show:
\begin{lem}
\label{lem:cp-calculus}Let $\phi:M\to\mathbb{R}\cup\{-\infty\}$
and let all statement be relative to some pair $(X,Y)$ of compact
subsets. Then the following holds:
\begin{enumerate}
\item $\phi\le\phi^{c_{p}\bar{c}_{p}}$ and $\phi^{c_{p}}=\phi^{c_{p}\bar{c}_{p}c_{p}}$
\item if $\phi$ is not identically $-\infty$ then $\phi$ is $c_{p}$-concave
iff $\phi=\phi^{c_{p}\bar{c}_{p}}$
\item if $\{\phi_{i}\}_{i\in I}\subset\mathcal{I}^{c_{p}}(X,Y)$ for some
index set $I$ and $\phi(x):=\inf_{I}\phi_{i}(x)$ is a proper function,
then $\phi\in\mathcal{I}^{c_{p}}(X,Y)$.
\item If $\phi$ is $c_{p}$-concave, then it is Lipschitz continuous and
its Lipschitz constant is bounded from above by a constant depending
only on $X,Y$ and $p$.
\end{enumerate}
\end{lem}
\begin{cor}
If $M$ is compact and $\phi$ is $c_{p}$-concave then $\phi$ is
Lipschitz continuous with Lipschitz constant bounded from above by
a constant only depending on $M$ and $p$. In particular, the set
of $c_{p}$-concave functions with $\phi(x_{0})=0$ is a precompact
subset of $C^{0}(M,\mathbb{R})$ with bounded Lipschitz constant only
depending on $M$.
\end{cor}
Since $X$ and $Y$ are compact, the $\inf$ in the definition of
$c_{p}/\bar{c}_{p}$-transform is actually achieved and the following
sets are non-empty for each $c_{p}/\bar{c}_{p}$-concave functions.
\begin{defn}
[$c_p$-subdifferential] Let $X$ and $Y$ be two compact subsets
of $M$ and $\phi:X\to\Rinf$ be a $c_{p}$-concave function relative
to $(X,Y)$ then the $c_{p}$-subdifferential of $\phi$ at $x\in X$
is the non-empty set 
\[
\partial^{c_{p}}\phi(x)=\{y\in Y\,|\,\phi(x)=c_{p}(x,y)-\phi^{c_{p}}(y)\}.
\]
Similarly, we define $\bar{c}_{p}$-subdifferential of a $\bar{c}_{p}$-concave
function $\psi:Y\to\Rinf$ as the non-empty set 
\[
\partial^{\bar{c}_{p}}\psi(y)=\{x\in X\,|\,\psi(y)=c_{p}(x,y)-\phi^{c_{p}}(x)\}.
\]

\end{defn}
It is not difficult to see that 
\[
y\in\partial^{c_{p}}\phi(x)\,\Longleftrightarrow\, x\in\partial^{\bar{c}_{p}}\phi^{c_{p}}(y)
\]
whenever $\phi$ is $c_{p}$-concave. Furthermore, $y\in\partial^{c_{p}}\phi\left(\partial^{\bar{c}_{p}}\phi^{c_{p}}(y)\right)$.
\begin{lem}
[Semicontinuity of the $c_p$-subdifferential] \label{lem:cts-subdiff}Let
$X,Y$ be two compact subsets of $M$ and $\phi$ be a $c_{p}$-concave
function relative to $(X,Y)$. Then, whenever $y_{n}\in\partial^{c_{p}}\phi(x_{n})$
for some sequence $(x_{n},y_{n})\in X\times Y$ such that $(x_{n},y_{n})\to(x,y)$,
we have $y\in\partial^{c_{p}}\phi(x)$. In particular, if $\partial^{c_{p}}\phi(x)=\{y\}$
is single-valued, then for every neighborhood $V$ of $y$, the set
$\left(\partial^{c_{p}}\phi\right)^{-1}(V)$ contains a neighborhood
$U$ of $x$ (relative to $X$), in particular, for any $x'\in U\cap X$
there is a $y'\in\partial^{c_{p}}\phi(x)\cap V\cap Y$.\end{lem}
\begin{proof}
Note that $\phi$ and $\phi^{c_{p}}$ are Lipschitz continuous on
$X$, resp. $Y$. Since $X$ and $Y$ are closed we have $(x,y)\in X\times Y$
and hence 
\[
0=\phi(x_{n})+\phi^{c_{p}}(y_{n})-c_{p}(x_{n},y_{n})\to\phi(x)+\phi^{c_{p}}(y)-c_{p}(x,y)=0,
\]
i.e. $y\in\partial^{c_{p}}\phi(x)$. 

The second statement directly follows from the set-wise continuity
of $x'\mapsto\partial^{c_{p}}\phi(x')$ at $x$ in case $\partial^{c_{p}}\phi(x)$
is single-valued.
\end{proof}
In case $M$ is non-compact and $X=Y=M$ we can show the following. 
\begin{lem}
\label{lem:cp-concave-lipschitz}Let $\phi$ be a $c_{p}$-concave
function and $\Omega\subset X$ the interior of $\{\phi>-\infty\}$.
Then $\phi$ is locally bounded and locally Lipschitz on $\Omega$
and for every compact set $K\subset\Omega$ the set $\cup_{x\in K}\partial^{c_{p}}\phi$
is bounded and not empty.\end{lem}
\begin{rem*}
This lemma extends \cite[Lemma 3.3]{Gigli2013} to all cases $p\ne2$.
The same result also holds for $c_{L}$-concave functions if we assume
that $L$ is strictly increasing and convex and satisfies the following
\[
L(R)-L(R-\epsilon)\to\infty
\]
as $R\to\infty$ for any $\epsilon>0$, i.e. if $L(R)=\int_{0}^{R}l(r)dr$
with $l$ increasing and unbounded.\end{rem*}
\begin{proof}
By definition $\phi=(\phi^{c_{p}})^{\bar{c}_{p}}$ and thus $\phi$
is the infimum of a family of continuous functions and therefore upper
semicontinuous and locally bounded from above. 

As in \cite{Gigli2013}, we prove that $\phi$ is locally bounded
from below by contradiction. Assuming $\phi$ is not locally bounded
near a point $x_{\infty}\in\Omega$, there is a sequence $\Omega\ni x\to x_{\infty}$
such that $\phi(x_{n})\to-\infty$.

Furthermore, for every $n\in\mathbb{N}$ we can find $y_{n}\in M$
such that 
\[
\phi(x_{n})\ge c_{p}(x_{n},y_{n})-\phi^{c_{p}}(y_{n})-1.
\]
This immediately yields $\phi^{c_{p}}(y_{n})\to\infty$. Because 
\[
\mathbb{R}\ni\phi(x_{\infty})\le c_{p}(x_{\infty},y_{n})-\phi^{c_{p}}(y_{n}),
\]
we must have $c_{p}(x_{\infty},y_{n})\to\infty$, i.e. $y_{n}$ is
an unbounded sequence. In addition, also note $c_{p}(x_{n},y_{n})\to\infty$.

So w.l.o.g. we can assume $c_{p}(x_{n},y_{n})\ge1$. Now let $\gamma^{n}:[0,d(x_{n},y_{n})]\to M$
be a unit speed minimal geodesic between $x_{n}$ and $y_{n}$. We
will show that 
\[
\sup_{\bar{B}_{1}(\gamma_{1}^{n})}\phi\to-\infty\mbox{ as }n\to\infty.
\]
In order to prove this, note that for $x\in\bar{B}_{1}(\gamma_{1}^{n})$
we have $d(x,\gamma_{1}^{n})\le1=d(x_{n},\gamma_{1}^{n})$ and thus
\begin{eqnarray*}
\phi(x) & \le & c_{p}(x,y_{n})-\phi^{c_{p}}(y_{n})\le\frac{\left(d(x,\gamma_{1}^{n})+d(\gamma_{1}^{n},y_{n})\right)^{p}}{p}-\phi^{c_{p}}(y_{n})\\
 & \le & \frac{\left(d(x_{n},\gamma_{1}^{n})+d(\gamma_{1}^{n},y_{n})\right)^{p}}{p}-\phi^{c_{p}}(y_{n})\\
 & = & c_{p}(x_{n},y_{n})-\phi^{c_{p}}(y_{n})\le\phi(x_{n})+1.
\end{eqnarray*}
Because $\phi(x_{n})\to-\infty$, we proved our claim.

Since $M$ is proper, we can assume $\gamma_{1}^{n}\to z$ such that
$d(x_{\infty},z)=1$. In addition, the claim implies that $\phi$
is identically $-\infty$ in the interior of $B_{1}(z)$. But this
contradicts $x_{\infty}\in\Omega$. Therefore, $\phi$ is locally
bounded in $\Omega$.

It remains to show that $\phi$ is locally Lipschitz. Choose $\bar{x}\in\Omega$
and $r>0$ such that $B_{2r}(\bar{x})\subset\Omega$. Choose $x\in B_{r}(\bar{x})$
and let $y_{n}$ be a sequence such that 
\[
\phi(x)=\lim_{n\to\infty}c_{p}(x,y_{n})-\phi^{c_{p}}(y_{n}).
\]
We will show that $y_{n}\in B_{C}(\bar{x})$ for some $C$ only depending
on $\bar{x},r$ and $\phi$. We may assume $d(x,y_{n})>r$ otherwise
we are done. Let $\gamma^{n}:[0,d(x,y_{n})]\to M$ a minimal unit
speed geodesic from $x$ to $y_{n}$. We have 
\[
\limsup_{n\to\infty}\phi(x)-\phi(\gamma_{r}^{n})\ge\limsup_{n\to\infty}c_{p}(x,y_{n})-c_{p}(\gamma_{r}^{n},y_{n})
\]
and we know already that the left hand side is bounded. If $R_{n}:=d(y_{n},x)\to\infty$
then for $l(r)=r^{p-1}$ 
\[
c_{p}(x,y_{n})-c_{p}(\gamma_{r}^{n},y_{n})=\int_{R_{n}-r}^{R_{n}}l(s)ds\ge r\cdot l(R_{n}-r)\to\infty
\]
which is a contradiction. Hence $y_{n}$ is bounded and by properness
has accumulation points which all belong to $\partial^{c_{p}}\phi(x)$.
Similarly, we can show that $\cup_{x\in K}\partial^{c}\phi(x)$ is
bounded for any compact $K$.

Finally, for all $x\in B_{r}(\bar{x})$
\begin{eqnarray*}
\phi(x) & = & \inf_{y\in M}c_{p}(x,y)-\phi^{c_{p}}(y)\\
 & = & \min_{B_{C}(\bar{x})}c_{p}(x,y)-\phi^{c_{p}}(y).
\end{eqnarray*}
Since for $y\in B_{C}(\bar{x})$ the functions $x\mapsto c_{p}(x,y)-\phi^{c_{p}}(y)$
are uniformly Lipschitz on $B_{r}(\bar{x)}$, $\phi$ is locally Lipschitz
as well.
\end{proof}
For $x,y\in M$ and $t\in[0,1]$ define $Z_{t}(x,y)\subset M$ as
\[
Z_{t}(x,y):=\{z\in M\,|\, d(x,z)=td(x,y)\,\mbox{ and }\, d(z,y)=(1-t)d(x,y)\}.
\]
If there is a unique geodesic between $x$ and $y$ then obviously
$Z_{t}(x,y)=\{\gamma(t)\}$. Furthermore, for general set $X,Y\subset M$
define
\[
Z_{t}(x,Y):=\bigcup_{y\in Y}Z_{t}(x,y)
\]
and $Z_{t}(X,Y)$ as 
\[
Z_{t}(X,Y):=\bigcup_{x\in X}Z_{t}(x,Y).
\]

The following three results are crucial ingredients to show absolute
continuity of the interpolation measure in the smooth setting (see
Lemma \ref{lem:abs-interp} below). It generalizes \cite[Claim 2.4]{CMS2001}
and will be used in Lemma \ref{lem:inf-dist} (see \cite[(3.1) p. 221]{Ohta2009}
for the case $p=2$). Lemma \ref{lem:star-shaped} will also help
to prove ``almost everywhere'' second order differentiability of
$c_{p}$-concave functions. This proof is much easier than the original
one given in \cite{CMS2001,Ohta2008}. There is also a counterpart
in the Orlicz-Wasserstein case which is stated and proved in the appendix
(see Lemma \ref{lem:dist-ineq-Orlicz}).
\begin{lem}
\label{lem:p-dist-ineq}If $x,y\in M$ and $z\in Z_{t}(x,y)$ for
some $t\in[0,1]$. Then for all $m\in M$
\[
t^{p-1}d^{p}(m,y)\le d^{p}(m,z)+t^{p-1}(1-t)d^{p}(x,y).
\]
Furthermore, choosing $x=m$ this becomes an equality.\end{lem}
\begin{proof}
Using the triangle inequality, the fact that $d(z,y)=(1-t)d(x,y)$
and that $r\mapsto r^{p}$ is convex for $p>1$, we get 
\begin{eqnarray*}
t^{p-1}d^{p}(m,y) & \le & t^{p-1}\left\{ t\cdot\frac{1}{t}d(m,z)+(1-t)d(x,y)\right\} ^{p}\\
 & \le & t^{p-1}\left\{ t\cdot\left(\frac{1}{t}d(m,z)\right)^{p}+(1-t)d^{p}(x,y)\right\} \\
 & = & d^{p}(m,z)+t^{p-1}(1-t)d^{p}(x,y).
\end{eqnarray*}
Furthermore, choosing $m=x$ we see that each inequality is actually
an equality.\end{proof}
\begin{lem}
\label{lem:inf-dist}Let $\eta:[0,1]\to M$ be a geodesic between
two distinct points $x$ and $y$. For $t\in(0,1]$ define 
\[
f_{t}(m):=-c_{p}(m,\eta_{t}).
\]
Then for some fixed $t\in[0,1]$ the function $h(m):=f_{t}(m)-t^{p-1}f_{1}(m)$
has a minimum at $x$. \end{lem}
\begin{proof}
Using Proposition \ref{lem:p-dist-ineq} above we have for $z=\eta_{t}\in Z_{t}(x,y)$
\begin{eqnarray*}
-ph(m)=t^{p-1}d^{p}(m,y)-d^{p}(m,z) & \le & t^{p-1}(1-t)d^{p}(x,y)\\
 & = & t^{p-1}d^{p}(x,y)-d^{p}(x,\eta_{t})=-ph(x).
\end{eqnarray*}

\end{proof}
The following lemma will be useful to describe the interpolation potential
of the optimal transport map. It generalizes \cite[5.1]{CMS2001}
to the cases $p\ne2$.
\begin{lem}
[$c_p$-concave functions form a star-shaped set] \label{lem:star-shaped}Let
$X$ and $Y$ be compact subsets of $M$ and let $t\in[0,1]$. If
$\phi\in\mathcal{I}^{c_{p}}(X,Y)$ then $t^{p-1}\phi\in\mathcal{I}^{c_{p}}(X,Z_{t}(X,Y))$.\end{lem}
\begin{proof}
Note that the cases $t=0$ and $t=1$ are trivial since $0\in\mathcal{I}^{c_{p}}(X,X)$.
For the rest we follow the strategy of \cite[Lemma 5.1]{CMS2001}.
Let $t\in[0,1]$ and $y\in Y$ and define $\phi(x):=c_{p}(x,y)=d_{y}^{p}(x)/p$.
We claim that the following representation holds
\[
t^{p-1}d_{y}^{p}(m)/p=\inf_{z\in Z_{t}(X,y)}\left\{ d_{z}^{p}(m)/p+\inf_{\{x\in X\,|\, z\in Z_{t}(x,y)\}}t^{p-1}(1-t)d_{y}^{p}(x)/p\right\} .
\]
Indeed, by Lemma \ref{lem:p-dist-ineq} the left hand side is less
than or equal to the right hand side for any $z\in Z_{t}(X,y)$. Furthermore,
choosing $x=m$ we get an equality and thus showing the representation.

Now note that the claim implies that $t^{p-1}\phi$ is the $\bar{c}_{p}$-transform
of the function 
\[
\psi(z)=-\inf_{\{x\in X\,|\, z\in Z_{t}(x,y)\}}t^{p-1}(1-t)d_{y}^{p}(x)/p
\]
(real-valued on $Z_{t}(X,Y)$) and therefore $t^{p-1}\phi$ is $c_{p}$-concave
relative to $(X,Z_{t}(X,y))$. Since $\mathcal{I}^{c_{p}}(X,Z_{t}(X,y))\subset\mathcal{I}^{c_{p}}(X,Z_{t}(X,Y))$
we see that each $t^{p-1}d_{y}^{p}/p$ is in $\mathcal{I}^{c_{p}}(X,Z_{t}(X,Y))$.

It remains to show that for an arbitrary $c_{p}$-concave function
$\phi$ and $t\in[0,1]$ the function $t^{p-1}\phi$ is $c_{p}$-concave
relative to $(X,Z_{t}(X,Y))$. Since $\phi=\phi^{c_{p}\bar{c}_{p}}$
we have 
\[
t^{p-1}\phi(x)=\inf_{y}t^{p-1}c_{p}(x,y)-t^{p-1}\phi^{c_{p}}(y).
\]
But each function 
\[
\psi_{y}(x)=t^{p-1}c_{p}(x,y)-t^{p-1}\phi^{c_{p}}(y)
\]
is $c_{p}$-concave relative to $(X,Z_{t}(X,Y))$ and $\phi$ is proper,
thus also the infimum is $c_{p}$-concave relative to $(X,Z_{t}(X,Y))$,
i.e. $t^{p-1}\phi\in\mathcal{I}^{c_{p}}(X,Z_{t}(X,Y))$.  
\end{proof}
Finally, assuming the space is non-branching, e.g. a Riemannian or
Finsler manifold, we want to show the well-known result that the optimal
transport rays cannot intersect at intermediate times. The proof is
easily adaptable to Orlicz-Wasserstein spaces and will give positivity
of the Jacobian for the interpolation measures. 
\begin{defn}
[non-branching spaces] A geodesic space $(M,d)$ is said to be non-branching,
if for all $x,y,y'\in M$ with $d(x,y)=d(x,y)>0$ one always has 
\[
Z_{t}(x,y)\cap Z_{t}(x,y')\ne\varnothing\mbox{ for some }t\in(0,1)\mbox{ }\implies y=y'.
\]
\end{defn}
\begin{lem}
\label{lem:inj}Assume $M$ is non-branching and $\mu_{0}$ and $\mu_{1}$
two measures in $\mathcal{P}_{p}(M)$. If $\pi$ is an optimal transport
plan between $\mu_{0}$ and $\mu_{1}$ then there is a subset $U$
of $M\times M$ of $\pi$-measure $1$ such that for $i=1,2$ let
$\gamma_{i}$ be a geodesic for $(x_{i},y_{i})\in U$, then $\gamma_{1}(t)=\gamma_{2}(t)$
for some $t\in[0,1]$ implies $(x_{1},y_{1})=(x_{2},y_{2})$.\end{lem}
\begin{rem*}
Exactly the same results for the optimal transport plan with cost
function $L(d(\cdot,\cdot))$. In particular, it holds for Orlicz-Wasserstein
spaces using \cite[Proposition 3.1]{Sturm2011} and $c_{\lambda}$-cyclicity
of the support where $\lambda=w_{L}(\mu_{0},\mu_{1})$ (see appendix
for definition of $w_{L}$).\end{rem*}
\begin{proof}
According to \cite[Theorem 5.10]{Villani2009} there is a subset $U$
of $M\times M$ of $\pi$-measure $1$ such that for each $(x_{i},y_{i})\in U$
\[
\frac{d(x_{1},y_{1})^{p}}{p}+\frac{d(x_{2},y_{2})^{p}}{p}\le\frac{d(x_{1},y_{2})^{p}}{p}+\frac{d(x_{2},y_{1})^{p}}{p},
\]
this property is called $c_{p}$-cyclically monotone (of order $2$)
(see \cite[Definition 5.1]{Villani2009}). 

Now assume for some $(x_{i},y_{i})\in U$ there is a $t\in(0,1)$
such that we have $z=\gamma_{1}(t)=\gamma_{2}(t)$. Then
\begin{eqnarray*}
d(x_{1},y_{2})^{p}+d(x_{2},y_{1})^{p} & \le & \left(d(x_{1},z)+d(z,y_{2})\right)^{p}+\left(d(x_{2},z)+d(z,y_{1})\right)^{p}\\
 & = & \left(td(x_{1},y_{1})+(1-t)d(x_{2},y_{2})\right)^{p}\\
 &  & +\,\left(td(x_{2},y_{2})+(1-t)d(x_{1},y_{1})\right)^{p}\\
 & \le & td(x_{2},y_{2})^{p}+(1-t)d(x_{2},y_{2})^{p}\\
 &  & +\, td(x_{2},y_{2})^{p}+(1-t)d(x_{1},y_{1})^{p}\\
 & = & d(x_{1},y_{1})^{p}+d(x_{2},y_{2})^{p}.
\end{eqnarray*}
Because $U$ is $c_{p}$-cyclically monotone we see that the inequality
actually must be an equality. Since $p>1$ we must have $d(x_{1},y_{1})=d(x_{2},y_{2})$
and 
\[
d(x_{1},y_{2})^{p}+d(x_{2},y_{1})^{p}=d(x_{1},y_{1})^{p}+d(x_{2},y_{2})^{p}.
\]
This also implies that $d(x_{1},y_{2})=d(x_{1},y_{1})=d(x_{2},y_{1})$.
Because $z$ is the common $t:(1-t)$ fraction point and there are
no branching geodesics, we must have $x_{1}=x_{2}$ and $y_{1}=y_{2}$.
\end{proof}

\section{Interpolation in the Smooth Setting}

In this section we will assume throughout that $M$ is a $C^{\infty}$-Finsler
manifold. We are going to show that the interpolation inequality can
be proven along $p$-Wasserstein geodesics. From this inequality and
a lower Ricci curvature bound, one can easily derive the curvature
dimension condition as defined in the next section. Furthermore, it
turns out to be equivalent to the lower Ricci curvature bound. As
Ohta \cite{Ohta2009} noted, in the Finsler setting one needs additional
assumptions on the background measure get a lower (weighted) Ricci
curvature bound from the curvature dimension condition.

\subsection*{Notation and technical ingredients}

Let $q$ be the Hölder conjugate of $1<p<\infty$, i.e. $\frac{1}{q}+\frac{1}{p}=1$
or equivalently $(p-1)(q-1)=1$.

In order to get a nice description of the interpolation maps we need
to define the following $q$-gradient 
\[
\nabla^{q}\phi:=|\nabla\phi|^{q-2}\nabla\phi.
\]
Note that for $v\in T_{x}M$ 
\[
\nabla\phi(x)=|v|^{p-2}v
\]
iff
\[
\nabla^{q}\phi=v.
\]
 Also note that $\nabla\phi=0$ iff $\nabla^{q}\phi=0$, and $x\mapsto\nabla^{q}\phi(x)$
is continuous iff $x\mapsto\nabla\phi(x)$ is. For $t>0$ we have
\[
\nabla^{q}(t^{p-1}\phi)=t\nabla^{q}\phi.
\]
In addition, we use the abbreviation $\mathcal{K}d\phi=\nabla^{q}\phi$
(note that $\mathcal{L}d\phi=\nabla\phi$). This is indeed invertible,
continuous from $T^{*}M\to TM$ and $C^{\infty}$ away from the zero
section. Furthermore, 
\[
\mathcal{K}_{x}t^{p-1}d\phi_{x}=t\nabla^{q}\phi(x).
\]

\begin{rem*}
$\mathcal{K}_{x}$ can actually be seen as the Legendre transform
from $T_{x}^{*}M\to T_{x}M$ that associates to each cotangent vector
$\alpha\in T^{*}M$ the unique tangent vector $v=\mathcal{K}(\alpha)\in TM$
such that $F(v)^{p}=F^{*}(\alpha)^{q}$ and $\alpha(v)=F^{*}(\alpha)^{q}$
where $F^{*}$ denotes the dual norm of $F$ on $T^{*}M$. 
\end{rem*}
In order to show that optimal transport is almost everywhere away
from the cut locus we need to following result. Its proof is based
on \cite[Lemma 3.1]{Ohta2009}. 
\begin{lem}
[Cut locus charaterization] \label{lem:cut}If $y\ne x$ is a cut
point of $x$, then $f(z):=d^{p}(z,y)/p$ satisfies
\[
\liminf_{v\to0\in T_{x}M}\frac{f(\xi_{v}(1))+f(\xi_{v}(-1))-2f(x)}{F(v)^{2}}=-\infty
\]
where $\xi_{v}:[-1,1]\to M$ is the geodesic with $\dot{\xi}_{v}(0)=v$.\end{lem}
\begin{proof}
First recall that $y$ is a cut point of $x$ if either there are
two minimal geodesics from $x$ to $y$, or $y$ is the first conjugate
point along a unique geodesic $\eta$ from $x$ to $y$, i.e. there
is a Jacobi field along $\eta$ vanishing only at $x$ and $y$ (see
\cite[Corollary 8.2.2]{BCS2000}).

So let's first assume there are two distinct unit speed geodesics
$\eta,\zeta:[0,d(x,y)]\to M$ from $x$ to $y$ and let $v=\dot{\zeta}(0)$
and $w=\dot{\eta}(0)$. For fixed small $\epsilon>0$ set $y_{\epsilon}=\eta(d(x,y)-\epsilon)$
then $y_{\epsilon}\notin\operatorname{Cut}(x)\cup\{x\}$ and using
the first variation formula we get for $t>0$
\begin{eqnarray*}
f(\xi_{v}(-t))-f(x) & \le & \left\{ d(\xi_{v}(-t),y_{\epsilon})+\epsilon\right\} ^{p}/p-\left\{ d(x,y_{\epsilon})+\epsilon\right\} ^{p}/p\\
 & = & t\left\{ d(x,y_{\epsilon})+\epsilon\right\} ^{p-1}g_{\dot{\eta}(0)}(v,\dot{\eta}(0))+\mathcal{O}(t^{2})\\
 & = & td^{p-1}(x,y)g_{\dot{\eta}(0)}(v,\dot{\eta}(0))+\mathcal{O}(t^{2}).
\end{eqnarray*}
The term $\mathcal{O}(t^{2})$ is ensured by smoothness of $\xi_{v}$
and by the fact that $x\ne y_{\epsilon}$. We also get by Taylor formula
\[
f(\xi_{v}(t))-f(x)=\left\{ d(x,y)-t\right\} ^{p}/p-d^{p}(x,y)/p=-td^{p-1}(x,y)+\mathcal{O}(t^{2}).
\]
Combining these two facts with $g_{w}(v,w)<1$ ($\eta$ and $\xi$
are distinct), we get 
\[
\frac{f(\xi_{v}(-t))+f(\xi_{v}(t))-2f(x)}{t^{2}}\le\frac{1-g_{w}(v,w)}{t}d^{p-1}(x,y)+t^{-2}\mathcal{O}(t^{2})\to-\infty\;\mbox{ as }t\to0.
\]

Next we will treat the case that $y$ is the first conjugate point
of $x$ along a unique minimal geodesic $\eta:[0,1]\to M$ from $x$
to $y$. By definition, let $J$ be a Jacobi field along $\eta$ vanishing
only at $x$ and $y$. For $v=D_{\dot{\eta}}^{\dot{\eta}}J(0)\in T_{x}M\backslash\{0\}$
let $V_{1}$ be the parallel vector field along $\eta$ (i.e. $D_{\dot{\eta}}^{\dot{\eta}}V_{1}\equiv0$)
such that $V_{1}(0)=v$. Furthermore, define for $t\in[0,1]$ the
vector field $V(t):=(1-t)V_{1}(t)$ and $J_{\epsilon}=J+\epsilon V$
for small $\epsilon>0$. Note that $J_{\epsilon}(0)=\epsilon v$ and
$J_{\epsilon}(1)=0$, and since $g_{\dot{\eta}(0)}(v,v)>0$ also $J_{\epsilon}\ne0$
on $[0,1)$ for sufficiently small $\epsilon>0$.

We define a variation $\sigma:[0,1]\times[-1,1]\to M$ by $\sigma(t,s)=\sigma_{s}(t):=\xi_{J_{\epsilon}(t)}(s)$.
Because $J_{\epsilon}\ne0$ on $[0,1)$ this variation is $C^{\infty}$
on $(0,1)\times(-1,1)$. According to the second variation formula
we get (see \cite[Proof of 3.1]{Ohta2009}) 
\[
\frac{\partial^{2}\mathcal{L}(\sigma_{s})}{\partial s^{2}}\bigg|_{s=0}=I(J_{\epsilon},J_{\epsilon})-\frac{g_{\dot{\eta}}(D_{J_{\epsilon}}^{\dot{\eta}}J_{\epsilon},\dot{\eta})}{d(x,y)}-\frac{1}{d(x,y)}\int\left\{ \frac{\partial F(\partial_{t}\sigma)}{\partial s}(t)\right\} ^{2}dt
\]
where $\mathcal{L}$ is the length functional 
\[
\mathcal{L}(\sigma_{s})=\operatorname{length}(\sigma_{s}(\cdot)).
\]
By definition of tangent curvature $\mathcal{T}$ (see \cite{Ohta2009}),
we have 
\[
\mathcal{T}_{\dot{\eta}(0)}(v)=g_{\dot{\eta}}(D_{v}^{v}v-D_{v}^{\dot{\eta}}v,\dot{\eta})=\epsilon^{-2}g_{\dot{\eta}(0)}(D_{J_{\epsilon}}^{J_{\epsilon}}J_{\epsilon}-D_{J_{\epsilon}}^{\dot{\eta}}J_{\epsilon},\dot{\eta})=-\epsilon^{-2}g_{\dot{\eta}(0)}(D_{J_{\epsilon}}^{\dot{\eta}}J_{\epsilon},\dot{\eta})
\]
where the last equality follows from the fact that $\sigma_{0}=\xi_{J_{\epsilon}(0)}$
is a geodesic. Combining these we get 
\begin{eqnarray*}
\frac{\partial^{2}\mathcal{L}(\sigma_{s})}{\partial s^{2}}\bigg|_{s=0} & \le & I(J,J)+2\epsilon I(J,V)+\epsilon^{2}I(V,V)+\epsilon^{2}\mathcal{T}_{\dot{\eta}(0)}(v)/d(x,y)\\
 & = & \left\{ \left[g_{\dot{\eta}}(D_{\dot{\eta}}^{\dot{\eta}}J,J)\right]_{t=0}^{1}+2\epsilon\left[g_{\dot{\eta}}(D_{\dot{\eta}}^{\dot{\eta}}J,V)\right]_{t=0}^{1}+\epsilon^{2}\mathcal{T}_{\dot{\eta}(0)}(v)\right\} /d(x,y)+\epsilon^{2}I(V,V)\\
 & = & \left\{ -2\epsilon g_{\dot{\eta}(0)}(v,v)+\epsilon^{2}\mathcal{T}_{\dot{\eta}(0)}(v)\right\} /d(x,y)+\epsilon^{2}I(V,V).
\end{eqnarray*}
Furthermore, note by the first variations formula and the fact that
$\sigma_{0}$ is a geodesic
\[
\frac{\partial\mathcal{L}(\sigma_{s})}{\partial s}\bigg|_{s=0}=\left[g_{\dot{\eta}}(J_{\epsilon},\dot{\eta})\right]_{t=0}^{1}=\left[\epsilon g_{\dot{\eta}}(V,\dot{\eta})\right]_{t=0}^{1}=-\epsilon g_{\dot{\eta}(0)}(v,\dot{\eta}(0))\ge-\epsilon F(v).
\]
So that we get 
\begin{eqnarray*}
\lim_{s\to0}\frac{\mathcal{L}(\sigma_{s})^{p}+\mathcal{L}(\sigma_{-s})^{p}-2\mathcal{L}(\sigma_{0})^{p}}{s^{2}} & = & p\mathcal{L}^{p-2}(\sigma_{0})\bigg[\mathcal{L}(\sigma_{0})\frac{\partial^{2}}{\partial s^{2}}L(\sigma_{s})\bigg|_{s=0}\\
 &  & \qquad\qquad+(p-1)\left(\frac{\partial\mathcal{L}(\sigma_{s})}{\partial s}\bigg|_{s=0}\right)^{2}\bigg]\\
 & \le & pd^{p-2}(x,y)\bigg(-2\epsilon g_{\dot{\eta}}(v,v)\\
 &  & \hspace{1em}+\epsilon^{2}\left\{ \mathcal{T}_{\dot{\eta}(0)}(v)+d(x,y)I(V,V)+(p-1)F(v)^{2}\right\} \bigg).
\end{eqnarray*}
Using the fact that $f(\xi_{v}(\epsilon s))\le\mathcal{L}(\sigma_{s})^{p}/p$
we obtain
\begin{eqnarray*}
\liminf_{s\to0}\frac{f(\xi_{v}(\epsilon s))+f(\xi_{v}(-\epsilon s))-2f(x)}{\epsilon^{2}s^{2}} & \le & \liminf_{s\to0}\frac{\mathcal{L}(\sigma_{s})^{p}+\mathcal{L}(\sigma_{-s})^{p}-2\mathcal{L}(\sigma_{0})^{p}}{p\epsilon^{2}s^{2}}\\
 & \le & d^{p-2}(x,y)\bigg(-2\epsilon^{-1}g_{\dot{\eta}}(v,v)\\
 &  & \qquad+\mathcal{T}(v)+d(x,y)I(V,V)+(p-1)F(v)^{2}\bigg).
\end{eqnarray*}
Letting $\epsilon$ tend to zero completes the proof.
\end{proof}

\subsection*{The Brenier-McCann-Ohta solution}

The first step to prove the interpolation inequality is showing the
existence of a transport map. This was first done by Brenier \cite{Brenier1991}
in the Euclidean setting and later by McCann \cite{McCann2001} for
Riemannian manifolds and any cost function $c_{L}$. Later Ohta proved
it for Finsler manifolds for the cost function $c_{2}$. The proof
easily adapts to any $p\in(1,\infty)$. 
\begin{lem}
\label{lem:brenier}Let $\phi:M\to\mathbb{R}$ be a $c_{p}$-concave
function. If $\phi$ is differentiable at $x$ then $\partial^{c_{p}}\phi(x)=\{exp_{x}(\nabla^{q}(-\phi)(x))\}$.
Moreover, the curve $\eta(t):=exp_{x}(t\nabla^{q}(-\phi)(x))$ is
a unique minimal geodesic from $x$ to $exp_{x}(\nabla^{q}(-\phi)(x))$.\end{lem}
\begin{proof}
Let $y\in\partial^{c_{p}}\phi(x)$ be arbitrary and define $f(z):=c_{p}(z,y)=d^{p}(z,y)/p$.
By definition of $\partial^{c_{p}}\phi(x)$ we have for any $v\in T_{x}M$
\[
f(exp_{x}v)\ge\phi^{c_{p}}(y)+\phi(exp_{x}v)=f(x)-\phi(x)+\phi(exp_{x}v)=f(x)+d\phi_{x}(v)+o(F(v)).
\]

Now let $\eta:[0,d(x,y)]\to M$ be a minimal unit speed geodesic from
$x$ to $y$. Given $\epsilon>0$, set $y_{\epsilon}=\eta(d(x,y)-\epsilon)$
and note that $\eta|_{[0,d(x,y)-\epsilon]}$ does not cross the cut
locus of $x$. By the first variation formula, we have
\begin{eqnarray*}
f(exp_{x}v)-f(x) & \le & \frac{1}{p}\left\{ \left(d(exp_{x}v,y_{\epsilon})+\epsilon\right)^{p}-\left(d(x,y_{\epsilon})+\epsilon\right)^{p}\right\} \\
 & = & -\left(d(x,y_{\epsilon})+\epsilon\right)^{p-1}g_{\dot{\eta}(0)}(v,\dot{\eta}(0))+o(F(v)).\\
 & = & -d^{p-1}(x,y)\mathcal{L}_{x}^{-1}(\dot{\eta}(0))(v)+o(F(v)).
\end{eqnarray*}
Therefore, $d\phi_{x}(v)\le-d^{p-1}(x,y)\mathcal{L}_{x}^{-1}(\dot{\eta}(0))(v)$
for all $v\in T_{x}M$ and thus $\nabla(-\phi)=d^{p-1}(x,y)\cdot\dot{\eta}(0)$.,
i.e. $\nabla^{q}(-\phi)=d(x,y)\cdot\dot{\eta}(0)$. In addition, note
that $\eta(t)=exp_{x}(t\nabla^{q}(-\phi)(x))$, which is uniquely
defined.
\end{proof}
Let $\operatorname{Lip}_{c_{p}}(X,Y)$ be the set of pairs of Lipschitz
function tuples $\phi:X\to\mathbb{R}$ and $\psi:Y\to\mathbb{R}$
such that 
\[
\phi(x)+\psi(y)\le c_{p}(x,y).
\]

\begin{lem}
Let $\mu_{0}$ and $\mu_{1}$ be two probability measures on $M$.
Then there exists a unique (up to constant) $c_{p}$-concave function
$\phi$ that solves the Monge-Kantorovich problem with cost function
$c_{p}$. Moreover, if $\mu_{0}$ is absolutely continuous, then the
vector field $\nabla^{q}(-\phi)$ is unique among such minimizers.\end{lem}
\begin{proof}
Note that if $(\phi,\psi)\in\operatorname{Lip}_{c_{p}}(X,Y)$ then
$(\phi^{c_{p}\bar{c}_{p}},\phi^{c_{p}})\in\operatorname{Lip}_{c_{p}}(X,Y)$
and $\phi^{c_{p}}\ge\psi$ and $\phi^{c_{p}\bar{c}_{p}}\ge\phi$. 

Now fix some $x_{0}\in X$ and let $\{(\phi_{n},\psi_{n})\}_{n\in\mathbb{N}}\subset\operatorname{Lip}_{c_{p}}(X,Y)$
be a maximizing sequence of the Kantrovich problem. By the remark
just stated, it is easy to see that also $(\hat{\phi}_{n},\hat{\psi}_{n})=(\phi_{n}^{c_{p}\bar{c}_{p}}-\phi_{n}^{c_{p}\bar{c}_{p}}(x_{0}),\phi_{n}^{c_{p}}-\phi_{n}^{c_{p}\bar{c}_{p}}(x_{0}))$
is maximizing and in addition $\phi_{n}^{c_{p}}$ is $c_{p}$-concave.
Since the sequence has uniform bound on the Lipschitz constant and
$\hat{\phi}_{n}(x_{0})=0$, the sequence is precompact and thus we
can assume w.l.o.g. that $(\hat{\phi}_{n})_{n\in\mathbb{N}}$ converges
to a Lipschitz function $\phi:X\to\mathbb{R}$. By similar arguments,
we can also assume that $(\hat{\psi}_{n})_{n\in\mathbb{N}}$ converges
to a function $\psi:Y\to\mathbb{R}$. In addition, note that $\phi^{c_{p}}=\psi$
and that because $ $each $\hat{\phi}_{i}$ is $c_{p}$-concave also
$\phi$ is, in particular, a solution of the Monge-Kantorovich problem
exists and each solution is a pair $(\phi,\phi^{c_{p}})\in\operatorname{Lip}_{c_{p}}(X,Y)$. 

It remains to show that this solution is unique: Let $(\phi_{1},\psi_{1}),(\phi_{2},\psi_{2})\in\operatorname{Lip}_{c_{p}}(X,Y)$
be two solutions of the problem. Now setting $\phi=(\phi_{1}+\phi_{2})/2$,
we see that $\phi^{c_{p}}\ge(\phi_{1}^{c_{p}}+\phi_{2}^{c_{p}})/2$
and thus $(\phi,\phi^{c_{p}})\in\operatorname{Lip}_{c_{p}}(X,Y)$
and hence, by maximality, $ $ $\phi^{c_{p}}=(\phi_{1}^{c_{p}}+\phi_{2}^{c_{p}})/2$
and $\phi$ is $c_{p}$-concave.

Now if $y\in\partial^{c_{p}}\phi(x)$ then $y\in\partial^{c_{p}}\phi_{1}(x)\cap\partial^{c_{p}}\phi_{2}(x)$.
Thus, using Lemma \ref{lem:brenier} above and the absolute continuity
of $\mu_{0}$ we see that 
\[
\nabla^{q}\phi(x)=\nabla^{q}\phi_{i}(x)\qquad\mu_{0}\mbox{-almost every }x\in X.
\]
\end{proof}
\begin{thm}
Let $\mu_{0}$ and $\mu_{1}$ be two probability measure on $M$ and
assume $\mu_{0}$ is absolutely continuous with respect to $\mu$.
Then there is a $c_{p}$-concave function $\phi$ such that $\pi=(\operatorname{Id}\times\mathcal{F})_{*}\mu_{0}$
is the unique optimal coupling of $(\mu_{0},\mu_{1})$, where $\mathcal{F}(x)=exp_{x}(\nabla^{q}(-\phi))$.
Moreover, $\mathcal{F}$ is the unique optimal transport map from
$\mu_{0}$ to $\mu_{1}$.\end{thm}
\begin{rem*}
The proof follows line to line from \cite[Theorem 4.10]{Ohta2009}.
For convenience, we include the whole proof.\end{rem*}
\begin{proof}
Let $\phi$ be given by the Lemma above. Define $ $$\mathcal{F}(x)=exp_{x}(\nabla^{q}(-\phi))$
for all points where $\phi$ is differentiable. Since $\mu_{0}$ is
absolutely continuous, $\mathcal{F}$ is well-defined and continuous
on some $\Omega$ with $\mu_{0}(\Omega)=1$, in particular it is measurable.

Now let $h$ be a continuous function and put $\psi_{\epsilon}=\phi^{c_{p}}+\epsilon h$
for $\epsilon\in\mathbb{R}$ close to $0$. Let $x\in M$ be arbitrary,
then we can find $y_{\epsilon}\in M$ such that 
\[
(\psi_{\epsilon})^{\bar{c}_{p}}(x)=c_{p}(x,y_{\epsilon})-\psi_{\epsilon}(y_{\epsilon}).
\]
Furthermore, whenever $\phi$ is differentiable at $x$ then $y_{\epsilon}$
converges to $y_{0}=\mathcal{F}(x)$. In addition, we have
\begin{eqnarray*}
\phi(x)-\epsilon h(y_{\epsilon}) & \le & c_{p}(x,y_{\epsilon})-\phi^{c_{p}}-\epsilon h(y_{\epsilon})=(\psi_{\epsilon})^{\bar{c}_{p}}(x)\\
 & \le & c_{p}(x,\mathcal{F}(x))-\psi_{\epsilon}(\mathcal{F}(x))=\phi(x)-\epsilon h(\mathcal{F}(x))
\end{eqnarray*}
and thus $(\psi_{\epsilon})^{\bar{c}_{p}}(x)=\phi(x)-\epsilon h(\mathcal{F}(x))+o(|\epsilon|)$
and $|o(|\epsilon|)|\le2\epsilon\|h\|_{\infty}$.

Now set $J(\epsilon)=\int(\psi_{\epsilon})^{\bar{c}_{p}}d\mu+\int\psi_{\epsilon}d\nu$
and by maximality of $(\phi,\phi^{c_{p}})$ we have 
\[
0=\lim_{\epsilon\to0}\frac{J(\epsilon)-J(0)}{\epsilon}=-\int hd\mathcal{F}_{*}\mu_{0}+\int hd\mu_{1}
\]
and hence $\mathcal{F}_{*}\mu_{0}=\mu_{1}$.

Obviously we have for $\pi_{\phi}:=(\operatorname{Id}\times\mathcal{F})_{*}\mu_{0}$
that $c_{p}(x,y)=\phi(x)+\phi^{c_{p}}(y)$ holds $\pi_{\phi}$-almost
everywhere and thus $\int c_{p}d\pi_{\phi}=\int\phi d\mu_{0}+\int\phi^{c_{p}}d\mu_{1}$,
which implies that $\pi_{\phi}$ is optimal. Conversely, if $\pi$
is an optimal coupling of $(\mu_{0},\mu_{1})$ then $c_{p}(x,y)=\phi(x)+\phi^{c_{p}}(y)$
holds $\pi$-almost everywhere, therefore $\pi\left(\bigcup_{x\in M}(x,\mathcal{F}(x))\right)=1$
which implies $\pi=\pi_{\phi}$.\end{proof}
\begin{cor}
If $\mu_{0}$ is absolutely continuous and $\phi$ is $c_{p}$-concave,
then the map $\mathcal{F}(x):=exp_{x}(\nabla^{q}(-\phi))$ is the
unique optimal transport map from $\mu_{0}$ to $\mathcal{F}_{*}\mu_{0}$.
\end{cor}
Furthermore, we will see in Lemma \ref{lem:abs-interp} below that
the interpolation measures are absolutely continuous if $\mu_{0}$
and $\mathcal{F}_{*}\mu_{0}$ are.
\begin{cor}
If $\phi$ is $c_{p}$-concave and $\mu_{0}$ is absolutely continuous,
then the map $\mathcal{F}_{t}(x):=exp_{x}(\nabla^{q}(-t^{p-1}\phi))$
is the unique optimal transport map from $\mu_{0}$ to $\mu_{t}=(\mathcal{F}_{t})_{*}\mu_{0}$
and $t\mapsto\mu_{t}$ is a constant geodesic from $\mu_{0}$ to $\mu_{1}$
in $\mathcal{P}_{p}(M)$.\end{cor}
\begin{proof}
We only need to show that 
\[
w_{p}(\mu_{s},\mu_{t})\le|s-t|w_{p}(\mu_{0},\mu_{1}).
\]
Let $\pi$ be the plan on $\operatorname{Geo}(M)=\{\gamma:[0,1]\to M\,|\,\gamma\mbox{ is a geodesic in }M\}$
$ $give by $\mu_{0}$, the map $\mathcal{F}$ and the unique geodesic
connecting $\mu$-almost every $x\in M$ to a point $\mathcal{F}_{1}(x)$
(see e.g. \cite[Theorem 4.2]{Lisini2006} and \cite[Chapter 7]{Villani2009}),
in particular, $\mu_{t}=(\mathcal{F}_{t})_{*}\mu_{0}$. Since $(e_{s},e_{t})_{*}\pi$
is a plan between $\mu_{s}$ and $\mu_{t}$ for $s,t\in[0,1]$, we
have 
\begin{eqnarray*}
w_{p}(\mu_{s},\mu_{t}) & \le & \left(\int d^{p}(\gamma_{s},\gamma_{t})d\pi(\gamma)\right)^{1/p}\\
 & = & |s-t|\left(\int d^{p}(\gamma_{0},\gamma_{1})d\pi(\gamma)\right)^{1/p}=|s-t|w_{p}^{p}(\mu_{0},\mu_{1}).
\end{eqnarray*}

\end{proof}

\subsection*{Almost Semiconcavity of $c_{p}$-concave functions}

This section will be one of the main ingredients to show Theorem \ref{thm:interp}.
In \cite{Ohta2008} Ohta showed that every $c_{p}$-concave function
is almost everywhere second order differentiable. He proved this by
showing the the square of the distance function $d_{y}^{2}=d^{2}(\cdot,y)$
for fixed $y\in M$ is semiconcave \cite[Corollary 4.4]{Ohta2008}
and extending the Alexandrov-Bangert Theorem \cite[Theorem 6.6]{Ohta2008}
to Finsler manifolds.
\begin{thm}
[Alexandrov-Bangert {\cite[14.1]{Villani2009}} {\cite[6.6]{Ohta2008}}]
Let $M$ be a smooth symmetric Finsler manifold, then every function
$\phi:M\to\mathbb{R}$ which is locally semiconvex in some open subset
$U$ of $M$ is almost everywhere second order differentiable in $U$.
\end{thm}
Even though for general $1<p<\infty$ we cannot show that every $c_{p}$-concave
function is semiconcave, we show that almost all points $x$ of differentiability
of a $c_{p}$-concave function $\phi$ with $d\phi_{x}\ne0$ are second
order differentiable. 

Instead of following the arguments in \cite{Ohta2008} (which is done
in the author's thesis), we give a new, shorter proof using star-shapedness
of the $c_{p}$-concave functions (Lemma \ref{lem:star-shaped}). 

For the proof, note the following: If the Finsler metric $F$ is $C^{\infty}$
then the function $d_{y}^{p}(z)=d(z,y)^{p}$ is $C^{\infty}$ in $U_{y}\backslash\{y\}$
for some sufficiently small neighborhood $U_{y}$ of $y$. This follows
from smoothness of the exponential map $exp_{y}$ in $V\backslash\{0\}\subset T_{y}M$
for some neighborhood $V$ of $0_{x}\in T_{x}M$ , see \cite[p. 315]{Shen1997}.
In particular, for $x\in U_{y}\backslash\{y\}$ we can choose a small
neighborhood $U_{1}\subset U$ of $x$ and an open set $V_{1}\subset U$
disjoint from $U_{1}$ such that $\{d_{y'}^{p}:U_{1}\to\mathbb{R}\}_{y'\in V_{1}}$
are uniformly bounded in $C^{2}$, in particular the functions are
uniformly semiconcave. In addition, note that since $M$ is compact,
$U_{y}$ can be chosen to contain a ball $B_{r_{min}}(y)$ where $r_{min}>0$
can be chosen locally uniformly on $M$, in case $M$ is compact even
uniformly.
\begin{rem*}
Note that we only need a $C^{2}$-bounds so that $F$ only needs to
be locally $C^{2}$. Also note that the same argument holds for any
convex function of the distance which is smooth enough away from the
origin. Furthermore, the theorem below holds for any $c_{L}$-concave
function if Lemma \ref{lem:star-shaped-Orlicz} is used instead of
Lemma \ref{lem:star-shaped}.\end{rem*}
\begin{thm}
\label{thm:semiconc}Let $\phi$ be a $c_{p}$-concave function. Let
$\Omega_{id}$ be the the points $x\in M$ where $\phi$ is differentiable
and $d\phi_{x}=0$, or equivalently $\partial^{c_{p}}\phi(x)=\{x\}$.
Then $\phi$ is locally semiconcave on an open subset $U\subset M\backslash\Omega_{id}$
of full measure (relative to $M\backslash\Omega_{id}$). In particular,
it is second order differentiable almost everywhere in $U$.\end{thm}
\begin{proof}
Since $\partial^{c_{p}}\phi(x)$ is non-empty for every $x\in M$
and semicontinuous in $x$, we have the following: if $\phi$ is differentiable
in $x$ with $d\phi_{x}\ne0$ then $x\in\operatorname{int}(M\backslash\Omega_{id})$.
Thus it suffices to show that each such points has a neighborhood
$U_{1}$ in which $\phi$ is uniformly semiconcave.

So fix such an $x$ with $d\phi(x)\ne0$ and note that $\phi$ is
semiconcave on $U_{1}$ iff $\lambda\phi$ is for an arbitrary $\lambda>0$.
Furthermore, by Lemma \ref{lem:star-shaped} we know that $\phi_{s}=s^{p-1}\phi$
is $c_{p}$-concave for any $s\in[0,1]$.

Since $d\phi(x)\ne0$, there is a unique $y\in M$ with $\partial^{c_{p}}\phi(x)=\{y\}$
and a unique geodesic $\eta:[0,1]\to M$ between $x$ and $y$ (see
Lemma \ref{lem:brenier}). Also note that $\phi_{s}$ is differentiable
at $x$ and 
\[
\partial^{c_{p}}\phi_{s}(x)=\{\eta(s)\}.
\]

Let $s\in[0,1]$ be such that $d(x,\eta(s))<\frac{r_{min}}{2}$. Because
$x\ne\eta(s)$ and $z\mapsto\partial^{c_{p}}\phi_{s}(z)$ is continuous
and single-valued at $x$, we can find a neighborhood $V_{1}\subset U$
of $y$ such that $(\partial^{c_{p}}\phi_{s})^{-1}(V_{1})\cap U$
contains some ball $B_{2\epsilon}(x)$ disjoint from $V_{1}$. Thus
the functions $\{d_{y}^{p}:B_{2\epsilon}(x)\to\mathbb{R}\}_{y\in V_{1}}$
are semiconcave with constant $C$.

Now let $\gamma:[0,1]\to B_{2\epsilon}(x)$ be a minimal geodesic
and set $x_{t}=\gamma(t)$. Choose $y_{t}\in\partial^{c_{p}}\phi_{s}(x_{t})\cap V_{1}$.
By the definition of $c_{p}$-concavity we have 
\begin{eqnarray*}
\phi_{s}(x_{0}) & \le & \phi_{s}(x_{t})+\frac{1}{p}d^{p}(x_{0},y_{t})-\frac{1}{p}d^{p}(x_{t},y_{t})\\
\phi_{s}(x_{1}) & \le & \phi_{s}(x_{t})+\frac{1}{p}d^{p}(x_{1},y_{t})-\frac{1}{p}d^{p}(x_{t},y_{t}).
\end{eqnarray*}
Further, because $y_{t}\in V_{1}$ we also have 
\[
d^{p}(x_{t},y_{t})\ge(1-t)d^{p}(x_{0},y_{t})+td^{p}(x_{1},y_{t})-C(1-t)td^{2}(x_{0},x_{1}).
\]

Therefore, taking the $(1-t),t$ convex combination of the first two
inequality we obtain
\begin{eqnarray*}
\phi_{s}(x_{t}) & \ge & (1-t)\phi_{s}(x_{0})+t\phi_{s}(x_{1})+\frac{d^{p}(x_{t},y_{t})}{p}-(1-t)\frac{d^{p}(x_{0},y_{t})}{p}-t\frac{d^{p}(x_{1},y_{t})}{p}\\
 & \ge & (1-t)\phi_{s}(x_{0})+t\phi_{s}(x_{1})-\frac{C}{p}(1-t)td^{2}(x_{0},x_{1}).
\end{eqnarray*}

\end{proof}

\subsection*{Volume distortion}

In order to describe the interpolation density, one needs to have
a proper definition of determinant of the differential of the transport
map. We follow Ohta's idea to describe the volume distortion as a
proper replacement. 

If $Q:T_{x}M\to T_{y}M$ we define $\mathbf{D}[Q]=\mu_{y}(Q(A))/\mu_{x}(A)$
where $\mu_{x}$ and $\mu_{y}$ are the measure on $T_{x}M$ induced
by $\mu$ and $A$ is a nonempty, open and bounded Borel subset of
$T_{x}M$. Note that $\mathbf{D}$ satisfies the classical Brunn-Minkowski
inequality, i.e. if $Q_{0},Q_{1}:T_{x}M\to T_{y}M$ then for $Q_{t}=(1-t)Q_{0}+tQ_{1}$
\[
\mathbf{D}[Q_{t}]\ge(1-t)\mathbf{D}[Q_{0}]+t\mathbf{D}[Q_{1}].
\]

Now if $B_{r}^{+}(x)$ denotes the forward ball of radius $r$ around
$x$, i.e. all $y\in M$ with $d(x,y)<r$ and $B_{r}^{-}(x)$ the
backward ball around $x$, i.e. all $y\in M$ with $d(y,x)<r$. then
define the volume distortion coefficients as follows
\[
\mathfrak{v}_{t}^{<}(x,y)=\lim_{r\to0}\frac{\mu(Z_{t}(x,B_{r}^{+}(y))}{\mu(B_{tr}^{+}(y))}\mbox{ and }\mathfrak{v}_{t}^{>}(x,y)=\lim_{r\to0}\frac{\mu(Z_{t}(B_{r}^{-}(x),y)}{\mu(B_{(1-t)r}^{-}(x))}.
\]

\begin{rem*}
In the symmetric setting one has $\mathfrak{v}_{t}^{>}(x,y)=\mathfrak{v}_{1-t}^{<}(y,x)$.\end{rem*}
\begin{thm}
[Volume distortion for $d^2$ {\cite[3.2]{Ohta2009}}]\label{thm:Volume-d2}For
point $x\ne y\in M$ with $y\notin\operatorname{Cut}(x)$, let $\eta:[0,1]\to M$
be the unique minimal geodesic from $x$ to $y$. For $t\in(0,1]$
define $g_{t}(z)=-d(z,\eta(t))^{2}/2$. 

Then we have 
\begin{eqnarray*}
\mathfrak{v}_{t}^{<}(x,y) & = & \mathbf{D}\left[d(exp_{x})_{\nabla g_{t}(x)}\circ[d(exp_{x})_{\nabla g_{1}(x)}]^{-1}\right]\\
\mathfrak{v}_{t}^{>}(x,y) & = & (1-t)^{-n}\mathbf{D}\left[d(exp_{x}\circ\mathcal{L}_{x})_{d(g_{1})_{x}}\circ[d\left(d(tg_{1})\right)_{x}-d\left(dg_{t}\right)_{x}]\right].
\end{eqnarray*}

\end{thm}

\begin{thm}
[Volume distortion for $d^p$] Let $x\ne y$ with $y\notin\operatorname{Cut}(x)$
and $\eta$ be as above. For $t\in(0,1]$ define $f_{t}(z)=-d(z,\eta(t))^{p}/p$. 

Then we have 
\begin{eqnarray*}
\mathfrak{v}_{t}^{<}(x,y) & = & \mathbf{D}\left[d(exp_{x})_{\nabla^{q}f_{t}(x)}\circ[d(exp_{x})_{\nabla^{q}f_{1}(x)}]^{-1}\right]\\
\mathfrak{v}_{t}^{>}(x,y) & = & (1-t)^{-n}\mathbf{D}\left[d(exp_{x}\circ\mathcal{K}_{x})_{d(t^{p-1}f_{1})_{x}}\circ[d\left(d(t^{p-1}f_{1})\right)_{x}-d\left(df_{t}\right)_{x}]\right].
\end{eqnarray*}
\end{thm}
\begin{proof}
The first equation follows from the fact that 
\[
\nabla^{q}f_{t}(x)=\nabla\left(-d(x,\eta(t))^{2}/2\right).
\]
For the second part note that 
\[
\mathcal{L}_{z}(d(tg_{1})_{z})=\mathcal{K}_{z}(d(t^{p-1}f_{1})_{z})
\]
 and thus 
\begin{eqnarray*}
\mathfrak{v}_{t}^{>}(x,y) & = & (1-t)^{-n}\mathbf{D}\left[d\left(exp\circ\mathcal{L}\circ(d(tg_{1})_{z})\right)\right]\\
 & = & (1-t)^{-n}\mathbf{D}\left[d(exp\circ\mathcal{K})_{d(t^{p-1}f_{1})_{x}}\circ d\left(d(t^{p-1}f_{1})\right)_{x}\right].
\end{eqnarray*}
Similar to \cite[Proof of 3.2]{Ohta2009} since $d(f_{t})_{x}=d(t^{p-1}f_{1})_{x}$
it suffices to show that 
\[
d(exp_{x}\circ\mathcal{K}_{x})_{d(f_{t})_{x}}\circ d(df_{t})_{x}=0.
\]
Note that 
\[
\mathcal{L}_{z}(d(g_{t})_{z})=\mathcal{K}_{z}(d(f_{t})_{z})
\]
and thus 
\begin{eqnarray*}
L(z)=exp_{z}\circ\mathcal{K}_{z}(d(f_{t})_{z}) & = & exp_{z}\circ\mathcal{L}_{z}(d(g_{t})_{z})\\
 & = & \eta(t).
\end{eqnarray*}
Which immediately implies $dL=0$.
\end{proof}

\subsection*{Interpolation inequality in the $p$-Wasserstein space}

The following proposition is a generalization of \cite[5.1]{Ohta2009}
to the case $p\ne2$. The proof is up to some changes in notation
and changes of powers the same as Ohta's.
\begin{prop}
\label{prop:eqns}Let $\phi:M\to\mathbb{R}$ be a $c_{p}$-concave
function and define $\mathcal{F}(z)=exp_{z}(\nabla^{q}(-\phi)(z))$
at all point of differentiability of $\phi$. Fix some $x\in M$ such
that $\phi$ is second order differentiable at $x$ and $d\phi_{x}\ne0.$
Then the following holds:
\begin{enumerate}
\item $y=\mathcal{F}(x)$ is not a cut point of $x$.
\item The function $h(z)=c_{p}(z,y)-\phi(z)$ satisfies $dh_{x}=0$ and
\[
\left(\frac{\partial^{2}h}{\partial x^{i}\partial x^{j}}(x)\right)\ge0
\]
in any local coordinate system $(x^{i})_{i=1}^{n}$ around $x$.
\item Define $f_{y}(z):=-c_{p}(z,y)$ and
\[
d\mathcal{F}_{x}:=d(exp_{x}\circ\mathcal{K}_{x})_{d(-\phi)_{x}}\circ\left[d(d(-\phi))_{x}-d(df_{y})_{x}\right]:T_{x}M\to T_{y}M
\]
where the vertical part of $T_{d(-\phi)_{x}}(T^{*}M)$ and $T_{d(-\phi)_{x}}(T^{*}M)$
are identified. Then the following holds for all $v\in T_{x}M$
\[
\sup\left\{ \left|u-d\mathcal{F}_{x}(v)\right|\,|\, exp_{y}u\in\partial^{c_{p}}\phi(\exp_{x}y),|u|=d(y,\exp_{y}u)\right\} =o(|v|).
\]

\end{enumerate}
\end{prop}
\begin{proof}
As $\phi$ is differentiable at $x$ we have $\partial^{c_{p}}\phi(x)=\{y\}$
and hence for any vector $v\in T_{x}M$ with $F(v)=1$ and sufficiently
small $t>0$, we have by $c_{p}$-concavity of $\phi$ 
\[
h(x)=c_{p}(x,y)-\phi(x)=\phi^{c_{p}}(y)\le c_{p}(\xi_{v}(\pm t),y)-\phi(\xi_{v}(\pm t))=h(\xi_{v}(\pm t))
\]
where $\xi_{v}:(-\epsilon,\epsilon)\to M$ is a geodesic with $\dot{\xi}_{v}(0)=v$.
Thus, we have 
\[
\frac{\phi(\xi_{v}(t))+\phi(\xi_{v}(-t))-2\phi(x)}{t^{2}}\le\frac{f_{y}(\xi_{v}(t))+f_{y}(\xi_{v}(-t))-2f_{y}(x)}{t^{2}}.
\]
Since $\phi$ is second order differentiable at $x$ we have 
\[
-\infty<\frac{\partial^{2}(\phi\circ\xi_{v})}{\partial t^{2}}(0)=\limsup_{t\to0^{+}}\frac{f_{y}(\xi_{v}(t))+f_{y}(\xi_{v}(-t))-2f_{y}(x)}{t^{2}}
\]
and hence $y$ is not a cut point of $x$ (Lemma \ref{lem:cut}).

Now the second statement follows immediately from the inequality above
and the fact that $y\notin\operatorname{Cut}(x)\cup\{x\}$ implies
that $f_{y}$ is $C^{\infty}$ at $x$ and $\nabla^{q}f_{y}(x)=\nabla^{q}\phi(x)$,
i.e. $h$ takes its minimum at $x$.

The last part follows from the fact that $dh_{x}=0$ implies $d(f_{y})_{x}=d\phi_{x}$
and thus the difference $d(d(-\phi))_{x}-d(df)_{x}$ makes sense.
Putting $x_{t}=exp_{x}tv$ for some $v\in T_{x}M$ and small $t\ge0$
we can find $u_{t}\in T_{y}M$ such that $y_{t}:=exp_{y}u_{t}\in\partial^{c_{p}}\phi(x_{t})$
and $d(y,y_{t})=F(u_{t})$. In addition, we have 
\[
-\phi(exp_{x_{t}}w)\ge-\phi(x_{t})-f_{y_{t}}(x_{t})+f(exp_{x_{t}}w)=-\phi(x_{t})+d(f_{y_{t}})_{x_{t}}(w)+o(F(w))
\]
for $w\in T_{x_{t}}M$. Differentiating $y_{t}=exp\circ\mathcal{K}(d(f_{y_{t}})_{x_{t}})$
at $t=0$ we get
\[
\frac{\partial y_{t}}{\partial t}\big|_{t=0}=d(exp\circ\mathcal{K})_{d(-\phi)_{x}}\circ d(d(-\phi)_{x})(v).
\]
Moreover, we have $exp\circ\mathcal{K}(d(f_{y})_{x_{t}})\equiv y$
and thus $d(exp\circ\mathcal{K})_{d(f_{y})_{x}}\circ d(df_{y})_{x}(v)=0$.
Therefore, 
\[
\frac{\partial y_{t}}{\partial t}\big|_{t=0}=d(exp\circ\mathcal{K})_{d(-\phi)_{x}}\circ\left[d(d(-\phi)_{x}-d(df_{y})_{x}\right](v)=d\mathcal{F}_{x}(v).
\]
Note that, because $d(d(-\phi)_{x})-d(df_{y})$ contains only vertical
terms (see also \cite[Proof of 5.1]{Ohta2009}) we regard it as living
in $T_{d(-\phi)_{x}}(T_{x}^{*}M)$ and thus replace $d(exp\circ\mathcal{K})_{d(-\phi)_{x}}$
by $d(exp\circ\mathcal{K}_{x})_{d(-\phi)_{x}}$. The last part follows
immediately by noticing that $\phi$ is second order differentiable
and thus $y_{t}=exp_{y}u_{t}$ with $u_{t}=d\mathcal{F}_{x}(tv)+o(t)$
where $o(t)$ can be chosen uniformly in $v$.\end{proof}
\begin{prop}
\label{prop:jac-eq}Let $\mu_{0}$ and $\mu_{1}$ be absolutely continuous
measure with density $f_{0}$ and $f_{1}$ resp. and assume that there
are open set $U_{i}$ with compact closure $X=\bar{U}_{0}$ and $Y=\bar{U}_{1}$
such that $\operatorname{supp}\mu_{i}\subset U_{i}$. Let $\phi$
be the unique $c_{p}$-concave Kantorovich potential and define $\mathcal{F}(z)=exp_{z}(\nabla^{q}(-\phi)(z))$.
Then $\mathcal{F}$ is injective $\mu_{0}$-almost everywhere and
for $\mu_{0}$-almost every $x\in M\backslash\Omega_{id}$ 
\begin{enumerate}
\item The function $h(z)=c_{p}(z,\mathcal{F}(z))-\phi(z)$ satisfies 
\[
\left(\frac{\partial^{2}h}{\partial x^{i}\partial x^{j}}(x)\right)>0
\]
in any local coordinate system $(x^{i})_{i=0}^{n}$ around $x$.
\item In particular, $\mathbf{D}[d\mathcal{F}_{x}]>0$ holds for the map
$d\mathcal{F}_{x}:T_{x}M\to T_{\mathcal{F}(x)}M$ defined as above
and
\[
\lim_{r\to0}\frac{\mu(\partial^{c_{p}}\phi(B_{r}^{+}(x)))}{\mu(B_{r}^{+}(x))}=\mathbf{D}[d\mathcal{F}_{x}]
\]
and 
\[
f(x)=g(\mathcal{F}(x))\mathbf{D}[d\mathcal{F}_{x}].
\]

\end{enumerate}
\end{prop}
\begin{rem*}
defining $d\mathcal{F}_{x}=\operatorname{Id}$ for points $x$ of
differentiability of $\phi$ with $d\phi_{x}=0$, we see that the
second statement above holds $\mu$-a.e.\end{rem*}
\begin{proof}
The proof follows without any change from \cite[Theorem 5.2]{Ohta2009},
see also \cite[Chapter 11]{Villani2009}.\end{proof}
\begin{thm}
\label{thm:interp}Let $\phi:M\to\mathbb{R}$ be a $c_{p}$-concave
function and $x\in M$ such that $\phi$ is second order differentiable
with $d\phi_{x}\ne0$. For $t\in(0,1]$, define $y_{t}:=exp_{x}(\nabla^{q}(-t^{p-1}\phi))$,
$f_{t}(z)=-c_{p}(z,y_{t})$ and $\mathbf{J}_{t}(x)=\mathbf{D}[d(\mathcal{F}_{t})_{x}]$
where 
\[
d(\mathcal{F}_{t})_{x}:=d(exp_{x}\circ\mathcal{K}_{x})_{d(-t^{p-1}\phi)_{x}}\circ\left[d(d(-t^{p-1}\phi))_{x}-d(d(f_{t}))_{x}\right]:T_{x}M\to T_{y_{t}}M.
\]
Then for any $t\in(0,1)$
\[
\mathbf{J}_{t}(x)^{\nicefrac{1}{n}}\ge(1-t)\mathfrak{v}_{t}^{>}(x,y_{1})^{\nicefrac{1}{n}}+t\mathfrak{v}_{t}^{<}(x,y_{1})^{\nicefrac{1}{n}}\mathbf{J}_{1}(x)^{\nicefrac{1}{n}}.
\]
\end{thm}
\begin{rem*}
The proof is based on the proof of \cite[Proposition 5.3]{Ohta2009}.\end{rem*}
\begin{proof}
Note first that 
\begin{eqnarray*}
d(d(-t^{p-1}\phi))_{x}-d(df_{t})_{x} & = & \left\{ d(d(-t^{p-1}\phi))_{x}-d(d(t^{p-1}f_{1}))_{x}\right\} \\
 &  & +\left\{ d(d(t^{p-1}f_{1}))_{x}-d(df_{t})_{x}\right\} 
\end{eqnarray*}
and 
\[
d(f_{t})_{x}=d(-t^{p-1}\phi)_{x}=d(-t^{p-1}f_{1})_{x}.
\]
Now define $\tau_{s}:T^{*}M\to T^{*}M$ as $\tau_{s}(v)=s^{p-1}v$
and note 
\begin{eqnarray*}
 &  & d(exp_{x}\circ\mathcal{K}_{x})_{d(-t^{p-1}\phi)_{x}}\circ\left(d(d(-t^{p-1}\phi))_{x}-d(d(t^{p-1}f_{1}))_{x}\right)\\
 & = & d(exp_{x}\circ\mathcal{K}_{x})_{d(-t^{p-1}\phi)_{x}}\circ d(\tau_{t})_{d(-\phi)_{x}}\circ\left[d(d(-\phi))_{x}-d(d(f_{1}))_{x}\right]\\
 & = & d(exp_{x}\circ\mathcal{K}_{x})_{d(-t^{p-1}\phi)_{x}}\circ d(\tau_{t})_{d(-\phi)_{x}}\circ\left[d(exp_{x}\circ\mathcal{K}_{x})_{d(-\phi)_{x}}\right]^{-1}\circ d(\mathcal{F}_{1})\\
 & = & d(exp_{x})_{\nabla^{q}(-t^{p-1}\phi)_{x}}\circ d(\mathcal{K}_{x}\circ\tau_{t}\circ\mathcal{K}_{x}^{-1})_{\nabla^{q}(-\phi)_{x}}\circ[d(exp_{x})_{\nabla^{q}(-\phi)_{x}}]^{-1}\circ d(\mathcal{F}_{1})\\
 & = & t\cdot d(exp_{x})_{\nabla^{q}(-t^{p-1}\phi)_{x}}\circ[d(exp_{x})_{\nabla^{q}(-\phi)_{x}}]^{-1}\circ d(\mathcal{F}_{1}),
\end{eqnarray*}
because $\mathcal{K}_{x}\circ\tau_{t}\circ\mathcal{K}_{x}^{-1}$ is
linear and for $v\in T_{x}M$
\[
\mathcal{K}_{x}\circ\tau_{t}\circ\mathcal{K}_{x}^{-1}(v)=\mathcal{K}_{x}(t^{p-1}\mathcal{K}_{x}^{-1}(v))=tv,
\]
i.e. $d(\mathcal{K}_{x}\circ\tau_{t}\circ\mathcal{K}_{x}^{-1})_{\nabla^{q}(-\phi)_{x}}=t\cdot\operatorname{Id}$.
Note that we identified $T_{\nabla^{q}(-t^{p-1}\phi)(x)}(T_{x}M)$
with $T_{\nabla^{q}(-\phi)(x)}(T_{x}M)$ to get the last inequality.

Because $\mathbf{D}$ is concave we get 
\begin{eqnarray*}
 &  & \mathbf{J}_{t}(x)^{\nicefrac{1}{n}}=\mathbf{D}[d(\mathcal{F}_{t})_{x}]^{\nicefrac{1}{n}}\\
 & = & \mathbf{D}\bigg[d(exp_{x}\circ\mathcal{K}_{x})_{d(-t^{p-1}\phi)_{x}}\circ\left[d(d(t^{p-1}f_{1}))_{x}-d(df_{t})_{x}\right]\\
 &  & +\, d(exp_{x}\circ\mathcal{K}_{x})_{d(-t^{p-1}\phi)_{x}}\circ\left(d(d(-t^{p-1}\phi))_{x}-d(d(t^{p-1}f_{1}))_{x}\right)\bigg]^{\nicefrac{1}{n}}\\
 & = & \mathbf{D}\bigg[d(exp_{x}\circ\mathcal{K}_{x})_{d(-t^{p-1}\phi)_{x}}\circ\left(d(d(t^{p-1}f_{1}))_{x}-d(df_{t})_{x}\right)\\
 &  & +\, t\cdot d(exp_{x})_{\nabla^{q}(-t^{p-1}\phi)_{x}}\circ[d(exp_{x})_{\nabla^{q}(-\phi)_{x}}]^{-1}\circ d(\mathcal{F}_{1})\bigg]^{\nicefrac{1}{n}}\\
 & \ge & (1-t)\mathbf{D}\bigg[(1-t)^{-1}d(exp_{x}\circ\mathcal{K}_{x})_{d(-t^{p-1}\phi)_{x}}\circ\left(d(d(t^{p-1}f_{1}))_{x}-d(df_{t})_{x}\right)\bigg]^{\nicefrac{1}{n}}\\
 &  & +\, t\mathbf{D}\bigg[d(exp_{x})_{\nabla^{q}(-t^{p-1}\phi)_{x}}\circ[d(exp_{x})_{\nabla^{q}(-\phi)_{x}}]^{-1}\circ d(\mathcal{F}_{1})\bigg]^{\nicefrac{1}{n}}\\
 & = & (1-t)\mathfrak{v}_{t}^{>}(x,y_{1})^{\nicefrac{1}{n}}+t\mathfrak{v}_{t}^{<}(x,y_{1})^{\nicefrac{1}{n}}\mathbf{J}_{1}(x)^{\nicefrac{1}{n}}.
\end{eqnarray*}

\end{proof}
Combing this with Lemma \ref{lem:inj} and Lemma \ref{lem:min-difference}
below we get similar to \cite[6.2]{Ohta2009}:
\begin{lem}
\label{lem:abs-interp}Given two absolutely continuous measures $\mu_{i}=\rho_{i}\mu$
on $M$, let $\phi$ be the unique $c_{p}$-concave optimal Kantorovich
potential. Define $\mathcal{F}_{t}(x):=exp_{x}(\nabla^{q}(-t^{p-1}\phi))$
for $t\in[0,1]$. Then $\mu_{t}=(\mathcal{F}_{t})_{*}\mu_{0}$  is
absolutely continuous for any $t\in[0,1]$.\end{lem}
\begin{proof}
By Lemma \ref{lem:inj} the map $\mathcal{F}_{t}$ is injective $\mu_{0}$-almost
everywhere. Let $\Omega_{id}$ be the points $x\in M$ of differentiability
of $\phi$ with $d\phi_{x}=0$. Then 
\[
\mu_{t}\big|_{\Omega_{id}}=(\mathcal{F}_{t})_{*}(\mu_{0}\big|_{\Omega_{id}})=\mu_{0}\big|_{\Omega_{id}}.
\]
By Theorem \ref{thm:semiconc} the potential $\phi$ is second order
differentiable in a subset $\Omega\subset M\backslash\Omega_{id}$
of full measure. In addition, $\mathbf{D}[d(\mathcal{F}_{1})]>0$
for all $x\in\Omega$ (see Proposition \ref{prop:jac-eq}) and $\mathcal{F}_{t}$
is continuous in $\Omega$ for any $t\in[0,1]$. The map $d(\mathcal{F}_{t})_{x}:T_{x}M\to T_{\mathcal{F}_{t}(x)}M$
defined in Proposition \ref{prop:eqns} as 
\[
d(\mathcal{F}_{t})_{x}:=d(exp_{x}\circ\mathcal{K}_{x})_{d(-t^{p-1}\phi)}\circ\left[d(d(-t^{p-1}\phi))_{x}-d(d(f_{t})_{x}\right]
\]
where $f_{t}(z):=-c_{p}(z,\mathcal{F}_{t}(x))$ for $t\in(0,1]$.
Also note that for $x\in\Omega$
\[
d(d(-t^{p-1}\phi))_{x}-d(df_{t})_{x}=\left\{ d(d(-t^{p-1}\phi))_{x}-d(d(t^{p-1}f_{1}))_{x}\right\} +\left\{ d(d(t^{p-1}f_{1}))_{x}-d(f_{t})_{x}\right\} .
\]
Which implies $\mathbf{D}[d(\mathcal{F}_{t})_{x}]>0$ because $\mathbf{D}[d(\mathcal{F}_{1})_{x})>0$
and the lemma below.

The result then immediately follows by \cite[Claim 5.6]{CMS2001}.\end{proof}
\begin{lem}
\label{lem:min-difference}Let $y\notin\operatorname{Cut}(x)\cup\{x\}$
and $\eta:[0,1]\to M$ be the unique minimal geodesic from $x$ to
$y$. Define 
\[
f_{t}(z)=-c_{p}(z,\eta(t)).
\]
Then the function $h(z)=t^{p-1}f_{1}(z)-f_{t}(z)$ satisfies
\[
\left(\frac{\partial^{2}h}{\partial x^{i}\partial x^{j}}(x)\right)\ge0
\]
in any local coordinate system around $x$.\end{lem}
\begin{proof}
This follows directly from \ref{lem:inf-dist}.
\end{proof}

\section{Abstract curvature condition}

In this section we define a curvature condition à la Lott-Villani-Sturm
(\cite{LV2007,LV2009} and \cite{Sturm2006,Sturm2006a}) with respect
to geodesics in $\mathcal{P}_{p}(M)$ with $p\in(1,\infty)$. For
simplicity, throughout this section, we assume that $M$ is a proper
geodesic space.

\subsection*{Curvature dimension}

In \cite{LV2009} (see also \cite[Part II-III]{Villani2009}) Lott
and Villani introduced the following set of real-valued functions.
\begin{defn}
[$\mathcal{DC}_N$] For $N\in[1,\infty]$ let $\mathcal{DC}_{N}$
all convex functions $U:[0,\infty)\to\mathbb{R}$ with $U(0)=0$ such
that for $N<\infty$ the function 
\[
\psi(\lambda)=\lambda^{N}U(\lambda^{-N})
\]
is convex on $(0,\infty)$. In case $N=\infty$ we require 
\[
\psi(\lambda)=e^{\lambda}U(e^{-\lambda})
\]
to be convex on $(-\infty,\infty)$.\end{defn}
\begin{lem}
[{\cite[Lemma 5.6]{LV2009}}] If $N\le N'$ then $\mathcal{DC}_{N'}\subset\mathcal{DC}_{N}$.\end{lem}
\begin{example}
Note the following examples
\begin{enumerate}
\item if $m=1-\frac{1}{N}$ for $N\in(1,\infty)$ then $U_{m}:x\mapsto\frac{1}{m(m-1)}x^{m}$
is in $\mathcal{DC}_{N}$ 
\item the classical entropy functional $U_{\infty}:x\mapsto x\log x$ is
in $\mathcal{DC}_{\infty}$
\item if $m>1$ then $U_{m}\in\mathcal{DC}_{\infty}$
\end{enumerate}
\end{example}
Given a function $U\in\mathcal{DC}_{N}$ for $N\in[1,\infty]$ we
write \textbf{$U'(\infty)=\lim_{r\to\infty}\frac{U(r)}{r}$}. Given
some reference measure $\mu\in\mathcal{P}(M)$ we define the functional
$\mathcal{U}_{\mu}:\mathcal{P}(M)\to\mathbb{R}\cup\{\infty\}$ by
\[
\mathcal{U}_{\mu}(\nu)=\int U(\rho)d\mu+U'(\infty)\nu_{s}(M)
\]
where $\nu=\rho\mu+\mu_{s}$ the the Lebesgue decomposition of $\nu$
w.r.t. $\mu$. 
\begin{rem*}
In the following we usually fix a metric measure space $(M,d,\mu)$
and drop the subscript $\mu$ from the functional $\mathcal{U}_{\mu}$.
In addition, we use $\mathcal{U}_{m}$, $\mathcal{U}_{\alpha}$ etc.
to denote the functional generated by $U_{m}$, $U_{\alpha}$, etc.
\end{rem*}
In \cite[Section 4]{LV2007} Lott and Villani defined for each $K\in\mathbb{R}$
and $N\in(1,\infty]$ the functions $\beta_{t}:M\times M\to\mathbb{R}\cup\{\infty\}$
and $t\in[0,1]$ as follows
\[
\beta_{t}(x_{1},x_{2})=\begin{cases}
e^{\frac{1}{6}K(1-t^{2})d(x_{0},x_{1})^{2}} & \mbox{ if }N=\infty,\\
\infty & \mbox{ if }N<\infty,K>0\mbox{ and }\alpha>\pi,\\
\left(\frac{\sin(t\alpha)}{t\sin\alpha}\right)^{N-1} & \mbox{ if }N<\infty,K>0\mbox{ and }\alpha\in[0,\pi],\\
1 & \mbox{ if }N<\infty\mbox{ and }K=0,\\
\left(\frac{\sinh(t\alpha)}{t\sinh\alpha}\right)^{N-1} & \mbox{ if }N<\infty\mbox{ and }K<0,
\end{cases}
\]
where 
\[
\alpha=\sqrt{\frac{|K|}{N-1}}d(x_{0},x_{1})
\]
and for $N=1$
\[
\beta_{t}(x_{0},x_{1})=\begin{cases}
\infty & \mbox{ if }K>0,\\
1 & \mbox{ if }K\le0.
\end{cases}
\]

Note that $\beta$ and $\alpha$ depend implicitly on an a priori
chosen $K$ and $N$ which will be suppressed to keep the notation
simple.
\begin{rem*}
In \cite{BS2010} Bacher and Sturm defined a reduced curvature dimension
condition with a different weight function $\sigma_{t}$ instead of
$ $$\beta_{t}$. Because of the localization and tensorization property
this weight function turned out to be powerful (\cite{AGS2011,AGMR2012,Rajala2012,Rajala2011,Garofalo2013,Erbar2013,Gigli2013a,Hua2013}).
Using the inequalities of the proof of Lemma \ref{lem:inj} most of
the things proven in \cite{BS2010} will also hold for localized version
$CD_{p}^{*}(K,N)$.\end{rem*}
\begin{defn}
[(strong) $CD_p(K,N)$] We say $(M,d,\mu)$ satisfies the strong $CD_{p}(K,N)$
condition if the following holds: Given two measure $\mu_{0},\mu_{1}\in\mathcal{P}(M)$
with Lebesgue decomposition $\mu_{i}=\rho_{i}\mu+\mu_{i,s}$. Then
there exists some optimal dynamical transference plan $\Pi\in\mathcal{P}(\operatorname{Geo})$
such that $\mu_{t}=(e_{t})_{*}\Pi$ is a geodesic from $\mu_{0}$
to $\mu_{1}$ in $\mathcal{P}_{p}(M)$ such that for all $U\in\mathcal{DC}_{N}$
and $t\in[0,1]$
\begin{eqnarray*}
\mathcal{U}(\mu_{t}) & \le & (1-t)\int_{M\times M}\beta_{1-t}(x_{0},x_{1})U\left(\frac{\rho_{0}(x_{0})}{\beta_{1-t}(x_{0},x_{1})}\right)d\pi(x_{1}|x_{0})d\mu(x_{0})\\
 &  & +t\int_{M\times M}\beta_{t}(x_{0},x_{1})U\left(\frac{\rho_{1}(x_{i})}{\beta_{t}(x_{0},x_{1})}\right)d\pi(x_{0}|x_{1})d\mu(x_{1})\\
 &  & +U'(\infty)\left((1-t)\mu_{0,s}(M)+t\mu_{1,s}(M)\right),
\end{eqnarray*}
where $\pi=(e_{0},e_{1})_{*}\Pi$ is the optimal transference plan
of $(\mu_{0},\mu_{1})$ w.r.t. $c_{p}$ associated to $\Pi$. Furthermore,
in case $\beta_{s}(x_{0},x_{1})=\infty$ we interpret $ $$\beta_{s}(x_{0},x_{1})U\left(\frac{\rho_{i}(x_{i})}{\beta_{s}(x_{0},x_{1})}\right)$
as $U'(0)\rho_{i}(x_{i})$.

In addition, we say that the very strong $CD_{p}(K,N)$ condition
holds if the inequality holds for all optimal dynamical transference
plans (and thus all geodesics).
\end{defn}
Note that this definition is Lott-Villani's \cite[Defnition 4.7]{LV2007}
by just requiring the geodesic $t\mapsto\mu_{t}$ to be in $\mathcal{P}_{p}(M)$
instead of $\mathcal{P}_{2}(M)$. And, in case both $\mu_{i}$ are
absolutely continuous looks like
\begin{eqnarray*}
\mathcal{U}(\mu_{t}) & \le & (1-t)\int\frac{\beta_{1-t}(x_{0},x_{1})}{\rho_{0}(x_{0})}U\left(\frac{\rho_{0}(x_{0})}{\beta_{1-t}(x_{0},x_{1})}\right)d\pi(x_{0},x_{1})\\
 &  & +t\int\frac{\beta_{t}(x_{0},x_{1})}{\rho_{1}(x_{1})}U\left(\frac{\rho_{1}(x_{1})}{\beta_{t}(x_{0},x_{1})}\right)d\pi(x_{0},x_{1}).
\end{eqnarray*}

An immediate consequence of the curvature condition is the following:
\begin{lem}
Assume $(M,d,\mu)$ satisfies the strong $CD_{p}(K,N)$ and $\mu_{0}$
and $\mu_{1}$ are absolutely continuous, if $t\mapsto\mu_{t}$ satisfies
the functional inequality then $\mu_{t}$ is absolutely continuous.\end{lem}
\begin{proof}
The proof follows from \cite[Theorem 5.52]{LV2009} (see also \cite[Theorem 4.30]{LV2007})
by noting that \cite[Lemma 5.43]{LV2009} does not need $\mu_{i}$
to be in $\mathcal{P}_{2}^{ac}(M)$.
\end{proof}
Furthermore, we will also define a variant of Sturm's curvature condition
\cite{Sturm2006a,Sturm2006}:
\begin{defn}
[(weak) $CD_p(K,N)$]We say $(M,d,\mu)$ satisfies the weak $CD_{p}(K,N)$
condition if for $N\in(1,\infty)$ the above inequality holds only
for the functionals 
\[
U_{N'}(r)=N'r(1-r^{-1/N'})
\]
for any $N'\ge N$. In case $N'=\infty$ the functional $\mathcal{U}_{\infty}$
generated by 
\[
U_{\infty}(r)=r\log r
\]
and has to be $K$-convex along a geodesic $t\mapsto\mu_{t}$ in $\mathcal{P}_{p}(M)$,
i.e. 
\[
\mathcal{U}_{\infty}(\mu_{t})\le(1-t)\mathcal{U}_{\infty}(\mu_{0})+t\mathcal{U}_{\infty}(\mu_{1})-\frac{K}{2}t(1-t)w_{p}^{2}(\mu_{0},\mu_{1}).
\]

\end{defn}
The following follows immediately from Theorem \ref{thm:interp} by
similar statements to the case $CD_{2}(K,N)$ (see e.g. \cite{Ohta2009,Villani2009}).
\begin{cor}
Any $n$-dimensional Finsler manifold with $N$-Ricci curvature bounded
from below by $K$ and $N>n$ satisfies the very strong $CD_{p}(K,N)$
condition for all $p\in(1,\infty)$.\end{cor}
\begin{rem*}
Note%
\footnote{We thank Shin-ichi Ohta for making this remark on an early version
of the paper%
} that in contrast to the case $p=2$ the strong $CD_{p}(K,\infty)$-condition
does not imply the weak one. Indeed the strong $CD_{p}(K,\infty)$-condition
\cite[Lemma 4.14]{LV2007} only gives
\[
\mathcal{U}_{\infty}(\mu_{t})\le(1-t)\mathcal{U}_{\infty}(\mu_{0})+t\mathcal{U}_{\infty}(\mu_{1})-\frac{1}{2}\lambda(U)t(1-t)\int d^{2}(x,y)d\pi_{opt}(x,y),
\]
where $\pi_{opt}$ is the $d^{p}$-optimal coupling between $\mu_{0}$
and $\mu_{1}$. However, using Hölder inequality we get for $p>2$
\[
\int d^{2}(x,y)d\pi_{opt}(x,y)\le\left(\int d^{p}(x,y)d\pi_{opt}(x,y)\right)^{\frac{2}{p}}=C_{p}w_{p}^{2}(\mu_{0},\mu_{1})
\]
and for $p<2$
\[
c_{p}w_{p}^{2}(\mu_{0},\mu_{1})\le\left(\int d^{p}(x,y)d\pi_{opt}(x,y)\right)^{\frac{2}{p}}\le\int d^{2}(x,y)d\pi_{opt}(x,y).
\]
 Thus we get $K'$-convexity for some $K'$ depending only on $p$
and $K$ follows if either $\lambda(U)>0$ and $p<2$ or $\lambda(U)<0$
and $p>2$. 

In the negatively curved case with bounded diameter one can also do
the following: the function 
\[
\lambda\mapsto e^{\lambda}U_{\infty}(e^{-\lambda})
\]
is convex and non-increasing. This means, if we take some $\beta_{t}^{'}(\cdot,\cdot)\le\beta_{t}(\cdot,\cdot)$
then we still have 
\begin{eqnarray*}
\mathcal{U}(\mu_{t}) & \le & (1-t)\int\frac{\beta_{1-t}^{'}(x_{0},x_{1})}{\rho_{0}(x_{0})}U\left(\frac{\rho_{0}(x_{0})}{\beta_{1-t}^{'}(x_{0},x_{1})}\right)d\pi(x_{0},x_{1})\\
 &  & +t\int\frac{\beta_{t}^{'}(x_{0},x_{1})}{\rho_{1}(x_{1})}U\left(\frac{\rho_{1}(x_{1})}{\beta_{t}^{'}(x_{0},x_{1})}\right)d\pi(x_{0},x_{1}),
\end{eqnarray*}
assuming $\mu_{0}$ and $\mu_{1}$ are absolutely continuous. Now
choose for $r<2$ and $D_{r}=\left(\operatorname{diam}M\right)^{2-r}$
then $d^{2}(x,y)\le D_{r}d^{r}(x,y)$ and define the following function
\[
\beta_{t}^{'}(x,y)=e^{\frac{1}{6}D_{r}K(1-t^{2})d^{r}(x,y)}.
\]
 If $K<0$ then obviously $\beta_{t}^{'}\le\beta_{t}$ and the interpolation
inequality above holds. As above we conclude that the functional is
$K'$-convex for some $K'$ depending on $D_{r}K$ and $p>r$. 
\end{rem*}

\subsection*{Positive curvature and global Poincar\'e inequality }

In this section we will show a Poincar\'e inequality for positively
curved spaces first proven by Lott and Villani in \cite{LV2007} for
the case $p=2$. 

For that fix a metric measure space $(M,d,\mu)$ and let $q$ be the
Hölder conjugate of $p$. Then for a given $U\in C^{2}(\mathbb{R})$
we define the $q$-Fisher information (associated to $(U,\mu)$) 
\begin{eqnarray*}
I_{q}(\nu) & = & \int U''(\rho)^{q}|D^{-}\rho|^{q}d\nu\\
 & = & \int\rho U''(\rho)^{q}|D^{-}\rho|^{q}d\mu
\end{eqnarray*}
where $\nu$ is an absolutely continuous measure w.r.t. $\mu$.

In case the $CD_{p}(K,N)$ holds for $K>0$ and $N\in(1,\infty)$
the following directly follows from \cite[Theorem 5.34]{LV2007} without
changing the proofs.
\begin{lem}
Let $(M,d,\mu)$ be a metric measure space satisfying $CD_{p}(K,N)$
for $K>0$ and $N\in(1,\infty)$. Then for any Lipschitz function
$f$ on $M$ with $\int fd\mu=0$ it holds 
\[
\int f^{2}d\mu\le\frac{N-1}{KN}\in|D^{-}f|^{2}d\mu.
\]

\end{lem}
However in case $N=\infty$ we need to adjust the proof using the
Lemma below.
\begin{lem}
Let $(M,d,\mu)$ be compact geodesic metric measure space and $U$
be continuous convex function on $[0,\infty)$ with $U(0)=0$. Let
$\nu\in\mathcal{P}_{p}(M)$ and assume $t\mapsto\mu_{t}$ is a geodesic
in $\mathcal{P}_{p}(M)$ from $\mu_{0}=\nu$ to $\mu_{1}=\mu$ such
that the functional $\mathcal{U}$ (associated to $(U,\mu)$) is $K$
convex along $\mu_{t}$, i.e. 
\[
\mathcal{U}(\mu_{t})\le(1-t)\mathcal{U}(\mu_{0})+t\mathcal{U}(\mu_{1})-\frac{K}{2}t(1-t)w_{p}^{2}(\mu_{0},\mu_{1}).
\]

Then
\[
\frac{K}{2}w_{p}(\nu,\mu)\le\mathcal{U}(\nu)-\mathcal{U}(\mu).
\]
If $U$ is $C^{2}$-regular on $(0,\infty)$, $\nu=\rho\mu$ for some
positive Lipschitz function $\rho$ on $M$ with $\mathcal{U}(\nu)<\infty$
and $\mu_{t}$ is absolutely continuous for each $t\in[0,1]$ then
\[
\mathcal{U}(\nu)-\mathcal{U}(\mu)\le w_{p}(\nu,\mu)\sqrt[q]{I_{q}(\nu)}-\frac{K}{2}w_{p}(\nu,\mu)^{2}.
\]
\end{lem}
\begin{proof}
The proof follows from \cite[Proposition 3.36]{LV2009} by making
some minor adjustments. We will include the whole proof, since it
can also be used to generalize \cite[Theorem 5.3]{LV2007} (note that
$\mathcal{U}$ with $U\in\mathcal{DC}_{N}$ is not necessarily $K$-convex).

The first part follows directly from the $K$-convexity: Let $\phi(t)=\mathcal{U}(\mu_{t})$,
then 
\[
\phi(t)\le t\phi(1)+(1-t)\phi(0)-\frac{1}{2}t(1-t)w_{p}(\nu,\mu)^{2}.
\]
If the inequality does not hold then $\phi(0)-\phi(1)<\frac{1}{2}w_{p}(\nu,\mu)^{2}$
and hence 
\[
\phi(t)-\phi(1)\le(1-t)\left(\phi(0)-\phi(1)-\frac{K}{2}tw_{p}(\nu,\mu)^{2}\right)
\]
which implies that $\phi(t)-\phi(1)$ is negative for $t$ close to
$1$. But this contradicts \cite[Lemma 3.36]{LV2009}, i.e. $\mathcal{U}(\mu)\ge\mathcal{U}(\nu)=U(1)$.
Therefore, the first inequality holds.

To prove the second part, let $\rho_{t}$ be the density of $\mu_{t}$.
Then $\phi(t)=\int U(\rho_{t})d\mu$ and from above we have 
\[
\phi(0)-\phi(1)\le-\frac{\phi(t)-\phi(1)}{t}-\frac{K}{2}(1-t)w_{p}(\nu,\mu)^{2}.
\]
So to prove the second inequality we just need to show 
\[
\liminf_{t\to\infty}\left(-\frac{\phi(t)-\phi(0)}{t}\right)\le w_{p}(\nu,\mu)\sqrt[q]{I_{q}(\nu)}.
\]

Since $U$ is convex we have 
\[
U(\rho_{t})-U(\rho_{0})\ge U'(\rho_{0})(\rho_{t}-\rho_{0}).
\]
Integrating w.r.t. $\mu$ and dividing by $-t<0$ we get
\begin{eqnarray*}
-\frac{1}{t}\left(\phi(t)-\phi(0)\right) & \le & -\frac{1}{t}\int U'(\rho_{0}(x))\left(d\mu_{t}(x)-d\mu(x)\right)\\
 & = & -\frac{1}{t}\int U'(\rho_{0}(\gamma_{t}))-U'(\rho_{0}(\gamma_{0}))d\Pi(\gamma)
\end{eqnarray*}
where $\Pi$ is the optimal transference plan in $\mathcal{P}(\operatorname{Geo})$
associated to $t\mapsto\mu_{t}$.

Since $U'$ is non-decreasing and $d(\gamma_{t},\gamma_{0})=td(\gamma_{0},\gamma_{1})$
we obtain for 
\begin{eqnarray*}
-\frac{1}{t}\int U'(\rho_{0}(\gamma_{t}))-U'(\rho_{0}(\gamma_{0}))d\Pi(\gamma) & \le & -\frac{1}{t}\int_{\rho_{0}(\gamma_{t})\le\rho_{0}(\gamma_{0})}\left[U'(\rho_{0}(\gamma_{t}))-U'(\rho_{0}(\gamma_{0}))\right]d\Pi(\gamma)\\
 & \le & \int\frac{U'(\rho_{0}(\gamma_{t}))-U'(\rho_{0}(\gamma_{0}))}{\rho_{0}(\gamma_{t})-\rho_{0}(\gamma_{0})}\\
 &  & \qquad\times\frac{[\rho_{0}(\gamma_{t})-\rho_{0}(\gamma_{0})]_{-}}{d(\gamma_{t},\gamma_{0})}d(\gamma_{1},\gamma_{0})d\Pi(\gamma).
\end{eqnarray*}
Applying Hölder inequality we get 
\begin{eqnarray*}
\sqrt[q]{\int\frac{[U'(\rho_{0}(\gamma_{t}))-U'(\rho_{0}(\gamma_{0}))]^{q}}{[\rho_{0}(\gamma_{t})-\rho_{0}(\gamma_{0})]^{q}}\frac{[\rho_{0}(\gamma_{t})-\rho_{0}(\gamma_{0})]_{-}^{q}}{d(\gamma_{t},\gamma_{0})^{q}}d\Pi(\gamma)}\\
\times\sqrt[p]{\int d(\gamma_{0},\gamma_{1})^{q}d\Pi(\gamma)}.
\end{eqnarray*}
where the second factor is just $w_{p}(\nu,\mu)$. Taking continuity
of $\rho_{0}$ and the definition of $|D-\rho_{0}|$ into account
we conclude as in the proof of \cite[Proposition 3.36]{LV2009} that
the first factor equals 
\[
\sqrt[q]{\int U''(\rho_{0})^{q}|D^{-}\rho_{0}|^{q}d\nu}=\sqrt[q]{I_{q}(\nu)}.
\]
\end{proof}
\begin{cor}
Assume that the (weak) $CD_{p}(K,\infty)$ condition holds for $K>0$
and some $N\in[1,\infty]$. Then for all $\nu\in\mathcal{P}_{p}(M)$
\[
\frac{K}{2}w_{p}(\nu,\mu)^{2}\le\mathcal{U}_{\infty}(\nu).
\]
If $\nu$ is absolutely continuous with positive Lipschitz density
$\rho$ then 
\[
\mathcal{U}_{\infty}(\nu)\le w_{p}(\nu,\mu)\sqrt[q]{I_{q}(\nu)}-\frac{K}{2}w_{p}(\nu,\mu)^{2}\le\frac{1}{2K}\left(I_{q}(\nu)\right)^{\frac{2}{q}}.
\]
\end{cor}
\begin{proof}
Just note that if $\mathcal{U}_{\infty}$ is $K$-convex along a geodesic
$t\mapsto\mu_{t}$ between absolutely continuous measures, then each
$\mu_{t}$ is absolutely continuous. 
\end{proof}
Note that in this case 
\[
I_{q}(\rho\mu)=\int\rho\frac{1}{\rho^{q}}|D^{-}\rho|^{q}d\mu=\int\frac{|D^{-}\rho|^{q}}{\rho^{q-1}}d\mu.
\]

Similar to \cite[Section 6.2]{LV2009} we will show that the $(2,q)$-log-Sobolev
inequality 
\[
\mathcal{U}_{\infty}(\rho\mu)\le\frac{1}{2K}\left(I_{q}(\rho\mu)\right)^{\frac{2}{q}}.
\]
implies a global $(2,q)$-Poincar\'e inequality. Note that the $(2,q)$-log-Sobolev
inequality is different from the one defined in \cite{GRS2012}.
\begin{cor}
Assume for $K>0$ and all positive Lipschitz functions 
\[
\mathcal{U}_{\infty}(\rho\mu)\le\frac{1}{2K}\left(I_{q}(\rho\mu)\right)^{\frac{2}{q}}.
\]
Then the $(2,q)$-Poincar\'e inequality holds with factor independent
of $q$, i.e.
\[
\]
\[
\left(\int(h-\bar{h})^{2}d\mu\right)^{\frac{1}{2}}\le\frac{1}{\sqrt{2K}}\left(\int|D^{-}h|^{q}d\mu\right)^{\frac{1}{q}}
\]
for $h\in\operatorname{Lip}(M)$. In particular, this holds if $(M,d,\mu)$
satisfies the weak $CD_{p}(K,\infty)$ condition.\end{cor}
\begin{proof}
We will first prove 
\begin{claim*}
If $f\in\operatorname{Lip}(M)$ satisfies $\int f^{p}d\mu=1$ then
\[
\left(\int f^{q}\log f^{q}d\mu\right)^{\frac{1}{2}}\le\frac{q}{\sqrt{2K}}\left(\int|D^{-}f|^{q}d\mu\right)^{\frac{1}{q}}.
\]
\end{claim*}
\begin{proof}
[Proof of the claim]For any $\epsilon>0$ let $\rho_{\epsilon}=\frac{f^{p}+\epsilon}{1+\epsilon}$
then from the previous corollary 
\[
\int\rho_{\epsilon}\log\rho_{\epsilon}d\mu\le\frac{1}{2K}\left(\int\frac{|D^{-}\rho_{\epsilon}|^{q}}{\rho_{\epsilon}^{q-1}}d\mu\right)^{\frac{2}{q}}.
\]
By chain rule we have 
\[
\frac{|D^{-}\rho_{\epsilon}|^{q}}{\rho_{\epsilon}^{q-1}}=\frac{1}{1+\epsilon}\frac{(qf^{q-1})^{q}}{(f^{q}+\epsilon)^{q-1}}|\nabla^{-}f|^{q}\to q^{q}|D^{-}f|^{q}
\]
as $\epsilon\to0$, which implies the claim.
\end{proof}
Assume w.l.o.g. $\int h=0$. For $\epsilon\in[0,\frac{1}{\|h\|_{\infty}})$
set $f_{\epsilon}=\sqrt[q]{1+\epsilon h}>0$. Then by chain rule 
\[
|D^{-}f_{\epsilon}|=\frac{\epsilon|D^{-}h|}{q\left(1+\epsilon h\right)^{\frac{q-1}{q}}}
\]
and thus 
\[
\lim_{\epsilon\to0}\frac{1}{\epsilon}\left(\int|D^{-}f_{\epsilon}|^{q}d\mu\right)^{\frac{1}{q}}=\frac{1}{q}\left(\int|D^{-}h|^{q}d\mu\right)^{\frac{1}{q}}.
\]
Note that the Taylor expansion of $x\log x-x+1$ around $x_{0}=1$
is given by $\frac{1}{2}(x-1)^{2}+\ldots$, and thus
\[
\lim_{\epsilon\to0}\frac{1}{\epsilon}\int f_{\epsilon}^{q}\log f_{\epsilon}^{q}d\mu=\int h^{2}d\mu.
\]
Combining this we get 
\[
\left(\int h^{2}d\mu\right)^{\frac{1}{2}}\le\frac{1}{\sqrt{2K}}\left(\int|D^{-}h|^{q}d\mu\right)^{\frac{1}{q}}.
\]

\end{proof}

\subsection*{Metric Brenier}
\begin{lem}
[{\cite[5.4]{Gigli2012}}]\label{lem:Gigli-pwconv}Let $(M,d,\mu)$
be a metric measure space and $(\mu_{n})_{n\in\mathbb{N}}$ be a sequence
$\mathcal{P}(M)$ and let $\mu_{0}\in\mathcal{P}(M)$ be such that
$\mu_{0}\ll\mu$ . Assume for some bounded closed set $B\subset M$
with $\mu(B)<\infty$ we have $\operatorname{supp}\mu_{n}\cup\operatorname{supp}\mu_{0}\subset B$,
$\mu_{n}$ converges weakly to $\mu$ and 
\[
\mathcal{U}_{N}(\mu_{n})\to\mathcal{U}_{N}(\mu_{0})\quad\mbox{as }n\to\infty.
\]
Then for every bounded Borel function $f:B\to\mathbb{R}$ it holds
\[
\lim_{n\to\infty}\int fd\mu_{n}=\int fd\mu
\]
 \end{lem}
\begin{prop}
\label{prop:brenier}Let $(M,d,\mu)$ be a metric measure space and
$B$ be a bounded closed subset of $M$ with $\mu(B)<\infty$. Assume
$\mu_{0}$ and $\mu_{1}$ are two probability measure in $\mathcal{P}_{p}(M)$
such that $\mu_{0}\ll\mu$ and there is an optimal coupling $\pi\in\operatorname{OptGeo}_{p}(\mu_{0},\mu_{1})$
such that 
\[
\lim_{t\to0}\mathcal{U}_{N}(\mu_{t})=\mathcal{U}_{N}(\mu_{0})
\]
and $\operatorname{supp}(\mu_{t})\subset B$, where $\mu_{t}=(e_{t})_{*}\pi$.
If $\phi$ is the associated Kantorovich potential of the pair $(\mu_{0},\mu_{1})$
and $\phi$ is Lipschitz on bounded subsets of $X$. Then for every
$\tilde{\pi}\in\operatorname{OptGeo}_{p}(\mu_{0},\mu_{1})$
\[
d(\gamma_{0},\gamma_{1})^{p}=\left(|D^{+}\phi|(\gamma_{0})\right)^{q}\quad\tilde{\pi}\mbox{-a.e. }\gamma.
\]
\end{prop}
\begin{rem*}
The proof follows by similar arguments as in \cite[5.5]{Gigli2012}
and \cite[10.3]{Ambrosio2013}.\end{rem*}
\begin{proof}
Let $x\in M$ be arbitrary and choose any $y\in\partial^{c_{p}}\phi(x)$,
then for all $z\in M$
\begin{eqnarray*}
\phi(x) & = & c_{p}(x,y)-\phi^{c_{p}}(y),\\
\phi(z) & \le & c_{p}(z,y)-\phi^{c_{p}}(y).
\end{eqnarray*}
Thus
\begin{eqnarray*}
\phi(z)-\phi(y) & \le & \frac{\left(d(z,x)+d(x,y)\right)^{p}-d^{p}(x,y)}{p}\\
 & = & (d(z,x)+h_{1}(d(z,x))\cdot d(x,y)^{p-1}
\end{eqnarray*}
where $h_{1}:\mathbb{R}\to\mathbb{R}$ is such that $h_{1}(r)=o(r)$
as $r\to0$ depending only on $p>1$. Therefore, dividing by $d(x,z)$
and letting $z\to x$ we see that 
\[
|D^{+}\phi|(x)\le\inf_{y\in\partial^{c_{p}}\phi(x)}d(x,y)^{p-1}.
\]
In particular, since for an arbitrary $\tilde{\pi}\in\operatorname{OptGeo}_{p}(\mu_{0},\mu_{1})$
we have $\gamma_{1}\in\partial^{c_{p}}\phi(\gamma_{0})$ for $\tilde{\pi}$-almost
every $\gamma$, we also have
\[
|D^{+}\phi|(\gamma_{0})\le d(\gamma_{0},\gamma_{1})^{p-1}\quad\tilde{\pi}\mbox{-a.e. }\gamma.
\]
Note that $q\cdot(p-1)=p$ and thus
\[
\int|D^{+}\phi|^{q}d\mu_{0}\ge w_{p}^{p}(\mu_{0},\mu_{1}).
\]

So it suffices to show the opposite inequality. For that let $\pi\in\operatorname{OptGeo}(\mu_{0},\mu_{1})$
as in the hypothesis. Because $\phi$ is a Kantorovich potential we
have for $t\in(0,1]$ 
\begin{eqnarray*}
\phi(\gamma_{0})-\phi(\gamma_{t}) & \ge & \frac{d(\gamma_{0},\gamma_{1})^{p}}{p}-\frac{d(\gamma_{t},\gamma_{1})^{p}}{p}\\
 & = & \frac{d(\gamma_{0},\gamma_{1})^{p}}{p}(1-(1-t)^{p})=d(\gamma_{0},\gamma_{1})^{p}(t+o(t)).
\end{eqnarray*}
Thus dividing by $d(\gamma_{0},\gamma_{t})=td(\gamma_{0},\gamma_{1})$
and integrating to the $q$-th power we get 
\[
\liminf_{t\to0}\int\left(\frac{\phi(\gamma_{0})-\phi(\gamma_{t})}{d(\gamma_{0},\gamma_{t})}\right)^{q}d\pi(\gamma)\ge\int d(\gamma_{0},\gamma_{1})^{p}d\pi(\gamma)=w_{p}^{p}(\mu_{0},\mu_{1}).
\]
Because $\phi$ is locally Lipschitz, $|D^{+}\phi|$ is an upper gradient
for $\phi$, we also have 
\begin{eqnarray*}
\int\left(\frac{\phi(\gamma_{0})-\phi(\gamma_{t})}{d(\gamma_{0},\gamma_{t})}\right)^{q}d\pi(\gamma) & \le & \int\frac{1}{t^{q}}\left(\int_{0}^{t}|D^{+}\phi|(\gamma_{s})ds\right)^{q}d\pi(\gamma)\\
 & \le & \int\frac{t^{\frac{q}{p}}}{t^{q}}\int_{0}^{t}|D^{+}\phi|^{q}(\gamma_{s})dsd\pi(\gamma)\\
 & = & \frac{1}{t}\int_{0}^{t}\int|D^{+}\phi|^{q}(\gamma_{s})d\pi(\gamma)ds
\end{eqnarray*}
 because $\frac{q}{p}=\frac{1}{p-1}=q-1$.

Now our assumptions imply that $|D^{+}\phi|^{q}$ is a bounded Borel
functions thus we can apply the previous lemma to get (see also \cite[5.5]{Gigli2012}
\[
\lim_{t\to0}\frac{1}{t}\int_{t}^{t}\int|D^{+}\phi|^{q}(\gamma_{s})d\pi(\gamma)ds=\int|D^{+}\phi|^{q}d\mu_{0}.
\]

\end{proof}
In order to avoid the introduction of complicated notation, we just
remark that one can also prove \cite[Corollary 5.8]{Gigli2012} and
show that the plan $\pi$ above weakly $q$-represents $\nabla(-\phi)$
(for definition see \cite[Definition 5.7]{Gigli2012}).

\subsection*{Laplacian comparison}

As an application to the metric Brenier theorem we get the following.
Since we do not prove the theorem, we refer to \cite{Gigli2012} for
a precise definition of infinitesimal strictly convex spaces.
\begin{thm}
[Comparison estimates] Let $K\in\mathbb{R}$ and $N\in(1,\infty)$
and $(M,d,\mu)$ be an infinitesimal strictly convex $CD_{p}(K,N)$-space.
If $\phi:X\to\mathbb{R}$ is a $c_{p}$-concave function. Then 
\[
\phi\in D(\Delta_{q})\qquad\mbox{ and }\qquad\Delta^{q}\phi\le N\tilde{\sigma}_{K,N}(|\nabla\phi|_{w}^{q-1})d\mu
\]
where 
\[
\tilde{\sigma}_{K,N}(\theta)=\begin{cases}
\frac{1}{N}\left(1+\theta\sqrt{K/(N-1)}\operatorname{cotan}\left(\theta\sqrt{\frac{K}{N-1}}\right)\right) & \mbox{if }K>0\\
1 & \mbox{if }K=0\\
\frac{1}{N}\left(1+\theta\sqrt{K/(N-1)}\operatorname{cotanh}\left(\theta\sqrt{\frac{K}{N-1}}\right)\right) & \mbox{if }K<0
\end{cases}
\]
\end{thm}
\begin{proof}
Follow \cite[Theorem 5.14]{Gigli2012} and just note that the metric
Brenier theorem implies $d(\gamma_{0},\gamma_{1})=|\nabla\phi|_{w}^{q-1}$. \end{proof}
\begin{cor}
[Laplacian comparision of the distance] For any $x_{0}$ one has
\[
\frac{d_{x_{0}}^{p}}{p}\in D(\Delta_{q})\qquad\mbox{ with }\qquad\Delta_{q}\frac{d_{x_{0}}^{p}}{p}\le N\tilde{\sigma}_{K,N}(d_{x_{0}})d\mu\quad\forall x_{0}\in X
\]
and 
\[
d_{x_{0}}\in D(\Delta_{q},X\backslash\{x_{0}\})\qquad\mbox{ with }\qquad\Delta^{q}d_{x_{0}}\big|_{X\backslash\{x_{0}\}}\le\frac{N\tilde{\sigma}_{K,N}(d_{x_{0}})}{d_{x_{0}}^{p-1}}d\mu.
\]
 
\[
\]
\end{cor}
\begin{rem*}
Note that formally 
\begin{eqnarray*}
\Delta^{q}\frac{d_{x_{0}}^{p}}{p} & = & \nabla\cdot\left(|\nabla\frac{d_{x_{0}}^{p}}{p}|^{q-2}\nabla\frac{d_{x_{0}}^{p}}{p}\right)\\
 & = & \nabla\cdot\left((d_{x_{0}}^{p-1})^{q-1}\nabla d_{x_{0}}\right)\\
 & = & \nabla\cdot\left(d_{x_{0}}\nabla d_{x_{0}}\right)=\Delta\frac{d_{x_{0}}^{2}}{2},
\end{eqnarray*}
thus the result might not give any new results in the smooth setting.\end{rem*}
\begin{proof}
[Proof of the Remark]Note first that $d_{x_{0}}^{p}/p$ is $c_{p}$-concave
and because $|\nabla d_{x_{0}}|=1$ almost everywhere and by chain
rule $|\nabla(d_{x_{0}}^{p}/p)|=d_{x_{0}}^{p-1}$.
\end{proof}

\subsection*{$c_{p}$-concavity of Busemann functions}

In \cite{Gigli2013a} Gigli used, beside many other things, $c_{2}$-concavity
of the Busemann function and linearity of the Laplacian to prove the
splitting theorem for $RCD(K,N)$-spaces, i.e. $CD(K,N)$-spaces with
a linear Laplacian. We will show that the Busemann function is $c_{p}$-concave
for any $p\in(1,\infty)$, even more general it is $c_{L}$-concave.
In the non-linear setting and the case $p=2$, Ohta \cite{Ohta2013a}
used a comparison principle to show that Busemann functions on Finsler
manifolds are harmonic. If such a principle holds in a more general
non-linear setting and even for the case $p\ne2$, one could also
conclude harmonicity (resp. $p$-harmonicity) of Busemann functions. 

A function $\gamma:[0,\infty)\to M$ is called geodesic ray if for
any $T>0$ the restriction to $[0,T]$ is a minimal geodesic. Furthermore,
we will always assume that geodesic rays are parametrized by arc length.
We can the Busemann function $b$ associated to $\gamma$ by 
\[
b(x)=\lim_{t\to\infty}b_{t}(x)\qquad\mbox{ where }b_{t}(x)=d(x,\gamma_{t})-t.
\]

Note
\[
t\mapsto b_{t}(x)\qquad\mbox{ is non-increasing}
\]

\begin{lem}
Let $(M,d)$ be a geodesic space and $b$ be the Busemann functions
associated to some geodesic ray $\gamma:[0,\infty)\to X$. Then $b$
is $c_{p}$-concave.\end{lem}
\begin{proof}
From Lemma \ref{lem:cp-calculus} we know $b^{c_{p}\bar{c}_{p}}\ge b$,
so that we only need to show the opposite inequality. 

Fix an arbitrary $x\in X$ and $t\ge0$ and let $\gamma^{t,x}:[0,d(x,\gamma_{t})]\to X$
be a unit speed geodesic connecting $x$ and $\gamma_{t}$. Then for
any $t\ge t_{x}$ we have $d(x,\gamma_{t})\ge1$ and 
\[
b^{c_{p}\bar{c}_{p}}(x)=\inf_{y\in X}\sup_{\tilde{x}\in X}\frac{d^{p}(x,y)}{p}-\frac{d^{p}(\tilde{x},y)}{p}+b(\tilde{x})\le\sup_{\tilde{x}\in X}\frac{1}{p}-\frac{d^{p}(\tilde{x},\gamma_{1}^{t,x})}{p}+b_{t}(\tilde{x}).
\]
Furthermore, for any $\tilde{x}\in X$ and $t\ge t_{x}$ we also have
\begin{eqnarray*}
\frac{1}{p}-\frac{d^{p}(\tilde{x},\gamma_{1}^{t,x})}{p}+b_{t}(\tilde{x}) & = & \frac{1}{p}-\frac{d^{p}(\tilde{x},\gamma_{1}^{t,x})}{p}+d(\tilde{x},\gamma_{t})-t\\
 & \le & \frac{1}{p}-\frac{d^{p}(\tilde{x},\gamma_{1}^{t,x})}{p}+d(\tilde{x},\gamma_{1}^{t,x})+d(\gamma_{1}^{t,x},\gamma_{t})-t\\
 & = & -\frac{p-1}{p}-\frac{d^{p}(\tilde{x},\gamma_{1}^{t,x})}{p}+d(\tilde{x},\gamma_{1}^{t,x})+d(x,\gamma_{t})-t\\
 & \le & d(x,\gamma_{t})-t=b_{t}(x)
\end{eqnarray*}
where we used Young's inequality and $(p-1)/p=1/q$. Therefore,
\[
b^{c_{p}\bar{c}_{p}}(x)\le\lim_{t\to\infty}b_{t}(x)=b(x).
\]

\end{proof}
Actually, we can also show that the Busemann function is $c_{L}$-concave
for any convex functional $L$ such that $c_{L}(x,y)=L(d(x,y))$ (see
chapter on Orlicz-Wasserstein spaces).
\begin{lem}
Let $(M,d)$ be a geodesic space and $b$ be the Busemann functions
associated to some geodesic ray $\gamma:[0,\infty)\to X$. Then $b$
is $c_{L}$-concave where such that $c_{L}(x,y)=L(d(x,y))$ for some
convex function $L:[0,\infty)\to[0,\infty)$ such that $L^{*}(1)=r-L(r)$
for some $r\ge0$.\end{lem}
\begin{rem*}
The condition for such an $r$ to exist rather weak, e.g. superlinearity
of $L$ is sufficient.\end{rem*}
\begin{proof}
Let $L^{*}$ be the Legendre transform of $L$, then Young's inequality
holds 
\[
xy\le L(x)+L^{*}(y),
\]
in particular $x\le L(x)+L^{*}(1)$. 

Let $r$ be such that $L^{*}(1)=r-L(r)$. As above, we only need to
show that $b^{c_{L}\bar{c}_{L}}\le b$. We have
\[
b^{c_{L}\bar{c}_{L}}(x)=\inf_{y\in X}\sup_{\tilde{x}\in X}L(d(x,y))-L(d(\tilde{x},y))+b(\tilde{x})\le\sup_{\tilde{x}\in X}L(r)-d(\tilde{x},\gamma_{r}^{t,x})+b_{t}(\tilde{x}).
\]
Furthermore, for all $\tilde{x}\in M$ and $t\ge t_{x}$ such that
$d(x,\gamma_{t})\ge r$ we get 
\begin{eqnarray*}
L(r)-L(d(\tilde{x},\gamma_{r}^{t,x}))+b_{t}(\tilde{x}) & = & L(r)-L(d(\tilde{x},\gamma_{r}^{t,x}))+d(\tilde{x},\gamma_{t})-t\\
 & \le & L(r)-L(d(\tilde{x},\gamma_{r}^{t,x}))+d(\tilde{x},\gamma_{1}^{t,x})+d(\gamma_{1}^{t,x},\gamma_{t})-t\\
 & = & L(r)-r-L(d(\tilde{x},\gamma_{r}^{t,x}))+d(\tilde{x},\gamma_{r}^{t,x})+d(x,\gamma_{t})-t\\
 & = & -L^{*}(1)-L(d(\tilde{x},\gamma_{r}^{t,x}))+d(\tilde{x},\gamma_{r}^{t,x})+d(x,\gamma_{t})-t\\
 & \le & d(x,\gamma_{t})-t=b_{t}(x).
\end{eqnarray*}
where we used Young's inequality to get the last inequality. Therefore,
\[
b^{c_{p}\bar{c}_{p}}(x)\le\lim_{t\to\infty}b_{t}(x)=b(x).
\]

\end{proof}
\appendix

\section{Appendix}

In this appendix we show that the interpolation inequality can be
proven also for Oclicz-Wasserstein spaces using similar arguments.
Before that we will define and investigate Orlicz-Wasserstein spaces.
The main difference between a general convex and increasing function
$L$ and a homogeneous function is that there is no well-defined dual
problem. However, one can use $c_{L}$-concave function and the geodesic
structure to determine the interpolation potentials.

\subsection*{Orlicz-Wasserstein spaces}

Let $L:[0,\infty)\to[0,\infty)$ be a strictly convex increasing functions
with $L(0)=0$. Assume further there is an increasing function $l:(0,\infty)\to(0,\infty)$
with $\lim_{r\to0}l(r)=0$ and 
\[
L(r)=\int_{0}^{r}l(s)ds
\]
and hence $L'(s)=l(s)$.

Define $L_{\lambda}(r)=L(r/\lambda)$ and note
\begin{eqnarray*}
L_{\lambda}(r) & = & \int_{0}^{r}l_{\lambda}(s)ds\\
 & = & \int_{0}^{r/\lambda}l(s)ds
\end{eqnarray*}
and thus 
\[
l_{\lambda}(s)=\frac{1}{\lambda}l\left(\frac{s}{\lambda}\right)
\]
and 
\[
l_{\lambda}^{-1}(t)=\lambda l^{-1}(\lambda t).
\]

We denote by $c_{L}$ the cost function given by $c_{L}(x,y)=L(d(x,y))$
and as an abbreviation $c_{\lambda}=c_{L_{\lambda}}$.

The $c_{L}$-transform of a function $\phi:X\to\Rinf$ relative to
$(X,Y)$ is defined as 
\[
\phi^{c_{L}}(y)=\inf_{x}c_{L}(x,y)-\phi(x)
\]
and similarly the $\bar{c}^{L}$-transform.
\begin{defn}
[Orlicz-Wasserstein space] Let $\mu_{i}$ be two probability measures
on $M$ and define 
\end{defn}
\[
w_{L}(\mu_{0},\mu_{1})=\inf\left\{ \lambda>0\,|\,\inf_{\pi\in\Pi(\mu_{0},\mu_{1})}\int L_{\lambda}\left(d(x,y)\right)d\pi(x,y)\le1\right\} .
\]
With convention $\inf\varnothing=\infty$. 

According to Sturm \cite[Proposition 3.2]{Sturm2011}, $w_{L}$ is
a complete metric on 
\[
\mathcal{P}_{L}(M):=\{\mu_{1}\in\mathcal{P}(M)\,|\, w_{L}(\mu_{1},\delta_{x_{0}})<\infty\}
\]
where $x_{0}$ is some fixed point. 

Even though the following lemma is not needed, it makes many proofs
below easier.
\begin{lem}
[{\cite[Proposition 3.1]{Sturm2011}}]For every $\mu_{i}\in\mathcal{P}_{L}(M)$
there is an optimal coupling $\pi_{opt}$ of $(\mu_{0},\mu_{1})$
such that 
\[
\lambda_{min}=w_{L}(\mu_{0},\mu_{1})\Rightarrow\int L_{\lambda_{min}}(d(x,y))d\pi_{opt}(x,y)=1.
\]

\end{lem}
Actually the Lemma shows that the whole theory of Kantorovich potentials
will depend on the distance. Furthermore, the $c_{L}$-convace functions
are not necessarily star-shaped. Nevertheless, we will show that $\mathcal{P}_{L}(M)$
is a geodesic space iff $M$ is and that a similar property to the
star-shapedness holds.
\begin{prop}
Let $\Phi$ be a convex increasing function with $\Phi(1)=1$, then
\[
w_{L}\le w_{\Phi\circ L}.
\]
\end{prop}
\begin{rem*}
This just uses Sturm's idea to show the same inequality for Luxemburg
norm of Orlicz spaces. Compare this also to \cite[Remark 6.6]{Villani2009},
but note that Villani defines $w_{p}$ without the factor $\frac{1}{p}$. \end{rem*}
\begin{proof}
This follows easily from Jensen's inequality. Let $\mu_{0},\mu_{1}$
be two measures and $\lambda>0$ and $\pi$ be a coupling such that
$\int(\Phi\circ L)_{\lambda}(d(x,y))d\pi(x,y)\le1$ then since $(\Phi\circ L)_{\lambda}=\Phi\circ L$
\[
\Phi(\int L_{\lambda}(d(x,y))d\pi(x,y))\le\int\Phi\circ L_{\lambda}(d(x,y))d\pi(x,y)\le1
\]
Since $\Phi(1)\le1$ and $\Phi$ is increasing, we see that $\int L_{\lambda}(d(x,y))d\pi(x,y)\le1$
which implies $w_{L}(\mu_{0},\mu_{1})\le w_{\Phi\circ L}(\mu_{0},\mu_{1})$.\end{proof}
\begin{prop}
Assume for all $\lambda>0$ 
\[
\sup_{R\to\infty}\frac{L(\lambda R)}{L(R)}<\infty.
\]
If $\mu_{n},\mu_{\infty}\in\mathcal{P}_{L}(M)$ and $\mu_{n}$ converges
weakly to $\mu_{\infty}$, then 
\[
w_{L}(\mu_{n},\mu_{\infty})\to0\;\Longleftrightarrow\;\lim_{R\to\infty}\limsup_{n\to\infty}\int_{M\backslash B_{R}(x_{0})}L_{\lambda}(d(x,x_{0}))d\mu_{n}=0
\]
for all $0<\lambda<\lambda_{0}$.\end{prop}
\begin{rem*}
This generalizes \cite[Theorem 7.12]{Villani2003}. The other equivalences
in Villani's theorem can be proven similarly. We, however, only need
the one stated above.\end{rem*}
\begin{proof}
Fix some $x_{0}\in M$. It is not difficult to see that for any $\lambda>0$
and any $\mu'\in\mathcal{P}_{L}(M)$
\[
\lim_{R\to\infty}\int_{M\backslash B_{R}(x_{0})}L_{\lambda}(d(x,x_{0}))d\mu'(x)=0\;\Longleftrightarrow\;\lim_{R\to\infty}\int_{M\backslash B_{R}(x_{0})}L(d(x,x_{0}))d\mu'(x)=0.
\]
First assume $w_{L}(\mu_{n},\mu_{\infty})$ and let $\pi_{n}$ be
the optimal plans with $l_{n}=w_{L}(\mu_{n},\mu_{\infty})$ and 
\[
\int L_{l_{n}}(d(x,y))d\pi_{n}(x,y)=1.
\]
For $n$ large, for any $\lambda>0$ choose a sequence $r_{n}\le\frac{1}{2}$
such that $l_{n}=r_{n}\lambda$. Then using the triangle inequality
and convexity of $L$ we get 
\begin{eqnarray*}
\int L_{\lambda}\left(d(x,x_{0})\right)d\mu_{n}(x) & = & \int L_{\lambda}\left(d(x,x_{0})\right)d\pi_{n}(x,y)\\
 & \le & r_{n}\int L_{r_{n}\lambda}\left(d(x,y)\right)d\pi_{n}(x,y)+(1-r_{n})\int L_{(1-r_{n})\lambda}\left(d(y,x_{0})\right)d\pi_{n}(x,y)\\
 & \le & r_{n}+(1-r_{n})\int L_{\frac{1}{2}\lambda}\left(d(y,x_{0})\right)d\mu_{\infty}(y).
\end{eqnarray*}
since $L_{(1-r_{n})\lambda}\le L_{\frac{1}{2}\lambda}$. Therefore,
\[
\lim_{R\to\infty}\limsup_{n\to\infty}\int_{M\backslash B_{R}(x_{0})}L_{\lambda}(d(x,x_{0}))d\mu_{n}(x)\le\lim_{R\to\infty}\int_{M\backslash B_{R}(x_{0})}L_{\frac{1}{2}\lambda}(d(x,x_{0}))d\mu_{\infty}(x)=0.
\]

Now assume that 
\[
\lim_{R\to\infty}\limsup_{n\to\infty}\int_{M\backslash B_{R}(x_{0})}L_{\lambda}(d(x,x_{0}))d\mu_{n}(x)=0
\]
for any $0<\lambda<\lambda_{0}$ and $\mu_{n}$ converges weakly to
$\mu_{\infty}$. This bound ensures that $\mu_{\infty}$ is in $\mathcal{P}_{L}(M)$.

Take any $\lambda>0$ and an optimal coupling $\pi_{n}$ of $(\mu_{n},\mu_{\infty})$
w.r.t. $L_{\lambda}$. For $R>0$ and $A\wedge B=\min\{A,B\}$ we
have
\[
d(x,y)\le d(x,y)\wedge R+2d(x,x_{0})\chi_{B_{R/2}(x_{0})}(x)+2d(x_{0},y)\chi_{B_{R/2}(x_{0})}(y)
\]
and thus by convexity of $L$ and $L(0)=0$
\[
L_{\lambda}(d(x,y))\le\frac{1}{3}L_{\frac{\lambda}{3}}(d(x,y)\wedge R)+\frac{1}{3}L_{\frac{\lambda}{6}}(d(x,x_{0})\chi_{B_{R/2}(x_{0})}(x))+\frac{1}{3}L_{\frac{\lambda}{6}}(d(x_{0},y)\chi_{B_{R/2}(x_{0})}(y)).
\]
Thus integrating over $\pi_{n}$ we get 
\begin{eqnarray*}
3\int L_{\lambda}(d(x,y))d\pi_{n}(x,y) & \le & \int L_{\frac{\lambda}{3}}(d(x,y)\wedge R)d\pi_{n}(x,y)\\
 &  & +\int_{M\backslash B_{R/2}(x_{0})}L_{\frac{\lambda}{6}}(d(x,x_{0}))d\mu_{n}(x)\\
 &  & +\int_{M\backslash B_{R/2}(x_{0})}L_{\frac{\lambda}{6}}(d(x_{0},y))d\mu_{\infty}(y).
\end{eqnarray*}
we first take the $\limsup$ with $n\to\infty$ and then $R\to\infty$
and conclude that the last two terms converges to zero by our assumption
and since $L_{\frac{\lambda}{3}}(d(x,y)\wedge R)$ is a bounded continuous
function and $\pi_{n}$ converges weakly to the trivial coupling $(\operatorname{Id}\times\operatorname{id})_{*}\mu_{\infty}$,
the first term converges to zero as well. In particular, for $n\ge N(\lambda)$
we have 
\[
\int L_{\lambda}(d(x,y))d\pi_{n}(x,y)\le1.
\]
and thus 
\[
w_{L}(\mu_{n},\mu_{\infty})\le\lambda.
\]
Since $\lambda$ was arbitrary we conclude $w_{L}(\mu_{n},\mu_{\infty})\to0$. 
\end{proof}

\begin{prop}
Assume $M$ is a proper metric space and $\Phi$ is convex, increasing,
$\Phi(1)=1$ and $L(r)\to\infty$ and $r/\Phi(r)\to0$ as $r\to\infty.$
In addition, assume for all $\lambda>0$ 
\[
\sup_{R\to\infty}\frac{L(\lambda R)}{L(R)}<\infty.
\]
 Suppose $A$ is closed subset of $\mathcal{P}_{L}(M)$ such that
$w_{\tilde{L}}$ is bounded where $\tilde{L}=\Phi\circ L$. Then $A$
is precompact in $\mathcal{P}_{L}(M)$.\end{prop}
\begin{rem*}
Compare this to \cite[Theorem 6]{Kell2011} for the case $L(t)=t^{p}$,
$\Phi(t)=t^{r}$ for $p\ge1$ and $r>1$. \end{rem*}
\begin{proof}
It suffices to show that each $w_{\tilde{L}}$-ball is compact in
$\mathcal{P}_{L}(M)$

So for some $r>0$ and $\mu_{0}\in\mathcal{P}_{\tilde{L}}(M)\subset\mathcal{P}_{L}(M)$
let 
\[
\tilde{B}:=\tilde{B}_{r}(\mu_{0})=\{\mu_{1}\in\mathcal{P}_{L}(M)\,|\, w_{\tilde{L}}(\mu_{0},\mu_{1})\le r\}.
\]
and let $(\mu_{n})_{n\in\mathbb{N}}$ be a sequence in $ $$\tilde{B}$.
Then there are (optimal) couplings $\pi_{n}$ such that 
\[
\int\tilde{L}_{r}(d(x,y))d\pi_{n}(x,y)\le1
\]
 (for $w_{\tilde{L}}(\mu_{n},\mu_{0})<r$ just take the definition.
Using the proposition above,we see
\[
\int L_{r}(d(x,y))d\pi_{n}(x,y)\le1.
\]

Because of the stability of optimal couplings are stable and lower
semicontinuity of the cost \cite[Theorem 5.20, Lemma 4.3]{Villani2009},
we only need to show that $(\mu_{n})_{n\in\mathbb{N}}$ is weakly
precompact and 
\[
\lim_{R\to\infty}\limsup_{n\to\infty}\int_{M\backslash B_{R}(x_{0})}L_{\lambda}(d(x,x_{0}))d\mu_{n}=0
\]
i.e. it is precompact in $\mathcal{P}_{L}(M)$ by the lemma above. 

Since $\tilde{B}$ is bounded w.r.t. $w_{\tilde{L}}$ we can assume
that for some $R>0$ 
\[
w_{\tilde{L}}(\mu_{n},\delta_{x_{0}})\le\lambda_{0}.
\]
Now set $\lambda_{0}=0$. For $c\lambda=\lambda_{0}$ and $c\in(0,1)$
we have 
\begin{eqnarray*}
\int_{M\backslash B_{R}(x_{0})}L_{\lambda}(d(x,x_{0}))d\mu(x) & \le & \frac{L_{\lambda}(R)}{\Phi(L_{\lambda_{0}}(R))}\int_{M\backslash B_{R}(x_{0})}\tilde{L}_{\lambda_{0}}(d(x,x_{0}))d\mu_{n}(x)\\
 & \le & \frac{L_{\lambda_{0}}(R)}{\Phi(L_{\lambda_{0}}(R))}\frac{L_{\lambda_{0}}(c^{-1}R)}{L_{\lambda_{0}}(R)}\le C\frac{L_{\lambda_{0}}(R)}{\Phi(L_{\lambda_{0}}(R))}
\end{eqnarray*}
for some $C>0$ depending only on $\lambda_{0},c$ and $L$. Hence
by the fact that $L(R),\Phi(R)\to\infty$ as $R\to\infty$ we conclude
\[
\lim_{R\to\infty}\limsup_{n\to\infty}\int_{M\backslash B_{R}(x_{0})}L_{\lambda}(d(x,x_{0}))d\mu_{n}=0.
\]

In order to show weak precompactness notice that $L(R)\ge1$ for $R\ge r_{0}=r_{0}(L)$
implies tightness, which is equivalent to precompactness by the classical
Prokhorov theorem. Indeed, $B_{R}(x_{0})$ is compact and for $r_{0}\le R\to\infty$
\[
\int_{M\backslash B_{R}(x_{0})}d\mu_{n}\le C\frac{L_{\lambda_{0}}(R)}{\Phi(L_{\lambda_{0}}(R))}\to0
\]
uniformly in $n$.\end{proof}
\begin{prop}
Assume $M$ is a geodesic space. Let $\pi_{opt}$ be the optimal coupling
of $(\mu_{0},\mu_{1})$ then there is a $\Pi$ supported on the geodesics
such that for $i=0,1$ 
\[
(e_{i})_{*}\Pi=\mu_{i}.
\]
Furthermore, let $\mu_{t}=(e_{t})_{*}\Pi$ then 
\[
w_{L}(\mu_{s},\mu_{t})=|s-t|w_{L}(\mu_{0},\mu_{1}).
\]
 In particular, $ $$\mathcal{P}_{L}(M)$ is a geodesic space.\end{prop}
\begin{proof}
The first part $ $follows from using the measurable selection theorem
for 
\[
(x,y)\mapsto\{\gamma:[0,1]\to M\,|\,\gamma\,\mbox{ is a geodesic from }x\mbox{ to }y\}
\]
similar to \cite{Lisini2006} in case of $p$-Wasserstein spaces.

For the second part note for $\lambda_{min}=w_{L}(\mu_{0},\mu_{1})$
\[
\int L\left(\frac{d(\gamma_{s},\gamma_{t})}{|s-t|\lambda_{min}}\right)d\Pi(\gamma)=\int L_{\lambda_{min}}\left(d(\gamma_{0},\gamma_{1})\right)d\Pi(\gamma)=1.
\]
Hence 
\[
w_{L}(\mu_{t},\mu_{s})\le|s-t|\lambda_{min}.
\]
So $t\mapsto\mu_{t}$ is absolutely continuous in $\mathcal{P}_{L}(M)$
and $|\dot{\mu}_{t}|\le\lambda_{min}$. But we also have
\[
\lambda_{min}=w_{L}(\mu_{0},\mu_{1})=\int_{0}^{1}|\dot{\mu}_{t}|dt.
\]
Therefore, $|\mu_{t}|=\lambda_{min}$ and 
\[
w_{L}(\mu_{s},\mu_{t})=\left|\int_{s}^{t}|\dot{\mu}_{r}|dr\right|=|s-t|w_{L}(\mu_{s},\mu_{t}).
\]

\end{proof}
It is also possible to define a dual problem by 
\[
\sup\{\lambda>0\,|\,\sup_{\phi\in L^{1}(\mu_{0})}\left\{ \int\phi d\mu_{0}+\int\phi^{c_{\lambda}}d\mu_{1}\right\} \ge1\}.
\]
However, we will not go into this dual problem and directly deal with
the $c_{\lambda}$-transform whenever Kantorovich potentials are needed.
Main ``problem'': the restriction property does not hold for $w_{L}$
and many results depend on (the number) $w_{L}(\mu_{0},\mu_{1})$.

The following inequality will help to show that $c_{L}$-conave functional
enjoy a similar property to star-shapedness. It will also show that
the Jacobians of the interpolation measures are positive semidefinite.
\begin{lem}
\label{lem:dist-ineq-Orlicz}If $x,y\in M$ and $z\in Z_{t}(x,y)$
for some $t\in[0,1]$. Then for all $m\in M$
\[
t^{-1}L(d(m,y))\le L_{t}(d(m,z))+t^{-1}(1-t)L(d(x,y)).
\]
Furthermore, choosing $x=m$, this becomes an equality.\end{lem}
\begin{rem*}
This extends Lemma \ref{lem:p-dist-ineq}.\end{rem*}
\begin{proof}
Since $L$ is convex and increasing
\begin{eqnarray*}
L(d(m,y)) & \le & L(t\cdot t^{-1}d(m,z)+(1-t)d(x,y))\\
 & \le & tL_{t}(d(m,z))+(1-t)L(d(x,y)).
\end{eqnarray*}
 Dividing by $t$ we get the inequality and choosing $x=m$ we see
that all inequalities are actually equalities.\end{proof}
\begin{lem}
\label{lem:inf-dist-Orlicz}Let $\eta:[0,1]\to M$ be a geodesic between
two distinct points $x$ and $y$. For $t\in(0,1]$ define 
\[
f_{t}(m):=-c_{t}(m,\eta_{t}).
\]
Then for some fixed $t\in[0,1]$ the function $h(m):=f_{t}(m)-t^{-1}f_{1}(m)$
has a minimum at $x$. \end{lem}
\begin{proof}
Using Proposition \ref{lem:p-dist-ineq} above for $t\in(0,1)$ we
have for $z=\eta_{t}\in Z_{t}(x,y)$ 
\begin{eqnarray*}
-h(m)=t^{-1}L(d(m,y))-L_{t}(d(m,z)) & \le & t^{-1}(1-t)L(d(x,y))\\
 & = & t^{-1}L(d(x,y))-L_{t}(d(x,\eta_{t}))=-h(x).
\end{eqnarray*}
\end{proof}
\begin{lem}
\label{lem:star-shaped-Orlicz}Let $X$ and $Y$ be compact subsets
of $M$ and let $t\in(0,1]$. If $\phi\in\mathcal{I}^{c_{L}}(X,Y)$
then $t^{-1}\phi\in\mathcal{I}^{c_{t}}(X,Z_{t}(X,Y))$.\end{lem}
\begin{proof}
For $t=1$ there is nothing to prove. For the rest we follow the strategy
of \cite[Lemma 5.1]{CMS2001}. Set $L_{y}(x)=L(d(x,y))$ and let $t\in(0,1]$
and $y\in Y$ and define $\phi(x):=c_{L}(x,y)=L_{y}(x)$. We claim
that the following representation holds
\[
t^{-1}L_{y}(m)=\inf_{z\in Z_{t}(X,y)}\left\{ (L_{t})_{z}(m)+\inf_{\{x\in X\,|\, z\in Z_{t}(x,y)\}}t^{-1}(1-t)L_{y}(x)\right\} .
\]
Indeed, by Lemma \ref{lem:p-dist-ineq} the left hand side is less
than or equal to the right hand side for any $z\in Z_{t}(X,y)$. Furthermore,
choosing $x=m$ we get an equality and thus showing the representation.

Now note that the claim implies that $t^{-1}\phi$ is the $\bar{c}_{p}$-transform
of the function 
\[
\psi(z)=-\inf_{\{x\in X\,|\, z\in Z_{t}(x,y)\}}t^{-1}(1-t)L_{y}(x)
\]
and therefore $t^{-1}\phi$ is $c_{t}$-concave relative to $(X,Z_{t}(X,y))$.
Since $\mathcal{I}^{c_{t}}(X,Z_{t}(X,y))\subset\mathcal{I}^{c_{t}}(X,Z_{t}(X,Y))$
we see that each $t^{-1}L_{y}$ is in $\mathcal{I}^{c_{t}}(X,Z_{t}(X,Y))$.

It remains to show that for an arbitrary $c_{L}$-concave function
$\phi$ and $t\in(0,1]$ the function $t^{-1}\phi$ is $c_{t}$-concave
relative to $(X,Z_{t}(X,Y))$. Since $\phi=\phi^{c_{L}\bar{c}_{L}}$
we have 
\[
t^{-1}\phi(x)=\inf_{y}t^{-1}L(d(x,y))-t^{-1}\phi^{c_{L}}(y).
\]
But each function 
\[
\psi_{y}(x)=t^{-1}L_{y}(x)-t^{-1}\phi^{c}(y)
\]
is $c_{p}$-concave relative to $(X,Z_{t}(X,Y))$ and $\phi$ is proper,
thus also the infimum is $c_{t}$-concave relative to $(X,Z_{t}(X,Y))$,
i.e. $t^{-1}\phi\in\mathcal{I}^{c_{t}}(X,Z_{t}(X,Y))$.  
\end{proof}

\subsection*{Orlicz-Wassterstein spaces on Finsler manifolds}

\subsubsection*{Technical ingredients}

For simplicity, assume throughout the section that $L$ is smooth
away from $0$. 

For $L_{x}=L(d(x,\cdot))$ and $x\ne y$ 
\[
\nabla L_{x}(y)=l(d(x,y))\nabla d_{x}(y).
\]

Define
\[
\nabla^{L}\phi:=\frac{l^{-1}(|\nabla\phi|)}{|\nabla\phi|}\nabla\phi.
\]
Note that for $v\in T_{x}M$ with $|v|=1$ and $r\ge0$ 
\[
\nabla\phi(x)=l(r)v
\]
iff
\[
\nabla^{L}\phi=rv.
\]
We also use the abbreviation 
\[
\nabla^{\lambda}\phi=\nabla^{L_{\lambda}}\phi.
\]

It is easy to see that under our assumptions that $\phi\mapsto\nabla^{L}\phi$
is continuous and (as) smooth (as $L$) wherever $\nabla\phi(x)\ne0$.

Similar to the $c_{p}$-case we will use the abbreviation $\mathcal{K}_{x}^{L}d\phi_{x}$
(resp. $\mathcal{K}_{x}^{\lambda}d\phi_{x}$) for $\nabla^{L}\phi(x)$
(resp. $\nabla^{\lambda}\phi(x)$). As mentioned above, this can also
be seen as a Legendre transform from $T^{*}M$ to $TM$. 
\begin{lem}
[Cut locus charaterization] \label{lem:cut-orlicz}If $y\ne x$ is
a cut point of $x$, then $f(z):=L(d(x,y))$ satisfies
\[
\liminf_{v\to0\in T_{x}M}\frac{f(\xi_{v}(1))+f(\xi_{v}(-1))-2f(x)}{F(v)^{2}}=-\infty
\]
where $\xi_{v}:[-1,1]\to M$ is the geodesic with $\dot{\xi}_{v}(0)=v$.\end{lem}
\begin{proof}
The proof follows in the same fashion as Lemma \ref{lem:cut}. We
will show the necessary adjustments.

As above, let's first assume there are two distinct unit speed geodesics
$\eta,\zeta:[0,d(x,y)]\to M$ from $x$ to $y$ and let $v=\dot{\zeta}(0)$
and $w=\dot{\eta}(0)$. For fixed small $\epsilon>0$ set $y_{\epsilon}=\eta(d(x,y)-\epsilon)$
then $y_{\epsilon}\notin\operatorname{Cut}(x)\cup\{x\}$ and using
the first variation formula we get for $t>0$
\begin{eqnarray*}
f(\xi_{v}(-t))-f(x) & \le & L(d(\xi_{v}(-t),y_{\epsilon})+\epsilon)-L(d(x,y_{\epsilon})+\epsilon)\\
 & = & tl(d(x,y_{\epsilon})+\epsilon)g_{\dot{\eta}(0)}(v,\dot{\eta}(0))+\mathcal{O}(t^{2})\\
 & = & tl(d(x,y))g_{\dot{\eta}(0)}(v,\dot{\eta}(0))+\mathcal{O}(t^{2}).
\end{eqnarray*}
The term $\mathcal{O}(t^{2})$ is ensured by smoothness of $\xi_{v}$
and by the facts that $x\ne y_{\epsilon}$ and that $L(d(\cdot,\cdot))$
is bounded in a neighborhood $(x,y)$. We also get by Taylor formula
\[
f(\xi_{v}(t))-f(x)=L(d(x,y)-t)-L(d(x,y))=-tl(d(x,y))+\mathcal{O}(t^{2}).
\]
Combining these two facts with $g_{w}(v,w)<1$ ($\eta$ and $\xi$
are distinct), we get 
\[
\frac{f(\xi_{v}(-t))+f(\xi_{v}(t))-2f(x)}{t^{2}}\le\frac{1-g_{w}(v,w)}{t}l(d(x,y))+t^{-2}\mathcal{O}(t^{2})\to-\infty\;\mbox{ as }t\to0.
\]

For the conjugate point case, we use the same construction and notation
as in the proof of Lemma \ref{lem:cut}. Note that 
\begin{eqnarray*}
\lim_{s\to0}\frac{L(\mathcal{L}(\sigma_{s}))+L(\mathcal{L}(\sigma_{-s}))-2L(\mathcal{L}(\sigma_{0}))}{s^{2}} & = & \bigg(l(\mathcal{L}(\sigma_{0}))\frac{\partial^{2}}{\partial s^{2}}\mathcal{L}(\sigma_{s})\bigg|_{s=0}\\
 &  & \quad+l'(\mathcal{L}(\sigma_{0}))\left(\frac{\partial L(\sigma_{s})}{\partial s}\bigg|_{s=0}\right)^{2}\bigg)\\
 & \le & l(d(x,y))\bigg(-2\epsilon g_{\dot{\eta}}(v,v)/d(x,y)\\
 &  & \qquad+\epsilon^{2}\left\{ \mathcal{T}_{\dot{\eta}(0)}(v)/d(x,y)+I(V,V))^{2}\right\} \bigg)\\
 &  & +l'(d(x,y))F(v).
\end{eqnarray*}
Using the fact that $f(\xi_{v}(\epsilon s))\le L(\mathcal{L}(\sigma_{s}))$
we obtain
\begin{eqnarray*}
\liminf_{s\to0}\frac{f(\xi_{v}(\epsilon s))+f(\xi_{v}(-\epsilon s))-2f(x)}{\epsilon^{2}s^{2}} & \le & \liminf_{s\to0}\frac{L(\mathcal{L}(\sigma_{s}))+L(\mathcal{L}(\sigma_{-s}))-2L(\mathcal{L}(\sigma_{0}))}{\epsilon^{2}s^{2}}\\
 & \le & l(d(x,y))\bigg(-2\epsilon^{-1}g_{\dot{\eta}}(v,v)/d(x,y)\\
 &  & \hspace{1em}+\mathcal{T}(v)/d(x,y)+d(x,y)I(V,V)\bigg)\\
 &  & +l'(d(x,y))F(v)^{2}.
\end{eqnarray*}
Letting $\epsilon$ tend to zero completes the proof.
\end{proof}

\subsubsection*{The Brenier-McCann-Ohta solution}
\begin{lem}
\label{lem:brenier-orlicz}Let $\phi:M\to\mathbb{R}$ be a $c_{L}$-concave
function. If $\phi$ is differentiable at $x$ then $\partial^{c_{L}}\phi(x)=\{exp_{x}(\nabla^{L}(-\phi)(x))\}$.
Moreover, the curve $\eta(t):=exp_{x}(t\nabla^{L}(-\phi)(x))$ is
a unique minimal geodesic from $x$ to $exp_{x}(\nabla^{L}(-\phi)(x))$.\end{lem}
\begin{rem*}
See also \cite[Theorem 13]{McCann2001} for the Riemannian case.\end{rem*}
\begin{proof}
Let $y\in\partial^{c_{L}}\phi(x)$ be arbitrary and define $f(z):=c_{L}(z,y)=L(d(z,y))$.
By definition of $\partial^{c_{L}}\phi(x)$ we have for any $v\in T_{x}M$
\[
f(exp_{x}v)\ge\phi^{c_{L}}(y)+\phi(exp_{x}v)=f(x)-\phi(x)+\phi(exp_{x}v)=f(x)+d\phi_{x}(v)+o(F(v)).
\]

Now let $\eta:[0,d(x,y)]\to M$ be a minimal unit speed geodesic from
$x$ to $y$. Given $\epsilon>0$, set $y_{\epsilon}=\eta(d(x,y)-\epsilon)$
and note that $\eta|_{[0,d(x,y)-\epsilon]}$ does not cross the cut
locus of $x$. By the first variation formula we have
\begin{eqnarray*}
f(exp_{x}v)-f(x) & \le & L\left(d(exp_{x}v,y_{\epsilon})+\epsilon\right)-L\left(d(x,y_{\epsilon})+\epsilon\right)\\
 & = & -l\left(d(x,y_{\epsilon})+\epsilon\right)g_{\dot{\eta}(0)}(v,\dot{\eta}(0))+o(F(v)).\\
 & = & -l(d(x,y))\mathcal{L}_{x}^{-1}(\dot{\eta}(0))(v)+o(F(v)).
\end{eqnarray*}
Therefore, $d\phi_{x}(v)\le-l(d(x,y))\mathcal{L}_{x}^{-1}(\dot{\eta}(0))(v)$
for all $v\in T_{x}M$ and thus $\nabla(-\phi)=l(d(x,y))\cdot\dot{\eta}(0)$.,
i.e. $\nabla^{L}(-\phi)=d(x,y)\cdot\dot{\eta}(0)$. In addition, note
that $\eta(t)=exp_{x}(t\nabla^{L}(-\phi)(x))$, which is uniquely
defined.\end{proof}
\begin{lem}
\label{lem:orlicz-potential}Let $t\mapsto\mu_{t}$ be a geodesic
between $\mu_{0}$ and $\mu_{1}$, i.e. $w_{L}(\mu_{0},\mu_{t})=t\lambda$.
If $\mu_{0}$ is absolutely continuous and the unique $\phi_{t}$
the Kantorovich potential of $(\mu_{0},\mu_{t})$ w.r.t. $L_{t\lambda}$
such that $\phi_{t}(x_{0})=0$. Then $\phi_{t}=t^{-1}\phi$.\end{lem}
\begin{proof}
For $x\ne y\in\partial^{c_{\lambda}}\phi_{1}(x)$ define $x_{t}=exp_{x}(t\nabla^{L}(-\phi)(x))$.
Since $x_{t}\in\partial^{c_{t\lambda}}\phi_{t}(x)$, we have for $t\in(0,1]$
\begin{eqnarray*}
x_{t} & = & exp_{x}\left(t\nabla^{L_{\lambda}}(-\phi)(x)\right)\\
 & = & exp_{x}\left(\frac{t\cdot l_{\lambda}^{-1}(t\cdot t^{-1}|\nabla(-\phi)|(x))}{|\nabla(-\phi)|(x)}\nabla(-\phi)(x)\right)\\
 & = & exp_{x}\left(\frac{l_{t\lambda}^{-1}(t^{-1}|\nabla(-\phi)|(x))}{|\nabla(-\phi)|(x)}\nabla(-\phi)(x)\right)\\
 & = & exp_{x}\left(\frac{l_{t\lambda}^{-1}(|\nabla(-t^{-1}\phi)|(x))}{|\nabla(-t^{-1}\phi)|(x)}\nabla(-t^{-1}\phi)(x)\right).
\end{eqnarray*}
Since $t^{-1}\phi$ is $c_{t}$-concave and $t^{-1}\phi(x_{0})=0$,
uniqueness implies $\phi_{t}=t^{-1}\phi$.\end{proof}
\begin{rem*}
Note that this agrees with the cases $L(r)=r^{p}/p$: Assume for simplicity
that $w_{p}(\mu_{0},\mu_{1})=1$ then $\phi^{L}=\phi^{c_{p}}$ and
$L_{t}=t^{p}d^{p}/p$. Hence 
\begin{eqnarray*}
\phi_{t}^{c_{t}}(y) & = & \inf t^{p}\frac{d^{p}(x,y)}{p}-t^{-1}\phi(x)\\
 & = & t^{-p}\inf\frac{d^{p}(x,y)}{p}-t^{p-1}\phi(x)=t^{-p}(t^{p-1}\phi)^{c_{p}}(y)
\end{eqnarray*}
Thus up to a factor the interpolation potentials are the same (recall
that $t^{p-1}\phi$ gives the potential of $(\mu_{0},\mu_{t})$ w.r.t.
$c_{p}$).
\end{rem*}
The next results follow using exactly the same arguments as for $c_{p}$.
\begin{lem}
Let $\mu_{0}$ and $\mu_{1}$ be two probability measures on $M$.
Then there exists a unique (up to constant) $c_{L}$-concave function
$\phi$ that solves the Monge-Kantorovich problem w.r.t. $L$. Moreover,
if $\mu_{0}$ is absolutely continuous, then the vector field $\nabla^{L}(-\phi)$
is unique among such minimizers.\end{lem}
\begin{rem*}
At this point we do not work with $\mathcal{P}_{L}(M)$ directly.
However all statements make sense also for $L_{\lambda}$ and any
$\lambda>0$ and we will see later that Lemma \ref{lem:star-shaped-Orlicz}
can be used to show that the interpolation inequality in Theorem \ref{thm:interp-Orlicz}
is actually an interpolation inequality w.r.t. the geodesic $t\mapsto\mu_{t}$
in $\mathcal{P}_{L}(M)$ if the function $L_{\lambda}$ is used with
$\lambda=w_{L}(\mu_{0},\mu_{1})$.\end{rem*}
\begin{thm}
Let $\mu_{0}$ and $\mu_{1}$ be two probability measure on $M$ and
assume $\mu_{0}$ is absolutely continuous with respect to $\mu$.
Then there is a $c_{L}$-concave function $\phi$ such that $\pi=(\operatorname{Id}\times\mathcal{F})_{*}\mu_{0}$
is the unique optimal coupling of $(\mu_{0},\mu_{1})$ w.r.t. $L$,
where $\mathcal{F}(x)=exp_{x}(\nabla^{L}(-\phi))$. Moreover, $\mathcal{F}$
is the unique optimal transport map from $\mu_{0}$ to $\mu_{1}$.\end{thm}
\begin{cor}
If $\phi$ is $c_{L}$-concave and $\mu_{0}$ is absolutely continuous,
then the map $\mathcal{F}(x):=exp_{x}(\nabla^{L}(-\phi))$ is the
unique optimal transport map from $\mu_{0}$ to $\mathcal{F}_{*}\mu_{0}$
w.r.t. the cost function $c_{L}(x,y)=L(d(x,y))$.
\end{cor}

\begin{cor}
Assume $\mu_{0}$ is absolutely continuous and $\phi$ is $c_{\lambda}$-concave
with $\lambda=w_{L}(\mu_{0},(\mathcal{F}_{1})_{*}\mu_{0})$ where
$\mathcal{F}_{t}(x):=exp_{x}(\nabla^{\lambda}(-t^{-1}\phi))$, then
$\mathcal{F}_{t}$ is the unique optimal transport map from $\mu_{0}$
to $\mu_{t}=(\mathcal{F}_{t})_{*}\mu_{0}$ w.r.t. $L_{\lambda}$ and
$t\mapsto\mu_{t}$ is a constant geodesic from $\mu_{0}$ to $\mu_{1}$
in $\mathcal{P}_{L}(M)$.\end{cor}
\begin{rem*}
We will see in Lemma \ref{lem:abs-interp-Orlicz} below that the interpolation
measures are absolutely continuous if $\mu_{0}$ and $(\mathcal{F})_{t*}\mu_{0}$
are.\end{rem*}
\begin{proof}
We only need to show that 
\[
w_{L}(\mu_{s},\mu_{t})\le|s-t|w_{L}(\mu_{0},\mu_{1}).
\]
Let $\pi$ be the plan on $\operatorname{Geo}(M)=\{\gamma:[0,1]\to M\,|\,\gamma\mbox{ is a geodesic in }M\}$
$ $give by $\mu_{0}$, the map $\mathcal{F}_{1}$ and the unique
geodesic connecting $\mu$-almost every $x\in M$ to a point $\mathcal{F}_{1}(x)$
(existence follows from \cite[Proof of Prop. 4.1]{Lisini2006}, see
also \cite[Chapter 7]{Villani2009}), in particular, $\mu_{t}=(\mathcal{F}_{t})_{*}\mu_{0}$.
We also have 
\[
\int L\left(\frac{d(\gamma_{0},\gamma_{1}}{\lambda}\right)d\pi(\gamma)=1
\]
for $\lambda=w_{L}(\mu_{0},\mu_{1})$ by definition $w_{L}$. Since
$(e_{s},e_{t})_{*}\pi$ is a plan between $\mu_{s}$ and $\mu_{t}$
for $s,t\in[0,1]$ we have 
\[
\int L\left(\frac{d(\gamma_{s},\gamma_{t})}{|t-s|\lambda}\right)d\pi(\gamma)=\int L\left(\frac{d(\gamma_{0},\gamma_{0})}{\lambda}\right)d\pi(\gamma)=1.
\]
Therefore, $w_{L}(\mu_{s},\mu_{t})\le|t-s|\lambda$.
\end{proof}

\subsubsection*{Almost Semiconcavity of Orlicz-concave functions}

The proof of almost semiconcavity of $c_{L}$-concave functions follows
along the lines of the proof of Theorem \ref{thm:semiconc} by noticing
that $\phi_{s}=s^{-1}\phi$ will be $c_{s}$-concave instead of $c_{L}$-concave,
i.e. the type of concavity changes since the ``distance changes''.
\begin{thm}
\label{thm:Orlicz-semiconc}Let $\phi$ be a $c_{L}$-concave function.
Let $\Omega_{id}$ be the the points $x\in M$ where $\phi$ is differentiable
and $d\phi_{x}=0$, or equivalently $\partial^{c_{L}}\phi(x)=\{x\}$.
Then $\phi$ is locally semiconcave on an open subset $U\subset M\backslash\Omega_{id}$
of full measure (relative to $M\backslash\Omega_{id}$). In particular,
it is second order differentiable almost everywhere in $U$.
\end{thm}

\subsection*{Proof of the interpolation inequality in the Orlicz case}
\begin{thm}
[Volume distortion for $L$] Let $x\ne y$ with $y\notin\operatorname{Cut}(x)$
and $\eta$ be the unique minimal geodesic from $x$ to $y$. For
$t\in(0,1]$ define $f_{t}(z)=-L_{t}(d(z,\eta(t))$. 

Then we have 
\begin{eqnarray*}
\mathfrak{v}_{t}^{<}(x,y) & = & \mathbf{D}\left[d(exp_{x})_{\nabla^{L_{t}}f_{t}(x)}\circ[d(exp_{x})_{\nabla^{L}f_{1}(x)}]^{-1}\right]\\
\mathfrak{v}_{t}^{>}(x,y) & = & (1-t)^{-n}\mathbf{D}\left[d(exp_{x}\circ\mathcal{K}_{x}^{t})_{d(t^{-1}f_{1})_{x}}\circ[d\left(d(t^{-1}f_{1})\right)_{x}-d\left(df_{t}\right)_{x}]\right].
\end{eqnarray*}
\end{thm}
\begin{rem*}
The statements hold equally if one take $L_{\lambda}$ and $L_{t\lambda}$,
they only depend on the smoothness of $L$.\end{rem*}
\begin{proof}
Recall Theorem \ref{thm:Volume-d2} and the function $g_{t}(z)=-d^{2}(x,\eta(t))/2$. 

We have for $L_{\eta(t)}=L_{t}(d(\cdot,\eta(t))$ 
\[
\nabla L_{\eta(t)}(x)=l_{t}(d(x,\eta(t))\nabla d(x,\eta(t))
\]
and thus 
\[
\nabla^{t}f_{t}(z)=l_{t}^{-1}(l_{t}(d(z,\eta(t))\nabla(-d(z,\eta(t))=\nabla g_{t}(z)
\]
which implies the first equation.

For the second part note that for (see calculations in the proof of
Lemma \ref{lem:orlicz-potential})

\begin{eqnarray*}
\mathcal{K}_{z}^{t}(d(t^{-1}f_{1})_{z}) & = & \frac{l_{t}^{-1}(t^{-1}|\nabla f_{1}|(z))}{|\nabla f_{1}|(z)}\nabla f_{1}(z)\\
 & = & t\frac{l^{-1}(|\nabla f_{1}|(z))}{|\nabla f_{1}|(z)}\nabla f_{1}(z)\\
 & = & t\nabla^{L}f_{1}(z)=\mathcal{L}_{z}(d(tg_{1})_{z})
\end{eqnarray*}
 and hence 
\begin{eqnarray*}
\mathfrak{v}_{t}^{>}(x,y) & = & (1-t)^{-n}\mathbf{D}\left[d\left(exp\circ\mathcal{L}\circ(d(tg_{1})_{z})\right)\right]\\
 & = & (1-t)^{-n}\mathbf{D}\left[d(exp\circ\mathcal{K}^{t})_{d(t^{-1}f_{1})_{x}}\circ d\left(d(t^{-1}f_{1})\right)_{x}\right].
\end{eqnarray*}
We have $d(f_{t})_{x}=d(t^{-1}f_{1})_{x}$. Indeed, since $l_{t}(r)=t^{-1}l(t^{-1}r)$
and $d(d(\cdot,\eta(t))_{x}=d(d(\cdot,y))_{x}$ 
\begin{eqnarray*}
-d(f_{t})_{x} & = & d(L_{t}(d(\cdot,\eta(t)))_{x}\\
 & = & l_{t}(d(x,\eta(t))d(d(\cdot,\eta(t))_{x}\\
 & = & t^{-1}l(t^{-1}td(x,y))d(d(\cdot,y))_{x}=-d(t^{-1}f_{1})_{x}
\end{eqnarray*}
Similar to \cite[Proof of 3.2]{Ohta2009} it suffices to show that
\[
d(exp_{x}\circ\mathcal{K}_{x}^{t})_{d(f_{t})_{x}}\circ d(t^{-1}df_{t})_{x}=0.
\]
Now since $\nabla f_{t}(z)=l_{t}(d(z,\eta(t)))\nabla d_{\eta(t)}(z)$
we get in a neighborhood $U$ of $x$ not containing $\eta(t)$. 
\begin{eqnarray*}
\mathcal{K}_{z}^{t}(d(f_{t})_{z}) & = & \nabla^{L_{t}}(t^{-1}f_{t})(z)\\
 & = & l_{t}^{-1}(l_{t}(d(z,\eta(t))))\nabla d_{\eta(t)}(z)\\
 & = & \mathcal{L}_{z}(d(g_{t})_{z})
\end{eqnarray*}
and thus the function $D:U\to M$ defined as 
\begin{eqnarray*}
D(z)=exp_{z}\circ\mathcal{K}_{z}(d(f_{t})_{z}) & = & exp_{z}\circ\mathcal{L}_{z}(d(g_{t})_{z})\\
 & = & \eta(t).
\end{eqnarray*}
is constant in a neighborhood of of $x$. This immediately implies
$dL_{x}=0$.\end{proof}
\begin{prop}
\label{prop:eqns-Orlicz}Let $\phi:M\to\mathbb{R}$ be a $c_{L}$-concave
function and define $\mathcal{F}(z)=exp_{z}(\nabla^{L}(-\phi)(z))$
at all point of differentiability of $\phi$. Fix some $x\in M$ such
that $\phi$ is second order differentiable at $x$ and $d\phi_{x}\ne0.$
Then the following holds:
\begin{enumerate}
\item $y=\mathcal{F}(x)$ is not a cut point of $x$.
\item The function $h(z)=c_{L}(z,y)-\phi(z)$ satisfies $dh_{x}=0$ and
\[
\left(\frac{\partial^{2}h}{\partial x^{i}\partial x^{j}}(x)\right)\ge0
\]
in any local coordinate system $(x^{i})_{i=1}^{n}$ around $x$.
\item Define $f_{y}(z):=-c_{L}(z,y)$ and
\[
d\mathcal{F}_{x}:=d(exp_{x}\circ\mathcal{K}_{x}^{L})_{d(-\phi)_{x}}\circ\left[d(d(-\phi))_{x}-d(df_{y})_{x}\right]:T_{x}M\to T_{y}M
\]
where the vertical part of $T_{d(-\phi)_{x}}(T^{*}M)$ and $T_{d(-\phi)_{x}}(T^{*}M)$
are identified. Then the following holds for all $v\in T_{x}M$
\[
\sup\left\{ \left|u-d\mathcal{F}_{x}(v)\right|\,|\, exp_{y}u\in\partial^{c_{L}}\phi(\exp_{x}y),|u|=d(y,\exp_{y}u)\right\} =o(|v|).
\]

\end{enumerate}
\end{prop}
\begin{proof}
The proof follows without any change from the proof of Proposition
\ref{prop:eqns} but using Lemma \ref{lem:cut-orlicz} instead and
the fact that $y\notin\operatorname{Cut}(x)\cup\{x\}$ implies that
$f_{y}$ is $C^{\infty}$ at $x$ and $\nabla^{L}f_{y}(x)=\nabla^{L}\phi(x)$.
\end{proof}
Similarly the Jacobian equation holds:
\begin{prop}
\label{prop:jac-eq-Orlicz}Let $\mu_{0}$ and $\mu_{1}$ be absolutely
continuous measure with density $f_{0}$ and $f_{1}$ and $\lambda=w_{L}(\mu_{0},\mu_{1})$.
Also assume that there are open set $U_{i}$ with compact closed $X=\bar{U}_{0}$
and $Y=\bar{U}_{1}$ such that $\operatorname{supp}\mu_{i}\subset U_{i}$.
Let $\phi$ be the unique $c_{\lambda}$-concave Kantorovich potential
and define $\mathcal{F}(z)=exp_{z}(\nabla^{\lambda}(-\phi)(z))$.
Then $\mathcal{F}$ is injective $\mu_{0}$-almost everywhere and
for $\mu_{0}$-almost every $x\in M\backslash\Omega_{id}$ 
\begin{enumerate}
\item The function $h(z)=c_{\lambda}(z,\mathcal{F}(z))-\phi(z)$ satisfies
\[
\left(\frac{\partial^{2}h}{\partial x^{i}\partial x^{j}}(x)\right)>0
\]
in any local coordinate system $(x^{i})_{i=0}^{n}$ around $x$.
\item In particular, $\mathbf{D}[d\mathcal{F}_{x}]>0$ holds for the map
$d\mathcal{F}_{x}:T_{x}M\to T_{\mathcal{F}(x)}M$ defined as above
and
\[
\lim_{r\to0}\frac{\mu(\partial^{c_{\lambda}}\phi(B_{r}^{+}(x)))}{\mu(B_{r}^{+}(x))}=\mathbf{D}[d\mathcal{F}_{x}]
\]
and 
\[
f(x)=g(\mathcal{F}(x))\mathbf{D}[d\mathcal{F}_{x}].
\]

\end{enumerate}
\end{prop}
\begin{rem*}
defining $d\mathcal{F}_{x}=\operatorname{Id}$ for points $x$ of
differentiability of $\phi$ with $d\phi_{x}=0$ we see that the second
statement above holds $\mu$-a.e.\end{rem*}
\begin{proof}
Similar to Proposition \ref{prop:jac-eq}, the proof follows without
any change from \cite[Theorem 5.2]{Ohta2009}, see also \cite[Chapter 11]{Villani2009}.\end{proof}
\begin{thm}
\label{thm:interp-Orlicz}Let $\phi:M\to\mathbb{R}$ be a $c_{L}$-concave
function and $x\in M$ such that $\phi$ is second order differentiable
with $d\phi_{x}\ne0$. For $t\in(0,1]$, define $y_{t}:=exp_{x}(\nabla^{t}(-t^{-1}\phi))$,
$f_{t}(z)=-c_{t}(z,y_{t})$ and $\mathbf{J}_{t}(x)=\mathbf{D}[d(\mathcal{F}_{t})_{x}]$
where 
\[
d(\mathcal{F}_{t})_{x}:=d(exp_{x}\circ\mathcal{K}_{x}^{t})_{d(-t^{-1}\phi)_{x}}\circ\left[d(d(-t^{-1}\phi))_{x}-d(d(f_{t}))_{x}\right]:T_{x}M\to T_{y_{t}}M.
\]
Then for any $t\in(0,1)$
\[
\mathbf{J}_{t}(x)^{\nicefrac{1}{n}}\ge(1-t)\mathfrak{v}_{t}^{>}(x,y_{1})^{\nicefrac{1}{n}}+t\mathfrak{v}_{t}^{<}(x,y_{1})^{\nicefrac{1}{n}}\mathbf{J}_{1}(x)^{\nicefrac{1}{n}}.
\]
\end{thm}
\begin{rem*}
The proof is based on the proof of \cite[Proposition 5.3]{Ohta2009}
but is notationally slightly more involved then the proof of Theorem
\ref{thm:interp}. \end{rem*}
\begin{proof}
Note first that 
\[
d(d(-t^{-1}\phi))_{x}-d(df_{t})_{x}=\left\{ d(d(-t^{-1}\phi))_{x}-d(d(t^{-1}f_{1}))_{x}\right\} +\left\{ d(d(t^{-1}f_{1}))_{x}-d(df_{t})_{x}\right\} 
\]
and 
\[
d(f_{t})_{x}=d(-t^{-1}\phi)_{x}=d(-t^{-1}f_{1})_{x}.
\]
Now define $\tau_{s}:T^{*}M\to T^{*}M$ as $\tau_{s}(v)=s^{-1}v$
and note for $\nabla\phi(x)\ne0$
\begin{eqnarray*}
\mathcal{K}_{x}^{t}(t^{-1}d\phi_{x}) & = & \frac{l_{t\lambda}^{-1}(|\nabla t^{-1}\phi(x)|)}{|\nabla t^{-1}\phi(x)|}\nabla t^{-1}\phi(x)\\
 & = & t\frac{l_{\lambda}(|\nabla\phi(x)|)}{|\nabla\phi(x)|}\nabla\phi(x)=t\mathcal{K}_{x}^{L}(d\phi_{x})
\end{eqnarray*}
and thus
\[
\mathcal{K}_{x}^{t}\circ\tau_{t}\circ(\mathcal{K}_{x}^{L})^{-1}=t\operatorname{Id}_{T_{x}M}
\]
which implies 
\begin{eqnarray*}
 &  & d(exp_{x}\circ\mathcal{K}_{x}^{t})_{d(-t^{-1}\phi)_{x}}\circ\left(d(d(-t^{-1}\phi))_{x}-d(d(t^{-1}f_{1}))_{x}\right)\\
 & = & d(exp_{x}\circ\mathcal{K}_{x}^{t})_{d(-t^{-1}\phi)_{x}}\circ d(\tau_{t})_{d(-\phi)_{x}}\circ\left[d(d(-\phi))_{x}-d(d(f_{1}))_{x}\right]\\
 & = & d(exp_{x}\circ\mathcal{K}_{x}^{t})_{d(-t^{-1}\phi)_{x}}\circ d(\tau_{t})_{d(-\phi)_{x}}\circ\left[d(exp_{x}\circ\mathcal{K}_{x}^{L})_{d(-\phi)_{x}}\right]^{-1}\circ d(\mathcal{F}_{1})\\
 & = & d(exp_{x})_{\nabla^{t}(-t^{-1}\phi)_{x}}\circ d(\mathcal{K}_{x}\circ\tau_{t}\circ\mathcal{K}_{x}^{-1})_{\nabla^{L}(-\phi)_{x}}\circ[d(exp_{x})_{\nabla^{L}(-\phi)_{x}}]^{-1}\circ d(\mathcal{F}_{1})\\
 & = & t\cdot d(exp_{x})_{\nabla^{t}(-t^{-1}\phi)_{x}}\circ[d(exp_{x})_{\nabla^{L}(-\phi)_{x}}]^{-1}\circ d(\mathcal{F}_{1}).
\end{eqnarray*}
where we identified $T_{\nabla^{t}(-t^{-1}\phi)(x)}(T_{x}M)$ with
$T_{\nabla^{L}(-\phi)(x)}(T_{x}M)$ to get the last inequality (remember
$t\nabla^{t}(-t^{-1}\phi)=\nabla^{L}(-\phi)$).

Because $\mathbf{D}$ is concave we get 
\begin{eqnarray*}
 &  & \mathbf{J}_{t}(x)^{\nicefrac{1}{n}}=\mathbf{D}[d(\mathcal{F}_{t})_{x}]^{\nicefrac{1}{n}}\\
 & = & \mathbf{D}\bigg[d(exp_{x}\circ\mathcal{K}_{x}^{t})_{d(-t^{-1}\phi)_{x}}\circ\left[d(d(t^{-1}f_{1}))_{x}-d(df_{t})_{x}\right]\\
 &  & +\, d(exp_{x}\circ\mathcal{K}_{x}^{t})_{d(-t^{-1}\phi)_{x}}\circ\left(d(d(-t^{-1}\phi))_{x}-d(d(t^{-1}f_{1}))_{x}\right)\bigg]^{\nicefrac{1}{n}}\\
 & = & \mathbf{D}\bigg[d(exp_{x}\circ\mathcal{K}_{x}^{t})_{d(-t^{-1}\phi)_{x}}\circ\left(d(d(t^{-1}f_{1}))_{x}-d(df_{t})_{x}\right)\\
 &  & +\, t\cdot d(exp_{x})_{\nabla^{t}(-t^{-1}\phi)_{x}}\circ[d(exp_{x})_{\nabla^{L}(-\phi)_{x}}]^{-1}\circ d(\mathcal{F}_{1})\bigg]^{\nicefrac{1}{n}}\\
 & \ge & (1-t)\mathbf{D}\bigg[(1-t)^{-1}d(exp_{x}\circ\mathcal{K}_{x}^{t})_{d(-t^{-1}\phi)_{x}}\circ\left(d(d(t^{-1}f_{1}))_{x}-d(df_{t})_{x}\right)\bigg]^{\nicefrac{1}{n}}\\
 &  & +\, t\mathbf{D}\bigg[d(exp_{x})_{\nabla^{t}(-t^{-1}\phi)_{x}}\circ[d(exp_{x})_{\nabla^{L}(-\phi)_{x}}]^{-1}\circ d(\mathcal{F}_{1})\bigg]^{\nicefrac{1}{n}}\\
 & = & (1-t)\mathfrak{v}_{t}^{>}(x,y_{1})^{\nicefrac{1}{n}}+t\mathfrak{v}_{t}^{<}(x,y_{1})^{\nicefrac{1}{n}}\mathbf{J}_{1}(x)^{\nicefrac{1}{n}}.
\end{eqnarray*}

\end{proof}

Combing this with Lemma \ref{lem:inj} (see remark after that lemma)
and Lemma \ref{lem:min-diff-Orlicz-1} below we get similar to Lemma
\ref{lem:abs-interp} and \cite[6.2]{Ohta2009}:
\begin{lem}
\label{lem:abs-interp-Orlicz}Given two absolutely continuous measures
$\mu_{i}=\rho_{i}\mu$ on $M$, let $\phi$ be the unique $c_{\lambda}$-concave
optimal Kantorovich potential with $\lambda=w_{L}(\mu_{0},\mu_{1})$.
Define $\mathcal{F}_{t}(x):=exp_{x}(\nabla^{t\lambda}(-t^{-1}\phi))$
for $t\in(0,1]$. Then $\mu_{t}=\rho_{t}d\mu$ is absolutely continuous
for any $t\in[0,1]$.\end{lem}
\begin{proof}
By Lemma \ref{lem:inj} the map $\mathcal{F}_{t}$ is injective $\mu_{0}$-almost
everywhere. Let $\Omega_{id}$ be the points $x\in M$ of differentiability
of $\phi$ with $d\phi_{x}=0$. Then 
\[
\mu_{t}\big|_{\Omega_{id}}=(\mathcal{F}_{t})_{*}(\mu_{0}\big|_{\Omega_{id}})=\mu_{0}\big|_{\Omega_{id}}.
\]
By Theorem \ref{thm:semiconc} the potential $\phi$ is second order
differentiable in a subset $\Omega\subset M\backslash\Omega_{id}$
of full measure. In addition, $\mathbf{D}[d(\mathcal{F}_{1})]>0$
for all $x\in\Omega$ (see Proposition \ref{prop:jac-eq}) and $\mathcal{F}_{t}$
is continuous in $\Omega$ for any $t\in[0,1]$. The map $d(\mathcal{F}_{t})_{x}:T_{x}M\to T_{\mathcal{F}_{t}(x)}M$
defined in Proposition \ref{prop:eqns} as 
\[
d(\mathcal{F}_{t})_{x}:=d(exp_{x}\circ\mathcal{K}_{x}^{t})_{d(-t^{-1}\phi)}\circ\left[d(d(-t^{-1}\phi))_{x}-d(d(f_{t})_{x}\right]
\]
where $f_{t}(z):=-c_{t\lambda}(z,\mathcal{F}_{t}(x))$ for $t\in(0,1]$.
Also note that for $x\in\Omega$
\[
d(d(-t^{-1}\phi))_{x}-d(df_{t})_{x}=\left\{ d(d(-t^{-1}\phi))_{x}-d(d(t^{-1}f_{1}))_{x}\right\} +\left\{ d(d(t^{-1}f_{1}))_{x}-d(f_{t})_{x}\right\} .
\]
Which implies $\mathbf{D}[d(\mathcal{F}_{t})_{x}]>0$ because $\mathbf{D}[d(\mathcal{F}_{1})_{x})>0$
and the lemma below.

The result then immediately follows by \cite[Claim 5.6]{CMS2001}.\end{proof}
\begin{lem}
\label{lem:min-diff-Orlicz-1}Let $y\notin\operatorname{Cut}(x)\cup\{x\}$
and $\eta:[0,1]\to M$ be the unique minimal geodesic from $x$ to
$y$. Define 
\[
f_{t}(z)=-c_{t}(z,\eta(t)).
\]
Then the function $h(z)=t^{-1}f_{1}(z)-f_{t}(z)$ satisfies
\[
\left(\frac{\partial^{2}h}{\partial x^{i}\partial x^{j}}(x)\right)\ge0
\]
in any local coordinate system around $x$.\end{lem}
\begin{proof}
This follows directly from \ref{lem:inf-dist-Orlicz}.
\end{proof}
Using this interpolation inequality, one can show that a curvature
dimension condition $CD_{L}(K,N)$ holds on any $n$-dimensional ($n<N$)
Finsler manifold $M$ with (weighted) Ricci curvature bounded from
below by $K$. The condition $CD_{L}(K,N)$ is nothing but a convexity
property of functionals in $\mathcal{DC}_{N}$ along geodesics in
$\mathcal{P}_{L}(M)$. Most geometric properties (Brunn-Minkowski,
Bishop-Gromov, local Poincar\'e and doubling) also hold under such
a condition. However, the lack of an ``easy-to-understand'' dual
theory makes it difficult to prove statements involving (weak) upper
gradients.
\begin{cor}
Any $n$-dimensional Finsler manifold with $N$-Ricci curvature bounded
from below by $K$ and $N>n$ satisfies the very strong $CD_{p}(K,N)$
condition for all strictly convex, increasing functional $L:[0,\infty)\to[0,\infty)$
which is smooth away from zero.
\end{cor}
\bibliographystyle{amsalpha}
\bibliography{bib}

\end{document}